\documentclass[reqno,11pt]{amsart}
\usepackage{fullpage}
\usepackage{amsfonts}
\usepackage{mathrsfs}
\usepackage{mathtools}
\usepackage{amsmath}
\usepackage{amssymb}
\usepackage{bbm}
\usepackage[all,cmtip]{xy}

\usepackage[pagebackref,hyperindex,linktocpage=true]{hyperref}

\newtheorem{thm}{Theorem}[section]

\newtheorem{prop}[thm]{Proposition}
\newtheorem{lem}[thm]{Lemma}

\newtheorem{lem-def}[thm]{Lemma-Definition}
\newtheorem{cor}[thm]{Corollary}

\theoremstyle{definition}

\newtheorem{ex}[thm]{Example}

\newtheorem{rem}[thm]{Remark}
\newtheorem{defn}[thm]{Definition}

\numberwithin{equation}{section}

\newcommand{\epic}{\twoheadrightarrow}
\newcommand{\into}{\hookrightarrow}

\newcommand{\bbA}{\mathbb{A}}
\newcommand{\bbB}{\mathbb{B}}
\newcommand{\bbC}{\mathbb{C}}

\newcommand{\bbF}{\mathbb{F}}
\newcommand{\bbG}{\mathbb{G}}
\newcommand{\bbH}{\mathbb{H}}
\newcommand{\bbI}{\mathbb{I}}
\newcommand{\bbJ}{\mathbb{J}}

\newcommand{\bbP}{\mathbb{P}}
\newcommand{\bbQ}{\mathbb{Q}}

\newcommand{\bbT}{\mathbb{T}}

\newcommand{\bbV}{\mathbb{V}}
\newcommand{\bbW}{\mathbb{W}}

\newcommand{\bbZ}{\mathbb{Z}}

\newcommand{\bbc}{\mathbbm{c}}

\newcommand{\bbg}{\mathbbm{g}}

\newcommand{\bbr}{\mathbbm{r}}

\newcommand{\bbt}{\mathbbm{t}}

\newcommand{\bfI}{\mathbf{I}}

\newcommand{\bfP}{\mathbf{P}}
\newcommand{\bfQ}{\mathbf{Q}}

\newcommand{\cB}{\mathcal{B}}

\newcommand{\cE}{\mathcal{E}}
\newcommand{\cF}{\mathcal{F}}
\newcommand{\cG}{\mathcal{G}}

\newcommand{\cL}{\mathcal{L}}

\newcommand{\cN}{\mathcal{N}}
\newcommand{\cO}{\mathcal{O}}

\newcommand{\cR}{\mathcal{R}}

\newcommand{\cT}{\mathcal{T}}
\newcommand{\cU}{\mathcal{U}}
\newcommand{\cV}{\mathcal{V}}

\newcommand{\fa}{\mathfrak{a}}
\newcommand{\fb}{\mathfrak{b}}
\newcommand{\fc}{\mathfrak{c}}

\newcommand{\fg}{\mathfrak{g}}
\newcommand{\fh}{\mathfrak{h}}

\newcommand{\fn}{\mathfrak{n}}

\newcommand{\fp}{\mathfrak{p}}

\newcommand{\ft}{\mathfrak{t}}

\newcommand{\sA}{\mathsf{A}}
\newcommand{\sB}{\mathsf{B}}

\newcommand{\sG}{\mathsf{G}}

\newcommand{\sT}{\mathsf{T}}
\newcommand{\sU}{\mathsf{U}}

\newcommand{\sk}{\mathsf{k}}

\newcommand{\coker}{\mathrm{coker}}
\newcommand{\ga}{\gamma}

\newcommand{\la}{\lambda}
\newcommand{\La}{\Lambda}

\newcommand{\spec}{\mathrm{Spec}}

\newcommand{\vol}{\mathrm{vol}}

\newcommand\bm[1]{\begin{bmatrix}#1\end{bmatrix}}

\title{Witt vector affine Springer fibers}
\author{Jingren Chi}
\address{Morningside Center of Mathematics and State Key Laboratory of Mathematical Sciences, Academy of Mathematics and Systems Science, Chinese Academy of Sciences, Beijing 100190, China}
\email{jrenchi@amss.ac.cn}
\date{}
    
\begin{document}
\begin{abstract}
    We establish dimension formulae for the Witt vector affine Springer fibers associated to a reductive group over a mixed characteristic local field, under the assumption that the group is essentially tamely ramified and the residue characteristic is not bad. Besides the discriminant valuations that show up in classical works on the usual affine Springer fibers, our formulae also involve the Artin conductors and the Kottwitz invariants of the relevant conjugacy classes.
\end{abstract}
\maketitle

\section{Introduction}\label{sec:intro}
\subsection{General motivation}
The affine Springer fibers are analogues of Springer fibers for loop groups, first introduced and studied over the field of complex numbers $\bbC$ by Kazhdan and Lusztig in \cite{KL}. Later their counterparts over finite fields have found applications in harmonic analysis on groups $G(F)$ where $F$ is an \emph{equal-characteristic} local field, i.e. a finite extension of the field of Laurent series $\bbF_p((t))$, and $G$ is a reductive group over $F$, see \cite{GKM}, \cite{Ngo10} and more recently \cite{BV21}. The main reason behind this is that they provide a geometrization of the orbital integrals on the Lie algebra $\fg(F)$, which are closely related to orbital integrals on the group $G(F)$. When $F$ is a mixed characteristic local field, i.e. a finite extension of the field of $p$-adic numbers $\bbQ_p$, the orbital integrals on $G(F)$ are the key ingredients in one side of the Arthur-Selberg trace formula over number fields, making them fundamental objects of investigation in the theory of automorphic representations and the Langlands program. Therefore, it is valuable to have geometrizations of orbital integrals over mixed characteristic local fields. Given that the classical affine Springer fibers are subschemes of the affine Grassmannians or (partial) affine flag varieties, the geometrizations of orbital integrals in the mixed-characteristic setting would involve subschemes defined analogously in the Witt vector affine Grassmannians, which have been explored in the works of Xinwen Zhu \cite{Zhu-mixed} and Bhatt-Scholze \cite{BS-Witt}. The study of these Witt vector affine Springer fibers is anticipated to employ techniques from algebraic geometry and sheaf theory, akin to the equal-characteristic scenario. In this paper, we take the initial step in this direction by establishing their fundamental geometric properties: finiteness properties, criteria for non-emptiness, and dimension formulae.\par
To be more precise, we first assume that $F$ is a finite extension of $\bbQ_p$ with valuation ring $\cO$ and residue field $k=\bbF_q$, where $q$ is a power of $p$. Let $G$ be a split reductive group over $F$ and let $\fg$ be the Lie algebra of $G$. (In the main body of the paper we will work under less restrictive assumptions.) The most basic orbital integrals one encounters are those of the characteristic function $1_{G(\cO)}$ of the hyperspecial maximal compact subgroup $G(\cO)$:
\[O_\ga(1_{G(\cO)}):=\int_{G_\ga(F)\backslash G(F)}1_{G(\cO)}(g^{-1}\ga g)d\dot{g}=\sum_{g\in G_\ga(F)\backslash X_\ga^G(k)}\frac{1}{\vol(gG(\cO)g^{-1}\cap G_\ga(F),dg_\ga)}\]
where 
\begin{itemize}
    \item $\ga$ is an arbitrary element of $G(F)$ and $G_\ga(F)$ is its centralizer in $G(F)$;
    \item we choose Haar measures $dg$ (resp. $dg_\ga$) on $G(F)$ (resp. $G_\ga(F)$) such that $\mathrm{vol}(G(\cO),dg)=1$ and let $d\dot{g}=\frac{dg}{dg_\ga}$ be the quotient measure on $G_\ga(F)\backslash G(F)$; 
    \item $X_\ga^G(k):=\{g\in G(F)/G(\cO)\mid g^{-1}\ga g\in G(\cO)\}$.
\end{itemize}
Here $X_\ga^G(k)$ is the set of $k$-points of the \emph{Witt vector affine Springer fiber} associated to $\ga$, which we denote by $X_\ga^G$. It is the fixed point loci of the natural action of $\ga$ (by left multiplication) on the Witt vector affine Grassmannian $\mathrm{Gr}_G$ that geometrize the quotient $G(F)/G(\cO)$. The sum in the formula above might be infinite 
for a general element $\ga\in G(F)$. Nevertheless it always converges by a result of Deligne and Rao (see \cite{Rao}), and when $\ga$ is regular semisimple the sum is finite.

\subsection{Affine Springer fibers for the Lie algebra}
A common practice in the study of the orbital integrals $O_\ga(1_{G(\cO)})$ is to relate them to analogous integrals defined on the Lie algebras. That is to say, for each element $\ga\in\fg(F)$, one considers the orbital integral
\[O_\ga(1_{\fg(\cO)}):=\int_{G_\ga(F)\backslash G(F)}1_{\fg(\cO)}(\mathrm{ad}(g)^{-1}\ga)d\dot{g}\]
whose geometrization is the Witt vector affine Springer fiber $X_\ga^\fg$ for the element $\ga$ in the Lie algebra. As in the group case, its set of $k$-points is described by
\[X_\ga^\fg(k)=\{g\in G(F)/G(\cO)\mid\mathrm{ad}(g)^{-1}\ga\in \fg(\cO)\}.\]
If one replaces the field $F$ (resp. its valuation ring $\cO$) in this description by the Laurent series field $k((t))$ (resp. the power series ring $k[[t]]$), then we get the set of $k$-points of the usual affine Springer fibers, whose basic geometric properties and dimension formulae are established in \cite{KL}, \cite{Be96}, see also \cite{Ngo10} and the generalizations to tamely ramified groups in \cite[Appendix]{OY}. Our first main result is concerned with the analogous statements in the mixed-characteristic setting. Since we are interested in the geometric properties of $X_\ga^\fg$, we may and do replace $F$ by the completion of its maximal unramified extension and assume that the residue field $k$ is algebraically closed. 
\begin{thm}[special case of Theorem \ref{thm:main-Lie-algebra-case}]
    Let $G$ be a split reductive group over $F$ with  Lie algebra $\fg$. Let $\ga\in\fg(F)$ be a bounded regular semisimple element (see Definition \ref{def:bounded-Lie-alg}). Then the Witt vector affine Springer fiber $X_\ga^\fg$ is represented by a finite dimensional perfect scheme locally perfectly of finite type (see Definition \ref{def:perfect-finite-type}). Suppose moreover that the residue characteristic $p$ is not bad for $G$, then we have 
    \[\dim X_\ga^\fg=\frac{d_\fg(\ga)-\mathrm{Art}_\ga}{2}\]
    where $d_\fg(\ga)$ is the discriminant valuation of $\ga$ and $\mathrm{Art}_\ga=\mathrm{Art}(G_\ga)$ is the Artin conductor of the centralizer of $\ga$. 
\end{thm}
We refer to Definition \ref{def:disc-Lie-alg} and Definition \ref{def:conductor-group} for the numerical invariants occurring in the formulae above. In fact we will only assume that $G$ is essentially tamely ramified (a condition slightly more general than tamely ramified, see Definition \ref{def:ess-tame-ram}) and the general formulae will also involve the Artin conductor of $G$ (which vanishes when $G$ is split). Besides,
we also study the Witt vector affine Springer fibers in the partial affine flag varieties. In other words, the hyperspecial maximal bounded subgroup $G(\cO)$ could be replaced by a general parahoric subgroup $\bfP\subset G(F)$, with the caveat that when $\bfP$ is not contained in any special parahoric subgroup, we only prove an upper bound on the dimension in the current work. We expect that the general dimension formulae in Theorem \ref{thm:main-Lie-algebra-case} should hold for any reductive group $G$ over $F$ and any parahoric subgroup $\bfP\subset G(F)$, with no assumption on the residue characteristic $p$.\par 
The primary innovation in our first result is the introduction of the Artin conductor of $G_\ga$. This is not unexpected in view of \cite{GG-Artin}. When the centralizer $G_\ga$ splits over a tamely ramified extension, the Artin conductor of $G_\ga$ is simply its \emph{defect}, i.e. the difference between its absolute rank and its $F$-rank. In these cases our formulae become the exact analogues of the dimension formulae for the usual affine Springer fibers in \cite{KL}, \cite{Be96} (see also \cite{Ngo10} and generalizations in \cite[Appendix]{OY}). However in these references it is assumed that the characteristic of the residue field is either $0$ or does not divide the order of the Weyl group, ensuring that the centralizers are always tamely ramified. By contrast, we only require that the characteristic $p$ is not bad for $G$ and hence our results are stronger than the direct analogues of the classical results in the equal characteristic case. For instance, when $G=\mathrm{GL}_n$ we impose no restriction on $p$, in contrast to the assumption $p>n$ in the classical works. To address the possible wild ramification phenomena in the case of small residue characteristics, a crucial ingredient is a result of Chai-Yu \cite{ChaiYu} on the base change conductors for tori, where the Artin conductors emerge naturally as the generalizations of the defects of tamely ramified tori. More precisely, this result will be used in Theorem \ref{thm:dim-reg} to compute the dimensions of the regular locus of the Witt vector affine Springer fibers.

\subsection{Affine Springer fibers for the group}
Our second main result is about the Witt vector affine Springer fibers $X_\ga^G$ defined for the group $G$ itself. We will study more generally the fixed point locus\footnote{More precisely, this is the fixed point locus of $\ga$ only if $G$ is semisimple or if the Kottwitz invariant of $\ga$ vanishes.} $X_{\bfP,\ga}^G$ of a regular semisimple element $\ga\in G(F)$ on the Witt vector affine partial flag variety $G(F)/\bfP$, where $\bfP\subset G(F)$ is any parahoric subgroup. Its set of $k$-points is
\[X_{\bfP,\ga}^G(k)=\{g\in G(F)/\bfP\mid g^{-1}\ga g\in\widetilde{\bfP}\}\]
where $\widetilde{\bfP}$ is the normalizer of $\bfP$ in $G(F)$. We will show that $X_{\bfP,\ga}^G$ is nonempty only if $\ga$ is bounded mod center in $G(F)$ (see Theorem \ref{thm:nonempty-group}). In the special case where $\ga$ lies in a parahoric subgroup of $G(F)$ (or equivalently, if the Kottwitz invariant of $\ga$ vanishes), we can replace the condition in the description of $X_{\bfP,\ga}^G(k)$ by the more familiar condition $g^{-1}\ga g\in\bfP$. The generality considered here will be necessary in the study of the geometrizations of orbital integrals of general functions in the parahoric Hecke algebras, see \cite{He2023affine}. The main novelty in this generality is that, compared to the Lie algebra case, the dimension formula for $X_{\bfP,\ga}^G$ will also involve the Kottwitz invariant $\kappa_G(\ga)$. For instance, when $G=\mathrm{GL}_n$, then $\kappa_G(\ga)$ is the valuation of the determinant of $\ga$.\par 
Now we state our second main result. After replacing $F$ by the completion of its maximal unramified extension we may and do assume that the residue field $k$ is algebraically closed. Then the group $G$ is automatically quasi-split. 
\begin{thm}[special case of Theorem \ref{thm:main-group-case}]
    Let $G$ be a tamely ramified reductive group over $F$. For any bounded (see \S\ref{sec:top-Jordan-group}) regular semisimple element $\ga\in G(F)^{\mathrm{rs}}$ and any parahoric subgroup $\bfP\subset G(F)$, the Witt vector affine Springer fiber $X_{\bfP,\ga}^G$ is a finite dimensional perfect scheme locally perfectly of finite type (see Definition \ref{def:perfect-finite-type}). Suppose moreover that the residue characteristic $p$ is not bad for $G$ and that $\bfP$ is contained in a special parahoric subgroup of $G(F)$, then we have
    \[\dim X_{\bfP,\ga}^G=\frac{1}{2}(d_G(\ga)+\mathrm{def}(\kappa_G(\ga))+\mathrm{Art}(G)-\mathrm{Art}_\ga)\]
    where 
    \begin{itemize}
        \item $d_G(\ga)$ is the discriminant valuation of $\ga$ (see Definition \ref{def:disc-group});
        \item $\mathrm{Art}_\ga=\mathrm{Art}(G_\ga)$ is the Artin conductor of the centralizer of $\ga$, and $\mathrm{Art}(G)$ is the Artin conductor of the group $G$ (see Definition \ref{def:conductor-group});
        \item $\mathrm{def}(\kappa_G(\ga))$ is the defect of the Kottwitz invariant $\kappa_G(\ga)$ (see Definition \ref{def:defect-Kott-invariant}).
    \end{itemize}
    Furthermore, if $\bfP=\bfI$ is an Iwahori subgroup, then $X_{\bfI,\ga}^{G}$ is equi-dimensional.
\end{thm}
We expect the dimension formulae above to be valid for any reductive group $G$ over $F$ and any parahoric subgroup $\bfP\subset G(F)$, with no assumption on the residue characteristic $p$. On the other hand, we do not know if the equi-dimensionality property is true for general parahoric subgroups.\par 
Let us briefly explain our strategy. For elements $\ga$ that are strongly topologically unipotent in the sense of Definition \ref{def:top-unip}, we can identify $X_{\bfP,\ga}^G$ with a Witt vector affine Springer fiber for the Lie algebra using a quasi-logarithm map (cf. \cite{KV}) and deduce the dimension formulae from the Lie algebra case. In general, we use the topological Jordan decomposition (cf. \cite{Spice08}, \cite{Ka-lifting}) and a geometric version of Harish-Chandra descent argument in \cite{BV21} to reduce to the case of topologically unipotent element. However, there is the remaining case of topologically unipotent but not strongly topologically unipotent elements. Under our assumption that the group $G$ is (essentially) tamely ramified and the residue characteristic $p$ is not bad, this case occurs only when the adjoint group has a $\mathrm{PGL}_n$ factor and $p$ divides $n$. This remaining case will be treated separately in \S\ref{sec:GLn} by a more involved argument. 
\subsection{Related works in the literature}
In the equal characteristic setting, there has been a vast amount of literature on various aspects of affine Springer fibers and their generalizations. In the Lie algebra case, our general strategy follows the works \cite{KL}, \cite{Be96}, \cite{OY}, with reformulations in \cite{Ngo10} using Chevalley quotients and regular centralizers. In the group case, we follow the strategy of \cite{BV21} to reduce to the Lie algebra case. In particular our methods are purely local.\par 
We also mention the generalizations of group version of affine Springer fibers from \cite{Bou}, \cite{BC}, \cite{Chi} (in which they are called ``Kottwitz-Viehmann" varieties), \cite{He2023affine} (in which they are called ``affine Lusztig varieties"). These are geometric incarnations of orbital integrals for general functions in the parahoric Hecke algebras, whereas the group version of the (Witt vector) affine Springer fibers we study in this paper mostly corresponds to the unit elements. Note however that in the works \cite{BC}, \cite{Chi}, global methods involving Hitchin-type moduli spaces play an important role in the proof of the dimension formulae. These methods (currently) do not generalize to the mixed-characteristic setting. In \cite{He2023affine}, the dimension formulae for general affine Lusztig varieties are reduced to the case that we study in this paper. The reduction methods of \emph{loc. cit.} are purely local and work uniformly in both equal and mixed characteristic setting. Thus after combining \cite{He2023affine} and the results in the current paper we could obtain the mixed-characteristic counterparts of the results in \cite{Bou}, \cite{BC} and \cite{Chi}.
\subsection{Organization of the article}
The first three sections \S\ref{sec:group}-\S\ref{sec:conj-Lie} are preparations for the proof of the main results. \S\ref{sec:group} contains a review of some standard facts on reductive groups over local fields. In \S\ref{sec:conj-group} we discuss numerical invariants of regular semisimple conjugacy classes in reductive groups: the defect of Kottwitz invariants, the topological Jordan decompositions, a result of Chai-Yu and its implication on the Artin conductors (Corollary \ref{cor:length-artin-conductor}) that is a crucial ingredient of our work. In \S\ref{sec:conj-Lie} we discuss some results on conjugacy classes in the Lie algebras: Chevalley quotients and universal regular centralizers, topological Jordan decompositions for Lie algebras and quasi-logarithm maps.\par 
In \S\ref{sec:transporter} we study a class of generalized affine Springer fibers, establish their basic finiteness properties (Theorem \ref{thm:finiteness}) and discuss their relations to orbital integrals in \S\ref{sec:orbital-integrals}. In \S\ref{sec:main-thm} we state our two main results on the dimension formulae of affine Springer fibers: the group case (Theorem \ref{thm:main-group-case}) and the Lie algebra case (Theorem \ref{thm:main-Lie-algebra-case}).\par 
In the remaining sections \S\ref{sec:dim-Lie-alg} and \S\ref{sec:dim-group} we prove our main results. We first establish the dimension formulae for the Lie algebras in \S\ref{sec:dim-Lie-alg} and then deduce from it the dimension formulae for the groups in \S\ref{sec:dim-group}.

\subsection{Notations and conventions}
For a group scheme $G$ over a base scheme $S$, we let $Z_G$ denote its center, $G^\circ$ denote its fiberwise identity component, and $\pi_0(G)=G/G^\circ$ denote its group scheme of connected components. The derived (resp. adjoint) group of $G$ will be denoted by $G^{\mathrm{der}}$ (resp. $G_{\mathrm{ad}}$). We use $\mathrm{Lie}(G)$ to denote the Lie algebra of $G$, understood either as an $\cO_S$-module or as an $S$-scheme (it should be clear from the context which point of view we take). A reductive group scheme is always assumed to be fiberwise connected. If $G$ is reductive, then $G^{\mathrm{der}}$ is semisimple and its simply connected cover will be denoted by $G^{\mathrm{sc}}$. Then there are canonical homomorphisms 
\[G^{\mathrm{sc}}\to G^{\mathrm{der}}\to G\to G_{\mathrm{ad}}=G/Z_G.\]
For an element $\ga\in G(S')$ for some $S$-scheme $S'$, we let $G_\ga$ be the  centralizer of $\ga$, which is a group scheme over $S'$.\par
For a scheme $X$ (resp. $Y$) over $S$ on which $G$ acts on the right (resp. left), we let $X\times^GY$ be the stack quotient of $X\times Y$ by the anti-diagonal right $G$ action: $g\in G$ acts by $(x,y)\mapsto(xg,g^{-1}y)$.\par 
We will be working with perfect (ind)-schemes/algebraic spaces and a standard reference is \cite[Appendix A]{Zhu-mixed}. For the reader's convenience, we recall the following notion that will be frequently used throughout the article: 
\begin{defn}\label{def:perfect-finite-type}
    Let $k$ be a perfect field of characteristic $p>0$. A perfect scheme $X$ over $k$ is said to be \emph{locally perfectly of finite type} if it has an open cover $\{U_i\}$ such that each $U_i$ is the perfection of a finite type affine scheme over $k$. A perfect scheme over $k$ is said to be \emph{perfectly of finite type} if it is locally perfectly of finite type and quasi-compact.
\end{defn}
For a reductive group $G$ over a field $F$, with a maximal torus $T\subset G$, let $T^{\mathrm{sc}}$ be the pre-image of $T$ in $G^{\mathrm{sc}}$, the \emph{algebraic fundamental group of $G$} (in the sense of Borovoi) is defined by $\pi_1(G):=X_*(T)/X_*(T^{\mathrm{sc}})$. See \cite[\S11.3]{KP23}. It is an abelian group with an action of the absolute Galois group of $F$, and is independent of the choice of $T$. \par
In the main body of the paper we will be working over a mixed-characteristic discrete valued field $F$ with valuation ring $\cO$ and residue field $k$. We use normal math italic fonts (e.g. $G$) for groups defined over $F$, calligraphic fonts (e.g. $\cG$) for their integral models over $\cO$, and sans-serif fonts (e.g. $\sG$) for their special fiber over the residue field $k$. When our groups are twisted forms of split groups, we will often use blackboard bold font (e.g. $\bbG$) for the corresponding split form. We will use boldface fonts (e.g. $\bfP,\bfI$) for open bounded (mod center) subgroups of $G(F)$, these mainly include parahoric subgroups, their pro-unipotent radicals and their normalizers. The Lie algebra of a group (e.g. $G$) will usually be denoted by the corresponding Gothic fonts (e.g. $\fg$).

\subsection*{Acknowledgments}
This work was partially motivated by a question of Vladimir Drinfeld in the geometric Langlands seminar at the University of Chicago. I would like to thank Professor Ng\^o Bao Chau for suggesting and encouraging me to pursue along this direction. I thank Xuhua He and Sian Nie for helpful conversations. I thank Alexis Bouthier, Kestutis {\v{C}}esnavi{\v{c}}ius, Xuhua He, Kazuhiro Ito, Cheng-Chiang Tsai, Yakov Varshavsky for feedback on the draft of the paper. I am grateful to the anonymous reviewer for valuable comments and suggestions. This work is supported by National Key R\&D Program of China No.2023YFA1009701, National Natural Science Foundation of China (Grant No. 12288201, No.12231001), CAS Project for Young Scientists in Basic Research (Grant No.YSBR-033).

\section{Reductive groups over local fields}\label{sec:group}
\subsection{Local fields and extensions}\label{sec:local-fields}
Fix a prime number $p>0$. Let $k$ be a perfect field of characteristic $p$ and fix an algebraic closure $\bar{k}$ of $k$. Let $F$ be a complete discrete valued field of characteristic $0$ with valuation ring $\cO$ and residue field $k$. Fix a uniformizer $\varpi\in\cO$. Let $\mathrm{val}:F\to\bbZ\cup\{\infty\}$ be the normalized valuation on $F$ with $\mathrm{val}(\varpi)=1$. Fix a (separable) algebraic closure $\overline{F}$ of $F$. For each positive integer $m$, let $\mu_m$ be the group of $m$-th roots of unity in $\overline{F}$. The valuation $\mathrm{val}$ on $F$ extends uniquely to $\overline{F}$ and the associated $p$-adic absolute value on $\overline{F}$ is defined by $|x|:=p^{-\mathrm{val}(x)}$ for any $x\in\overline{F}$. Let $\overline{\cO}\subset\overline{F}$ be the valuation ring, which is also the integral closure of $\cO$ in $\overline{\cO}$. Let $F^{\mathrm{ur}}$ be the maximal unramified extension of $F$ in $\overline{F}$. Let $\Breve{F}$ be the completion of $F^{\mathrm{ur}}$ and let $\Breve{\cO}\subset\Breve{F}$ be its valuation ring. Let $\Gamma:=\mathrm{Gal}(\overline{F}/F)$ be the absolute Galois group of $F$ and let $\Gamma_1=\mathrm{Gal}(\overline{F}/F^{\mathrm{ur}})$ be the inertia subgroup, which is canonically isomorphic to the absolute Galois group of $\Breve{F}$.\par 
For any perfect $k$-algebra $R$, we let $W(R)$ be its ring of $p$-typical Witt vectors. Let $\cO_0=W(k)$ and let $F_0=W(k)[\frac{1}{p}]$ be the fraction field of $\cO_0$. Then $F$ is a finite totally ramified extension of $F_0$. Let $W_\cO(R):=W(R)\otimes_{W(k)}\cO$ be the ring of $\cO$-ramified Witt vectors. In particular, we have $\Breve{\cO}=W_\cO(\bar{k})$. There is a unique group homomorphism
\[[\cdot]:R^\times\to W_\cO(R)^\times\] 
that lifts the identity map modulo $\varpi$. When $R=\Bar{k}$, the image $[\Bar{k}^\times]\subset\Breve{\cO}$ is called the set of \emph{Teichm\"{u}ller representatives} of $\Bar{k}^\times$.\par
In the remaining of the paper, except \S\ref{sec:orbital-integrals},  we will assume that $k=\Bar{k}$ so that $F=\Breve{F}$, $\cO=\Breve{\cO}$.

\subsection{Quasi-split groups and parahoric subgroups}\label{sec:parahoric}
Let $G$ be a connected reductive group over $F$. Since $F$ is complete and the residue field $k$ is algebraically closed, $G$ is automatically quasi-split by Steinberg's theorem. We will fix a pinning of $G$ as follows.
Let $F'/F$ be the \emph{minimal} Galois extension such that $G$ becomes split over $F'$. Let $\bbG$ be the split connected reductive group scheme over $\bbZ$ with Lie algebra $\bbg:=\mathrm{Lie}(\bbG)$, together with an isomorphism $\bbG_{F'}\cong G_{F'}$. Fix a pinning $(\bbT,\bbB,(x_\alpha)_{\alpha\in\Delta})$ of $\bbG$ defined over $\bbZ$, consisting of a split maximal torus $\bbT$ and a Borel subgroup $\bbB$ containing $\bbT$ and for each simple root $\alpha\in\Delta$, a collection of $\bbZ$-basis vectors $x_\alpha\in\bbg_\alpha(\bbZ)$ in the root space $\bbg_\alpha\subset\bbg$. Let $\bbt$ be the Lie algebra of $\bbT$, viewed either as a $\bbZ$-module or a scheme. Let $\sG:=\bbG_k$ (resp. $\sT:=\bbT_k$, $\sB:=\bbB_k$) be the base change of $\bbG$ (resp. $\bbT,\bbB$) to $k$. In the following we will often simplify notation and still use $\bbG$, $\bbT$ etc. to denote their base change to other rings. We hope this will not cause any confusion since in each case it will be clear from the context what are the base rings.\par
The group $\mathrm{Out}(\bbG)$ of outer automorphisms of $\bbG$ can be identified with the subgroup of the automorphism group scheme $\mathrm{Aut}(\bbG)$ that stabilizes this pinning. The group $G$ is the twisted form of $\bbG$ associated to a homomorphism 
\[\rho_G:\mathrm{Gal}(F'/F)\to\mathrm{Out}(\bbG).\]
We also get a Borel subgroup $B$ and maximal torus $T$ of $G$, both defined over $F$. Let $\ft$ (resp. $\fb$) be the Lie algebra of $T$ (resp. $B$).\par

Let $\cB(G)$ be the (reduced) Bruhat-Tits building and let $\widetilde{\cB}(G)$ be the extended Bruhat-Tits building of $G(F)$. By definition $\cB(G)$ is also the Bruhat-Tits building of $G_{\mathrm{ad}}(F)$ or $G^{\mathrm{sc}}(F)$. By Bruhat-Tits theory (cf. \cite{KP23} for a modern reference), to each point $x\in\widetilde{\cB}(G)$ in the \emph{extended} building is associated a sequence of open bounded (mod center) subgroups of $G(F)$:
\[\bfP_{x+}\subset\bfP_x\subset\bfP_x^\dagger\subset\widetilde{\bfP}_x\]
where $\bfP_x$ is the parahoric subgroup, $\bfP_{x+}$ is the pro-unipotent radical of $\bfP_x$, $\bfP_x^\dagger$ is the stabilizer of $x$ in $G(F)$ and $\widetilde{\bfP}_x$ is the normalizer of $\bfP_x$ in $G(F)$ (which coincides with the stabilizer in $G(F)$ of the image of $x$ in the reduced building $\cB(G)$, see \cite[\S4.3]{KP23}). We will omit the subscript ``$x$" when it does not play any role and simply write $\bfP_+,\bfP,\bfP^\dagger,\widetilde{\bfP}$ instead. For example, if $G=\mathrm{GL}_n$ and $\bfP=G(\cO)$ is the hyperspecial parahoric subgroup, then we have $G(\cO)_+=\ker(G(\cO)\to G(k))$, $\bfP^\dagger=\bfP$ and $\widetilde{\bfP}=\bfP Z_G(F)$. We note that the group $\bfP^\dagger$ is bounded while $\widetilde{\bfP}$ is only bounded mod center but not bounded in general.\par
For each parahoric subgroup $\bfP\subset G(F)$ we let $\cG_\bfP$ (resp. $\cG_\bfP^\dagger$) be the smooth group scheme over $\cO$ with generic fiber $G$ and group of $\cO$-points $\cG_\bfP(\cO)=\bfP$ (resp. $\cG^\dagger_\bfP(\cO)=\bfP^\dagger$). Let $\sG_\bfP$ (resp. $\sG^\dagger_\bfP$) be the reductive quotient of the special fiber of $\cG_\bfP$ (resp. $\cG^\dagger_\bfP$). Then $\cG_\bfP$ is the fiberwise identity component of $\cG^\dagger_\bfP$ and $\sG_\bfP$ is the identity component of $\sG^\dagger_\bfP$. 

\subsection{Restriction of scalars and fixed point loci}
We will explicitly construct some parahoric group schemes by restriction of scalars and taking fixed points under finite group actions. Here we review the related facts, following \cite{BLR-Neron} and \cite{Edi}.\par 
Let $S'\to S$ be a morphism of schemes. For any scheme $X'$ over $S'$, recall that the \emph{restriction of scalars of $X'$} along $S'\to S$, denoted $\mathrm{Res}_{S'/S}X'$ is the functor that sends any $S$-scheme $T$ to the set
\[(\mathrm{Res}_{S'/S}X')(T):=\mathrm{Hom}_{S'}(T\times_SS',X').\]
When $S=\spec A$ and $S'=\spec A'$ are affine schemes, we also write $\mathrm{Res}_{A'/A}X'$ instead of $\mathrm{Res}_{S'/S}X'$.
\begin{lem}\label{lem:res}
    Let $A'$ be a finite locally free $A$-algebra and let $X'$ be an affine scheme over $\spec A'$. Then we have
    \begin{enumerate}
        \item The restriction of scalars $\mathrm{Res}_{A'/A}X'$ is represented by an affine scheme. 
        \item If $X'$ is smooth over $\spec A'$, then $\mathrm{Res}_{A'/A}X'$ is smooth over $\spec A$. 
    \end{enumerate}
\end{lem}
\begin{proof}
    (1) From the definition we see that the formation of $\mathrm{Res}_{A'/A}X'$ is Zariski local on the base $\spec(A)$. Therefore to show that $\mathrm{Res}_{A'/A}X'$ is affine we may assume after localization on $A$ that $A'$ is a finite free $A$-algebra by \cite[\href{https://stacks.math.columbia.edu/tag/01S8}{Tag 01S8}]{stacks-project}. Then we proceed as in the proof of \cite[\S7.6, Theorem 4]{BLR-Neron} to see that $\mathrm{Res}_{A'/A}X'$ is represented by an affine scheme over $A$. \par 
    (2) This follows from \cite[Lemma 2.2]{Edi}.
\end{proof}
The following result is a special case of \cite[Proposition 3.4]{Edi}. For the reader's convenience, we provide a more straightforward proof in the current setting.
\begin{lem}\label{lem:fixed-point-locus-smooth}
    Let $S$ be a scheme and $\Sigma$ be a finite \'etale group scheme over $S$. Let $X$ be an affine $S$-scheme with an action of $\Sigma$. Assume that the order of $\Sigma$ is invertible over $S$ and that $X$ is smooth over $S$, then the fixed point subscheme $X^\Sigma$ is smooth over $S$.
\end{lem}
\begin{proof}
    Replacing $S$ by an \'etale cover we may assume that $\Sigma$ is a finite constant group scheme of order $|\Sigma|\in\cO_S^\times$. We write $X=\spec(A)$ and let $J\subset A$ be the ideal generated by $\{a-\sigma(a) \mid a\in A,\sigma\in\Sigma\}$. Then we have $X^\Sigma=\spec (A/J)$. 
    Let $R$ be an $\cO_S$-algebra and $I\subset R$ an ideal such that $I^2=0$. Suppose we are given a homomorphism of $\cO_S$-algebras $\bar{f}:A/J\to R/I$ corresponding to a morphism $\spec(R/I)\to X^\Sigma$. By the infinitesimal lifting criterion, to prove smoothness of $X^\Sigma=\spec(A/J)$, it suffices to find a homomorphism $A/J\to R$ lifting $\Bar{f}$.\par 
    Since $X=\spec A$ is smooth over $S$, there exists an $\cO_S$-algebra  homomorphism $f:A\to R$ making the following diagram commutative:
    \[\xymatrix{
    A\ar[r]^f\ar[d] & R\ar[d]\\
    A/J\ar[r]^{\Bar{f}} & R/I
    }\]
    In particular we have $f(J)\subset I$, so for any $\sigma\in\Sigma$
    there exists a map $i_\sigma:A\to I$ such that $f(\sigma(a))=f(a)+i_\sigma(a)$ for all $a\in A$. For each $\sigma\in\Sigma$, from the identity $f(\sigma(ab))=f(\sigma(a))f(\sigma(b))$ for any $a,b\in A$ we get that
    \[i_\sigma(ab)=f(b)i_\sigma(a)+f(a)i_\sigma(b)\]
    Define an $\cO_S$-linear map $\Tilde{f}:A\to R$ by 
    \[\Tilde{f}(a)=\frac{1}{|\Sigma|}\sum_{\sigma\in\Sigma}f(\sigma(a))\]
    By the previous identity and the fact that $I^2=0$, we get that for any $a,b\in A$, 
    \begin{equation*}
        \begin{split}
            \Tilde{f}(ab)&=f(ab)+\frac{1}{|\Sigma|}\sum_{\sigma\in\Sigma}i_\sigma(ab)\\
            &=f(a)f(b)+\frac{1}{|\Sigma|}f(b)\sum_{\sigma\in\Sigma}i_\sigma(a)+\frac{1}{|\Sigma|}f(a)\sum_{\sigma\in\Sigma}i_\sigma(b)\\
            &=(f(a)+\frac{1}{|\Sigma|}\sum_{\sigma\in\Sigma}i_\sigma(a))(f(b)+\frac{1}{|\Sigma|}\sum_{\sigma\in\Sigma}i_\sigma(b))\\
            &=\Tilde{f}(a)\Tilde{f}(b).
        \end{split}
    \end{equation*}
    Thus $\tilde{f}$ is an algebra homomorphism and by definition it factors through an $\cO_S$-algebra homomorphism $A/J\to R$ that lifts $\Bar{f}$.
\end{proof}

\subsection{Special integral model for tamely ramified groups}\label{sec:group-model}
Let $F$ be as in \S\ref{sec:local-fields} and assume that its residue field $k$ is algebraically closed. Let $G$ be a connected reductive group over $F$ as in \S\ref{sec:parahoric}. In this subsection we assume moreover that the splitting extension $F'/F$ of $G$ is \emph{tamely} ramified of degree $e\ge1$, $p\nmid e$. We will construct $G$ together with a special integral model $\cG_0$ and an Iwahori subgroup $\bfI\subset G(F)$ explicitly.\par 
Let $\cO'$ be the valuation ring of $F'$ and choose a uniformizer $\varpi'\in \cO'$ such that $\varpi=(\varpi')^e$ is a uniformizer of $\cO$. Then we have a canonical isomorphism $\mathrm{Gal}(F'/F)\cong\mu_e$ where $\mu_e$ is the group of $e$-th roots of unit in $F$.\par
Let $\cG_0^\dagger:=(\mathrm{Res}_{\cO'/\cO}\bbG)^{\mu_e}$ where $\mu_e$ acts simultaneously on $\cO'$ and $\bbG$ (through $\rho_G$). Let $\cG_0$ be the fiberwise identity component of $\cG_0^\dagger$. Similarly, we let $\cT$ be the fiberwise identity component of $(\mathrm{Res}_{\cO'/\cO}\bbT)^{\mu_e}$. By Lemma \ref{lem:res} and Lemma \ref{lem:fixed-point-locus-smooth}, the group schemes $\cG_0,\cG_0^\dagger,\cT$ are affine and smooth over $\cO$. The special fiber of $\cG_0$ is $\mathsf{G}_0:=\sG^{\mu_e,\circ}$, the identity component of the fixed point loci of the $\rho_G(\mu_e)$ action on $\sG$. Let $\mathsf{A}:=\bbT^{\mu_e,\circ}$ be the special fiber of $\cT$. Then $\sG_0$ is a reductive group over $k$ and $\sA$ is a maximal torus of $\sG_0$. \par
Let $\bfP_0^\dagger:=\cG_0^\dagger(\cO)$ and $\bfP_0:=\cG_0(\cO)$. Then $\bfP_0$ is a special parahoric subgroup of $G(F)$ corresponding to a special point that we denote by ``$0$" in the Bruhat-Tits building $\cB(G)$. By construction there is a natural reduction map $\mathrm{Red}_0:\cG_0^\dagger\to\sG$
whose image is $\sG^{\mu_e}$. It restricts to a surjective homomorphism (that we denote by the same symbol)
\begin{equation}\label{eq:special-group-reduction}
    \mathrm{Red}_0:\cG_0\to\sG_0
\end{equation}
Let $\bfI\subset\cG_0(\cO)$ be the inverse image of $\bbB(k)\cap\sG_0(k)$ under $\mathrm{Red}_0$. Then $\bfI$ is an Iwahori subgroup of $G(F)$ which corresponds to an alcove in $\cB(G)$ whose closure contains the special point $0$.\par 
In the following, for simplicity we will often denote the special $\cO$-model $\cG_0$ also by $G$ and denote $\fg=\mathrm{Lie}(\cG_0)$ so that we can write $\bfP_0=G(\cO)$, $\mathrm{Lie}(\bfP_0)=\fg(\cO)$ etc. If $G$ is split, then $F'=F$, $\cG_0=\cG_0^\dagger=\bbG=G$ and $\bfP_0=\bfP_0^\dagger=G(\cO)$ is a hyperspecial parahoric subgroup of $G(F)$.\par
Our main theorems will be valid under a condition that is slightly more general than being tamely ramified. We formulate it as follows.
\begin{defn}\label{def:ess-tame-ram}
    A reductive group $G$ over $F$ is \emph{essentially tamely ramified} if its adjoint group $G_{\mathrm{ad}}$ is a product of Weil restrictions of tamely ramified groups over finite extensions of $F$.
\end{defn}

\subsection{Conditions on the residue characteristic}\label{sec:p-assumption}
We review some standard conditions on the residue characteristic $p$ of $F$ determined by a reductive group.\par  
First let us recall the notion of bad primes and torsion primes for a root system, following \cite[Chapter I, \S4]{SS-conj}.
\begin{defn}
    A prime number $p$ is \emph{bad} (resp. \emph{torsion}) for a root system $\Sigma$ if $\bbZ\Sigma/\bbZ\Sigma_1$ (resp. $\bbZ\Sigma^\vee/\bbZ\Sigma_1^\vee$) has nontrivial $p$-torsion for some closed subsystem $\Sigma_1\subset\Sigma$. Here $\Sigma^\vee$ (resp. $\Sigma_1^\vee$) is the dual root system of $\Sigma$ (resp. $\Sigma_1$).\par 
    A prime number $p$ is \emph{bad} (resp. \emph{torsion}) for a reductive group or a reductive Lie algebra over a field if it is a bad (resp. torsion) for its (absolute) root system. A prime number $p$ is \emph{good} for a root system (or reductive group, or reductive Lie algebra) if it is not bad for it.
\end{defn}
The bad primes for an irreducible root system $\Sigma$ are listed according to the types of $\Sigma$ as follows:
\begin{itemize}
    \item Type $\mathrm{A}_n$: none;
    \item Type $\mathrm{B}_n(n\ge2),\mathrm{C}_n(n\ge3),\mathrm{D}_n(n\ge4)$: $p=2$;
    \item Type $\mathrm{E}_6,\mathrm{E}_7,\mathrm{F}_4,\mathrm{G}_2$: $p=2,3$;
    \item Type $\mathrm{E}_8$: $p=2,3,5$.
\end{itemize}
The torsion primes for an irreducible root system $\Sigma$ are listed according to the types of $\Sigma$ as follows:
\begin{itemize}
    \item Type $\mathrm{A}_n,\mathrm{C}_n(n\ge2)$: none;
    \item Type $\mathrm{B}_n(n\ge3), \mathrm{D}_n(n\ge4),\mathrm{G}_2$: $p=2$;
    \item Type $\mathrm{E}_6,\mathrm{E}_7,\mathrm{F}_4$: $p=2,3$;
    \item Type $\mathrm{E}_8$: $p=2,3,5$.
\end{itemize}
A prime number $p$ is bad (resp. torsion) for a general root system if it is so for some of its simple factors. We observe from the list above that the torsion primes for a root system are also bad primes (but not vice versa). Moreover, any bad prime of a root system divides the order of its Weyl group by \cite[I.4.10]{SS-conj}.\par 
Sometimes we need conditions on $p$ that depend not only on the root system but also the \emph{root datum} of a reductive group. Let $\cR=(X,\Sigma,X^\vee,\Sigma^\vee)$ be a root datum consisting of a pair of finite free $\bbZ$-modules $X,X^\vee$ dual to each other, and a subset $\Sigma\subset X$ (resp. $\Sigma^\vee\subset X^\vee$) of roots (resp. coroots). Let $\cR^{\mathrm{sc}}:=(X^{\mathrm{sc}},\Sigma,(X^{\mathrm{sc}})^\vee, \Sigma^\vee)$ be the associated simply-connected root datum in which $(X^{\mathrm{sc}})^\vee=\bbZ\Sigma^\vee$ is the lattice generated by the coroots of $\cR$. The root datum $\cR$ and $\cR^{\mathrm{sc}}$ have the same underlying root system $\Sigma$. 
\begin{defn}
    A prime number $p$ is \emph{torsion for the root datum $\cR=(X,\Sigma,X^\vee,\Sigma^\vee)$} if $X^\vee/\bbZ\Sigma_1^\vee$ has nontrivial $p$-torsion for some closed subsystem $\Sigma_1\subset\Sigma$.
\end{defn}
Note that 
\[|(X^\vee/\bbZ\Sigma_1^\vee)_{\mathrm{tor}}|=|(\bbZ\Sigma^\vee/\bbZ\Sigma_1^\vee)_{\mathrm{tor}}|\cdot|(X^\vee/\bbZ\Sigma^\vee)_{\mathrm{tor}}|.\]
The prime divisors of the first factor $|(\bbZ\Sigma^\vee/\bbZ\Sigma_1^\vee)_{\mathrm{tor}}|$ (when $\Sigma_1$ runs through all possible closed subsystems $\Sigma$) are precisely the torsion primes of the underlying root system $\Sigma$. The second factor above has several equivalent interpretations. Let $\cR^{\mathrm{der}}:=(X^{\mathrm{der}},\Sigma,(X^{\mathrm{der}})^\vee,\Sigma^\vee)$ be the ``derived" root datum of $\cR$ where $X^{\mathrm{der}}:=\mathrm{Im}(X\to X^{\mathrm{sc}})$ and $(X^{\mathrm{der}})^\vee:=X^\vee\cap\bbQ\Sigma^\vee$ (intersection inside $X^\vee\otimes_\bbZ\bbQ$) is the saturation of $\bbZ\Sigma^\vee$ in $X^\vee$. Then we have 
\[|(X^\vee/\bbZ\Sigma^\vee)_{\mathrm{tor}}|=|(X^{\mathrm{der}})^\vee/(X^{\mathrm{sc}})^\vee|=|X^{\mathrm{sc}}/X^{\mathrm{der}}|=|\mathrm{coker}(X\to X^{\mathrm{sc}})|.\]
In particular, $p$ is torsion for $\cR$ if either 
\begin{itemize}
    \item $p$ is torsion for the underlying root system of $\cR$, or
    \item $p$ divides the order of the cokernel of the natural homomorphism $X\to X^{\mathrm{sc}}$. 
\end{itemize}
So our definition of torsion primes for a root datum is the same as the definition in \cite[\S5]{Dem-torsion} by \emph{loc. cit.} Proposition 6, see also \cite[\S4.1.12]{BC22}.\par 
If $\cR$ is the root datum of a split reductive group $G$, then $\cR^{\mathrm{der}}$ (resp. $\cR^{\mathrm{sc}}$) is the root datum of the derived group $G^{\mathrm{der}}$ (resp. its simply connected cover $G^{\mathrm{sc}}$). The finite abelian groups $\mathrm{coker}(X\to X^{\mathrm{sc}})$ and $\ker(G^{\mathrm{sc}}\to G)$ are dual to each other (in the sense of algebraic groups) and hence they have the same order. Here order should be understood as dimension of coordinate ring of the algebraic group, e.g. $\mu_n$ has order $n$ over any field. On the other hand recall that the algebraic fundamental group of $G$ is defined by $\pi_1(G)=X^\vee/\bbZ\Sigma^\vee$ and its torsion subgroup $\pi_1(G)_{\mathrm{tor}}=\pi_1(G^{\mathrm{der}})=(X^{\mathrm{der}})^\vee/(X^{\mathrm{sc}})^\vee$ is a quotient of $\pi_1(G_{\mathrm{ad}})$. In the following table, we list the structure of $\pi_1(G_{\mathrm{ad}})$ for a split reductive group $G$ according to the types.
\begin{table}[!htbp]\label{tab:pi-1}
    \centering
    \begin{tabular}{|c|c|c|c|c|c|c|}
        \hline
        Type of $G$ & $A_{n-1}$ & $B_n,C_n,E_7$ & $D_n$ ($n$ even) & $D_n$ ($n$ odd) & $E_6$ & $E_8,F_4,G_2$\\
        \hline
        $\pi_1(G_{\mathrm{ad}})$ & $\bbZ/n\bbZ$ & $\bbZ/2\bbZ$ & $\bbZ/2\bbZ\times\bbZ/2\bbZ$ & $\bbZ/4\bbZ$ & $\bbZ/3\bbZ$ & 0\\
        \hline
    \end{tabular}
\end{table}
\newline
For later reference, we recall some classical results below about centralizers of semisimple elements in Lie algebras. This also helps us have a glimpse of the relevance of the various conditions on $p$ introduced above.
\begin{prop}\label{prop:centralizer-Levi-connected}
    Let $\sG$ be a connected reductive group over an algebraically closed field $\sk$ and let $\ga\in\mathrm{Lie}(\sG)$ be a semisimple element in its Lie algebra. Let $\sG_\ga$ be the centralizer of $\ga$ in $\sG$ and let $\sG_\ga^\circ$ be its identity component.  
    \begin{enumerate}
        \item The Lie algebra of $\sG_\ga$ coincides with the centralizer of $\ga$ in $\mathrm{Lie}(\sG)$. In other words, we have $\mathrm{Lie}(\sG_\ga)=\mathrm{Lie}(\sG)_\ga$. 
        \item Suppose that either $\mathrm{char}(\sk)=0$ or $\mathrm{char}(\sk)=p>0$ is a good prime for $\sG$, then $\sG_\ga^\circ$ is a Levi subgroup of $\sG$.
        \item Suppose that either $\mathrm{char}(\sk)=0$ or $\mathrm{char}(\sk)=p>0$ is not a torsion prime for the root datum of $\sG$, then $\sG_\ga$ is connected. 
    \end{enumerate}
\end{prop}
\begin{proof}
    (1) This is \cite[III.9, Proposition (2)]{Borel-LAG}.\par
    (2) This is \cite[Lemma 2.6.13(ii)]{Let}. Let us provide more details. Let $\sT\subset\sG$ be a maximal torus and let $\Sigma$ be the set of $\sT$-roots on $\mathrm{Lie}(\sG)$. After $\sG$-conjugation, we may and do assume that $\ga\in\mathrm{Lie}(\sT)$ by \cite[Corollary 4.5]{BS}. Let $\Sigma_\ga:=\{\alpha\in\Sigma, d\alpha (\ga)=0\}$. Then $\Sigma_\ga\subset\Sigma$ is a closed sub-system of $\Sigma$ and $\sG_\ga^\circ$ is a reductive group with maximal torus $\sT$ and root system $\Sigma_\ga$ by \cite[Lemma 3.7]{St-torsion} (see also \cite[Lemma 4.1.6]{BC22}). \par
    Take any root $\alpha\in\Sigma\cap\bbQ\Sigma_\ga$, where $\bbQ\Sigma_\ga$ is the $\bbQ$-subspace of $X^*(\sT)\otimes_\bbZ\bbQ$ spanned by $\Sigma_\ga$. Let $m$ be the smallest positive integer such that $m\alpha$ is a $\bbZ$-linear combination of elements in $\Sigma_\ga$. Then we have $m\cdot d\alpha(\ga)=0$ and $m$ divides the order of the torsion subgroup of $\bbZ\Sigma/\bbZ\Sigma_\ga$. Then the assumption on $\mathrm{char}(\sk)$ implies that $m$ is invertible in $\sk$ (see \cite[I.4.3]{SS-conj}). Hence $d\alpha(\ga)=0$ and we have $\alpha\in\Sigma_\ga$. This shows that $\Sigma\cap\bbQ\Sigma_\ga=\Sigma_\ga$. By \cite[VI \S1.7, Proposition 24]{Bourbaki-Lie}, there exists a set of simple roots for $\Sigma_\ga$ that can be enlarged to a set of simple roots for $\Sigma$. This means that $\sG_\ga^\circ$ is a Levi subgroup of $\sG$.\par 
    (3) If $\mathrm{char}(\sk)=0$, this is \cite[Corollary 3.11]{St-torsion}. If $\mathrm{char}(\sk)=p>0$ is a non-torsion prime for the root datum of $\sG$, this is \cite[Theorem 3.14]{St-torsion}.
\end{proof}

\section{Conjugacy classes in the groups}\label{sec:conj-group}
Let $F,\cO,k$ be as in \S\ref{sec:local-fields} and assume moreover that $k$ is algebraically closed. Let $G$ be a connected reductive group over $F$. In this section we will introduce some numerical invariants of conjugacy classes in $G(F)$.
\subsection{Discriminant valuation}
Recall that an element $\ga\in G(F)$ is \emph{regular semisimple} if it is semisimple and its centralizer $G_\ga$ has minimal possible dimension. In fact, $\ga$ is regular semisimple if and only if the identity component $G_\ga^\circ$ is a maximal torus of $G$ defined over $F$. There is a nonempty open subscheme $G^{\mathrm{rs}}\subset G$ such that $G^{\mathrm{rs}}(F)$ is the set of all regular semisimple elements. 
\begin{defn}\label{def:disc-group}
The \emph{discriminant valuation} of a regular semisimple element $\ga\in G^\mathrm{rs}(F)$ is the integer defined by
    \[d_G(\ga):=\mathrm{val}\det(\mathrm{Id}-\mathrm{ad}(\ga) | \fg(F)/\fg_\ga(F)).\]
\end{defn}
We can compute $d_G(\ga)$ as follows. After $G(\overline{F})$-conjugation, we may assume that $\ga\in T(\overline{F})$ for a maximal torus $T$ in $G$. Let $R$ be the set of roots of $T$ on $\fg$. Then we have
\begin{equation}
    d_G(\ga)=\sum_{\alpha\in R}\mathrm{val}(1-\alpha(\ga)).
\end{equation}

\subsection{Kottwitz homomorphism}
The Kottwitz homomorphism is a generalization of the determinant valuation on $\mathrm{GL}_n(F)$ to all reductive groups. It gives an obstruction for an element to lie in a parahoric subgroup. Let $\pi_1(G)=\pi_1(G_{\overline{F}})$ be the algebraic fundamental group of $G$, which is a finitely generated abelian group with an action of the Galois group $\Gamma=\mathrm{Gal}(\overline{F}/F)$. We have the \emph{Kottwitz homomorphism}
\[\kappa_G:G(F)\to\pi_1(G)_\Gamma\]
where\footnote{In general, one needs to first take coinvariants under the inertia group of $F$ and then take invariants under a Frobenius. In our case, the residue field of $F$ is algebraically closed and hence the inertia group equals to the whole Galois group.} $\pi_1(G)_\Gamma$ is the group of $\Gamma$-coinvariants of $\pi_1(G)$. \par 
Let $G(F)^0:=\ker(\kappa_G)$ and let $G(F)^\dagger$ be the inverse image of the torsion subgroup of $\pi_1(G)_\Gamma$ under $\kappa_G$. Any parahoric subgroup $\bfP\subset G(F)$ is contained in $G(F)^0$ and the Kottwitz homomorphism $\kappa_G$ induces an injective homomorphism $\widetilde{\bfP}/\bfP\into\pi_1(G)_\Gamma$ which is bijective if $\bfP$ is an Iwahori subgroup, see \cite[Proposition 11.6.1]{KP23}.\par
Let $S$ be a maximal $F$-split torus in $G$ and let $N$ (resp. $T$) be the normalizer (resp. centralizer) of $S$ in $G$. By Steinberg's theorem, the assumption on $F$ implies that $G$ is quasi-split and $T$ is a maximal torus of $G$ defined over $F$. Let $\cT$ be the fiberwise identity component of the finite type N\'eron model of $T$. Then $\cT(\cO)$ is the unique parahoric subgroup of $T(F)$ and the Kottwitz homomorphism $\kappa_T$ for $T$ induces an isomorphism $X_*(T)_\Gamma\cong T(F)/\cT(\cO)$. Let $W=N(F)/T(F)$ be the (relative) Weyl group of $G$ and let $\widetilde{W}:=N(F)/\cT(\cO)$ be the Iwahori-Weyl group of $G$. Let $W_{\mathrm{aff}}\subset\widetilde{W}$ be the affine Weyl group of the root system of $G$.\par
We choose a special point, denoted ``$0$", and an alcove $\fa$ whose closure contains $0$, both in the apartment for $S$ in the Bruhat-Tits building of $G$. They determine decompositions 
\[\widetilde{W}=X_*(T)_\Gamma\rtimes W,\quad \widetilde{W}=W_{\mathrm{aff}}\rtimes\widetilde{W}_\fa\]
where $\widetilde{W}_\fa$ is the stabilizer of $\fa$ in $\widetilde{W}$. Moreover, the Kottwitz homomorphism induces an isomorphism 
\[\iota_\fa:\widetilde{W}_\fa\cong\pi_1(G)_\Gamma.\] 
See for example \cite[\S7.8, \S11.5]{KP23} (note that in \emph{loc. cit.}, the Iwahori-Weyl group is denoted by $\widetilde{W}^0$).
\begin{defn}\label{def:defect-Kott-invariant}
    Define the homomorphism $w^G_\fa:\pi_1(G)_\Gamma\to W$ to be the composition of the following natural homomorphisms:
    \[w^G_\fa:\pi_1(G)_\Gamma\cong\widetilde{W}_\fa\into\widetilde{W}\to W\]
    where the first map is the inverse of $\iota_\fa$.\par  
    The \emph{defect} of an element $x\in\pi_1(G)_\Gamma$ is the non-negative integer 
    \[\mathrm{def}(x):=\dim_\bbQ X_*(S)_\bbQ-\dim_\bbQ X_*(S)_\bbQ^{w^G_\fa(x)}\]
    where $X_*(S)_\bbQ:=X_*(S)\otimes_\bbZ\bbQ$ and $X_*(S)_\bbQ^{w^G_\fa(x)}$ is the subspace of $w^G_\fa(x)$-invariants. 
\end{defn}
The invariant $\mathrm{def}(x)$ is independent of the choice of the alcove $\fa$. Indeed, the group $N(F)$ acts transitively on the set of alcoves in the apartment of $S$ and for any $n\in N(F)$ the maps $w_\fa$ and $w_{n(\fa)}$ are conjugate by $n$. Therefore $\dim_\bbQ X_*(S)_\bbQ^{w_\fa(x)}=\dim_\bbQ X_*(S)_\bbQ^{w_{n(\fa)}(x)}$. Moreover $\mathrm{def}(x)$ is also independent of the choice of the maximal $F$-split torus $S$ since any two such tori are conjugate by an element of $G(F)$, which induces canonical isomorphisms between the associated (Iwahori) Weyl groups. It will be convenient to have a description of $\mathrm{def}(x)$ that does not depend on the torus $S$. For this we have the following simple lemma which follows directly from the definitions:
\begin{lem}\label{lem:defect-Kott-invariant}
    Let the notation be as in Definition \ref{def:defect-Kott-invariant}. For any $x\in\pi_1(G)_\Gamma$, choose an element $\Tilde{x}\in\kappa_G^{-1}(x)\subset G(F)$ that stabilizes the barycenter of the alcove $\fa$ (in other words, $\tilde{x}$ lies in the normalizer of the Iwahori subgroup of $G(F)$ corresponding to $\fa$). Then the defect $\mathrm{def}(x)$ equals to the codimension of the fixed point loci $\fa^{\Tilde{x}}$ in the alcove $\fa$. 
\end{lem}
\subsection{Topological Jordan decomposition for the group}\label{sec:top-Jordan-group}
Recall from \cite[Definition 2.2.1]{KP23} that for an affine $F$-scheme $S$ of finite type, a subset $B\subset S(F)$ is \emph{bounded} if for any regular function $f\in F[S]$ the set $\{|f(b)|,b\in B\}$ is bounded above. We say that an element $\ga\in G(F)$ is \emph{bounded} if the subgroup $\langle\ga\rangle\subset G(F)$ generated by $\ga$ is bounded. 

\begin{lem}\label{lem:bounded}
    For an element $\ga\in G(F)$, the following are equivalent:
    \begin{enumerate}
        \item $\ga$ is bounded;
        \item For any algebraic representation $\rho$ of $G$ (defined over $\overline{F}$), any generalized eigenvalue $\la$ of $\rho(\ga)$ in $\overline{F}$ satisfies $|\la|=1$;
        \item There exists a parahoric subgroup $\bfP\subset G(F)$ such that $\ga\in\bfP^\dagger$. In other words, $\ga$ has a fixed point in the extended Bruhat-Tits building of $G(F)$.
        \item There exists an Iwahori subgroup $\bfI\subset G(F)$ such that $\ga\in\bfI^\dagger$. In other words, $\ga$ fixes an alcove in the extended Bruhat-Tits building of $G(F)$.
    \end{enumerate}
\end{lem}
\begin{proof}
    (1)$\Leftrightarrow$(2): Suppose $\ga\in G(F)$ is bounded and let $\rho:G\to\mathrm{GL}(V)$ be an algebraic representation (defined over $\overline{F}$). Then $\rho(\ga)$ is a bounded subgroup of $\mathrm{GL}(V)$ and hence for any generalized eigenvalue $\la\in\overline{F}$ of $\rho(\ga)$, the set $\{|\la^n|,n\in\bbZ\}$ is bounded above and so we must have $|\la|=1$. \par
    Suppose that $\ga\in G(F)$ is unbounded. Let $H$ be the Zariski closure of $\langle\ga\rangle$ in $G$. Then $\ga$ is also unbounded in $H$. Let $\rho:G\to\mathrm{GL}(V)$ be an algebraic representation with finite kernel. By \cite[Lemma 2.2.11]{KP23} there is an eigenvalue $\la$ of $\rho(\ga)$ such that $|\la|>1$.\par 
    (1)$\Leftrightarrow$(3) By \cite[Proposition 2.2.13]{KP23}, any bounded subgroup of $G(F)$ is contained in a maximal bounded open subgroup. Therefore the equivalence follows from \emph{loc.cit.} Theorem 4.2.15.\par 
    (3)$\Leftrightarrow$(4): Let $\bfI\subset G(F)$ be an Iwahori subgroup contained in $\bfP$ and suppose $\ga\in\bfP^\dagger$. Let $\cG_\bfP^\dagger$ be the smooth $\cO$-model of $G$ such that $\cG_\bfP^\dagger(\cO)=\bfP^\dagger$ and let $\sG_\bfP^\dagger$ be its special fiber. Let $\Bar{\ga}\in\sG_\bfP^\dagger(k)$ be the image of $\ga\in\bfP^\dagger$ under the natural reduction map. The image of $\bfI$ in $\sG_\bfP$ is a Borel subgroup that we denote by $\sB_\bfP$. By \cite[Theorem 7.2 on page 49]{St-End}, the flag variety $\sG_\bfP/\sB_\bfP$ has a fixed point under the conjugation action of $\Bar{\ga}$. Therefore $\ga$ is $\bfP^\dagger$ conjugate to an element in $\bfI^\dagger$ and we are done.  
\end{proof}

\begin{defn}\label{def:top-unip}
    An element $\ga\in G(F)$ is said to be:
    \begin{itemize}
        \item \emph{strongly semisimple} if it is semisimple and for any algebraic representation $\rho$ of $G$ (defined over $\overline{F}$), each eigenvalue of $\rho(\ga)$ lies in the set of Teichm\"{u}ller representatives $[\Bar{k}^\times]$;
        \item \emph{topologically unipotent} if for any algebraic representation $\rho$ of $G$ (defined over $\overline{F}$) and any generalized eigenvalue $\alpha$ of $\rho(\ga)$ in $\overline{F}$ (i.e. $\alpha$ is a root of the characteristic polynomial of $\rho(\ga)$), we have $|\alpha-1|<1$;
        \item \emph{strongly topologically unipotent} if there exists a parahoric subgroup $\bfP\subset G(F)$ with pro-unipotent radical $\bfP_+$ such that $\ga\in\bfP_+$.
    \end{itemize}
    An element $\ga\in G(F)$ is \emph{(strongly) topologically unipotent} (resp. \emph{strongly semisimple}, \emph{bounded}) \emph{mod center} if its image in $G_{\mathrm{ad}}(F)$ has this property. 
\end{defn}
\begin{lem}\label{lem:tame-centralizer}
    Suppose that the reductive group $G$ is (essentially) tamely ramified over $F$ in the sense of Definition \ref{def:ess-tame-ram}. For any element $s\in G(F)$ that is strongly semisimple mod center, the identity component of its centralizer $G_s^\circ$ is also (essentially) tamely ramified. 
\end{lem}
\begin{proof}
    We may assume that $G$ is adjoint and tamely ramified. By \cite[Corollary 2.37]{Spice08}, the element $s\in G(F)$ is $F$-tame in the sense of \emph{loc. cit.} Definition 2.4. This means that after passing to a tamely ramified extension of $F$, we may and do assume that $s$ is contained in a split torus. Then $s$ is contained in a maximal $F$-split torus $S$. Let $T$ be the centralizer of $S$ in $G$. Then $T$ is a maximal $F$-torus of $G_s^\circ$ that splits over a tamely ramified extension since $G$ is tamely ramified. This means that $G_s^\circ$ is tamely ramified. 
\end{proof}

\begin{lem}
    Let $G$ be a reductive group over $F$ and let $\kappa_G:G(F)\to\pi_1(G)_\Gamma$ be its Kottwitz homomorphism, where $\Gamma=\mathrm{Gal}(\overline{F}/F)$. Recall that $F$ is complete and the residue field $k$ is algebraically closed.  
    \begin{enumerate}
        \item An element $\ga\in G(F)$ is topologically unipotent if and only if there exists a finite extension $E/F$ such that $\ga$ is strongly topologically unipotent in $G(E)$. In particular, any strongly topologically unipotent element in $G(F)$ is topologically unipotent.
        \item An element $\ga\in G(F)$ is strongly topologically unipotent if and only if it is topologically unipotent and $\kappa_G(\ga)=1$.
        \item If the residue characteristic $p$ does not divide the order of the torsion subgroup of $\pi_1(G)_\Gamma$, then an element $\ga\in G(F)$ is  topologically unipotent if and only if it is strongly topologically unipotent.
    \end{enumerate}
\end{lem}
\begin{proof}
    (1) If $\ga\in G(E)$ is strongly topologically unipotent for some finite extension $E/F$, then $\ga\in\bfP_+$ for some parahoric subgroup $\bfP\subset G(E)$ and therefore $\lim_{n\to\infty}\ga^{p^n}=1$. In other words, $\ga$ is topologically $p$-unipotent in the sense of \cite[Definition 1.3]{Spice08}. This implies that $\ga$ is topologically unipotent. Conversely if $\ga\in G(F)$ is topologically unipotent, then it is topologically $p$-unipotent and hence by \cite[Lemma 2.21]{Spice08} there is a finite extension $E/F$ such that $\ga$ is strongly topologically unipotent in $G(E)$.\footnote{In fact, this is the definition of ``topologically $F$-unipotent" from \cite[Definition 2.15]{Spice08}.}\par
    (2) The implication ``$\Rightarrow$" is clear. Suppose that $\ga\in G(F)$ is topologically unipotent and $\kappa_G(\ga)=1$. By Lemma \ref{lem:bounded} there exists a parahoric subgroup $\bfP\subset G(F)$ such that $\ga\in\bfP^\dagger$. Then we must have $\ga\in\bfP$ since $\kappa_G(\ga)=1$. By \cite[Theorem 8.4.19]{KP23} there is a parahoric subgroup $\bfQ\subset G(F)$, whose corresponding facet in the Bruhat-Tits building $\cB(G)$ contains the facet of $\bfP$ in its closure, such that $\ga\in\bfQ_+$ and hence $\ga$ is strongly topologically unipotent.\par
    (3) Suppose moreover that $p$ does not divide the order of the torsion subgroup of $\pi_1(G)_\Gamma$. Let $\ga\in G(F)$ be a topologically unipotent element. Then $\lim_{n\to\infty}\ga^{p^n}=1$ and hence $\kappa_G(\ga)$ lies in the $p$-power torsion subgroup of $\pi_1(G)_\Gamma$. By assumption we must have $\kappa_G(\ga)=1$ and hence $\ga$ is strongly topologically unipotent by (2).
\end{proof}

\begin{ex}
    Suppose the residue characteristic of $F$ equals to $2$. Then the element in $\mathrm{PGL}_2(F)$ represented by the matrix $\left(\begin{smallmatrix}0&1\\ \varpi&0\end{smallmatrix}\right)$ is topologically unipotent but not strongly topologically unipotent. Thus in $\mathrm{GL}_2(F)$, this matrix is topologically unipotent mod center but not strongly topologically unipotent mod center. 
\end{ex}

Finally we state the following result of Spice \cite[Theorem 2.38]{Spice08} on the existence and uniqueness of topological Jordan decompositions.
\begin{thm}[Spice]
    For any bounded element $\ga\in G(F)$, there exists a unique strongly semisimple element $s\in G(F)$ and a unique topologically unipotent element $u\in G(F)$ such that $\ga=su=us$. Moreover, we have $s,u\in G_s^\circ(F)$ where $G_s^\circ$ is the identity component of the centralizer of $s$ in $G$. 
\end{thm}

\subsection{Arithmetic invariants of tori}
In this subsection we introduce a numerical invariant of a conjugacy class defined by the Artin conductor of its centralizer. Then we review a result of Chai-Yu \cite{ChaiYu} on the Artin conductor of tori. This result will later be used in Theorem \ref{thm:dim-reg} to compute the dimension of the regular locus of the Witt vector affine Springer fiber for Lie algebras.\par 
Let $T$ be an $F$-algebraic group of multiplicative type. We are only interested in the case of tori but the slightly more general setting would make no difference in most of the following discussion. Let $\La:=X_*(T_{\overline{F}})$ be the lattice of coweights of $T$ defined over $\overline{F}$, equipped with an action of $\mathrm{Gal}(\overline{F}/F)$ through a homomorphism $\rho_T:\mathrm{Gal}(\overline{F}/F)\to\mathrm{GL}(\La):=\mathrm{Aut}_\bbZ(\La)$. We will define some isogeny invariants of $T$ associated to the representation
\[\rho_{T,\bbQ}=\rho_T\otimes_\bbZ\bbQ:\mathrm{Gal}(\overline{F}/F)\to\mathrm{GL}(\La_\bbQ)\]
where $\La_\bbQ:=\La\otimes_\bbZ\bbQ$.
\begin{defn}\label{def:tori-invariant}
    The \emph{Artin conductor} of $T$, denoted $\mathrm{Art}(T)$, is defined to be the Artin conductor of the representation $\rho_{T,\bbQ}$. The \emph{defect} of $T$ is the non-negative integer
    \[\mathrm{def}(T):=\dim_\bbQ\La_\bbQ-\dim_\bbQ\La_\bbQ^{\mathrm{Gal}(\overline{F}/F)}\]
    which is the difference between the $\overline{F}$-rank of $T$ and the $F$-rank of $T$.
\end{defn}
The Artin conductor $\mathrm{Art}(T)$ equals to $\mathrm{def}(T)$ if and only if $T$ splits over a tamely ramified extension of $F$. 
In general $\mathrm{Art}(T)$ is always a non-negative integer by the Hasse-Arf theorem (see \cite[Chapter VI.\S2 Theorem 1']{Serre-local}) and $\mathrm{Art}(T)\ge\mathrm{def}(T)$. Their difference $\mathrm{Art}(T)-\mathrm{def}(T)$ is the \emph{Swan conductor} of the representation $\rho_{T,\bbQ}$ and measures the wild ramification of $T$. \par
We extend this definition to general reductive groups over $F$ as follows.
\begin{defn}\label{def:conductor-group}
    The \emph{Artin conductor} (resp. \emph{defect}) of a reductive group $G$ over $F$ are the integers
    \[\mathrm{Art}(G):=\mathrm{Art}(T),\quad\mathrm{def}(G):=\mathrm{def}(T)\]
    where $T\subset G$ is a maximally split maximal torus defined over $F$, i.e. $T$ is a maximal torus of $G$ defined over $F$ and contains a maximal $F$-split torus $S\subset G$. When $G$ is quasi-split, then $T=Z_G(S)$ is the centralizer of $S$.       
\end{defn}
Now let $T$ be a torus over $F$ and let $\ft=\mathrm{Lie}(T)$ be the Lie algebra of $T$, viewed either as an $F$-scheme or an $F$-vector space. We specialize the discussion in \S\ref{sec:parahoric} to the case where $G=T$. Then we get a unique smooth $\cO$-model $T^\dagger$ of $T$ such that $T^\dagger(\cO)$ is the maximal bounded subgroup of $T(F)$ and the fiberwise identity component $T^{\dagger,0}$ is the unique parahoric group scheme of $T$. In fact $T^\dagger$ is the finite type N\'eron model of $T$ in the terminology of \cite[\S3]{ChaiYu} and \cite{Ngo10}\footnote{In \cite{ChaiYu} it is denoted by $\underline{T}^{\mathrm{NR}}=\underline{T}^{\mathrm{ft\; NR}}$, whereas in \cite{Ngo10} it is denoted by $T^\flat$.}. The Lie algebra $\ft^\dagger:=\mathrm{Lie}(T^\dagger)$ is an $\cO$-lattice in $\ft(F)$. Let $L/F$ be a finite Galois extension contained in $\overline{F}$ such that $\rho_T$ factors through $\mathrm{Gal}(L/F)$ and let $\cO_L$ be the valuation ring of $L$. Then $T$ is split over $L$ and there is a natural homomorphism $(T^\dagger)_{\cO_L}\to(T_L)^\dagger$ that induces an embedding of $\cO_L$-modules $\ft^\dagger\otimes_\cO\cO_L\subset\ft_L^\dagger$. 
\begin{thm}[\cite{ChaiYu}]\label{thm:Chai-Yu}
    Let $e(L|F)$ be the ramification index of the extension $L/F$. Then we have
    \[\frac{1}{e(L|F)}\mathrm{length}_{\cO_L}\left(\frac{\ft_L^\dagger}{\ft^\dagger\otimes_\cO\cO_L}\right)=\frac{1}{2}\mathrm{Art}(T).\]
\end{thm}
\begin{proof}
    This follows by combining Proposition 10.2 and Theorem 11.3 in \cite{ChaiYu}.
\end{proof}
Now suppose that there exists a non-degenerate symmetric $F$-bilinear form $(\cdot,\cdot)$ on $\La\otimes_\bbZ F$ that is invariant under $\mathrm{Gal}(L/F)$ (which acts on $\La$ via $\rho_T$). Recall that the dual of an $\cO$-lattice $M\subset\La\otimes_\bbZ F$ is defined to be
\[M^\vee:=\{x\in \La\otimes_\bbZ F\mid (x,M)\subset\cO\}.\]
We would also like to define the dual of $\cO$-lattices in $\ft(F)$. For this we first extend $(\cdot,\cdot)$ to an $L$-bilinear form $(\cdot,\cdot)_L$ on $\La\otimes_\bbZ L$. Since $T$ splits over $L$, there are canonical identifications $\ft(L)=\La\otimes_\bbZ L$ and 
\[\ft(F)=(\La\otimes_\bbZ L)^{\mathrm{Gal}(L/F)}\] 
where $\mathrm{Gal}(L/F)$ acts diagonally. By $\mathrm{Gal}(L/F)$-invariance of $(\cdot,\cdot)$, the extended form $(\cdot,\cdot)_L$ restricts to an $F$-bilinear form on $\ft(F)$. In other words, for any $x,y\in\ft(F)$ we have $(x,y)_L\in F$. Then for any $\cO$-lattice $M\subset\ft(F)$, we define its dual to be
\[M^\vee:=\{x\in\ft(F)\mid (x,M)_L\subset\cO\}=\{x\in\ft(F)\mid (x,M)_L\subset\cO_L\}\]
where the equality follows from the equality $\cO_L\cap F=\cO$. 
\begin{cor}\label{cor:length-artin-conductor}
    With notations as above. Assume moreover that $\La\otimes_\bbZ\cO$ is self-dual under $(\cdot,\cdot)$. Then we have
    \[\mathrm{length}_\cO(\ft^{\dagger,\vee}/\ft^\dagger)=\mathrm{Art}(T).\]
\end{cor}
\begin{proof}
    First we claim that for any $\cO$-lattice $M\subset\ft(F)$ we have \[M^\vee\otimes_\cO\cO_L=(M\otimes_\cO\cO_L)^\vee\]
    where the right side means the dual of the $\cO_L$-lattice under the pairing $(\cdot,\cdot)_L$. Indeed, after choosing an $\cO$-bases $\{m_1,\dotsc,m_r\}$ for $M$ and letting $m_1^\vee,\dotsc,m_r^\vee$ be the dual bases satisfying $(m_i,m_j^\vee)_L=\delta_{ij}$, both sides are identified with the $\cO_L$ submodule of $\La\otimes_\bbZ L$ spanned by $\{m_i^\vee,1\le i\le r\}$.\par 
    Apply this identity to $M=\ft^\dagger$, we get that
    \[\mathrm{length}_\cO(\ft^{\dagger,\vee}/\ft^\dagger)=\frac{1}{e(L|F)}\mathrm{length}_{\cO_L}(\ft^{\dagger,\vee}/\ft^\dagger)\otimes_\cO\cO_L
    =\frac{1}{e(L|F)}\mathrm{length}_{\cO_L}\frac{(\ft^\dagger\otimes_\cO\cO_L)^\vee}{\ft^\dagger\otimes_\cO\cO_L}.\]
    Since $T$ splits over $L$ we have $\ft_L^\dagger=\La\otimes_\bbZ\cO_L$, which is self-dual by assumption. So we have
    \[\ft^\dagger\otimes_\cO\cO_L\subset\ft_L^\dagger=\ft_L^{\dagger,\vee}\subset(\ft^\dagger\otimes_\cO\cO_L)^\vee\]
    and then by Theorem \ref{thm:Chai-Yu} we conclude that
    \[\mathrm{length}_\cO(\ft^{\dagger,\vee}/\ft^\dagger)=\frac{2}{e(L|F)}\mathrm{length}_{\cO_L}\left(\frac{\ft_L^\dagger}{\ft^\dagger\otimes_\cO\cO_L}\right)=\mathrm{Art}(T).\]
\end{proof}
We end this section with a result on the behavior of Artin conductor under restrictions of scalars. Later this will be used in several reduction steps (see \S\ref{sec:more-assumption} and Proposition \ref{prop:HC-descent}) in the proof of the main theorems.
\begin{prop}\label{prop:Artin-conductor-Res}
    Let $E/F$ be a finite extension field and let $\mathrm{Disc}_{E/F}\subset\cO$ be its discriminant ideal. Let $H$ be a reductive group over $E$ with absolute rank $r$. Then we have
    \[\mathrm{Art}(\mathrm{Res}_{E/F}H)=\mathrm{val}(\mathrm{Disc}_{E/F})r+\mathrm{Art}(H)\]
\end{prop}
\begin{proof}
    We may assume that $E\subset\overline{F}$. After replacing $H$ by a maximally $F$-split sub-torus, we may assume that $H=T$ is a torus over $E$. Then the $\mathrm{Gal}(\overline{F}/F)$-module $X_*(\mathrm{Res}_{E/F}T)$ is induced from the $\mathrm{Gal}(\overline{F}/E)$-module $X_*(T)$ and the result follows from \cite[VI \S2, Cor. to Prop. 4]{Serre-local}. 
\end{proof}

\section{Conjugacy classes in the Lie algebras}\label{sec:conj-Lie}
Let $F,\cO,k$ be as in \ref{sec:local-fields} and assume moreover that $k$ is algebraically closed. Let $G$ be a connected reductive group over $F$ with Lie algebra $\fg$. Keep the notations from \S\ref{sec:parahoric}. In particular $G$ splits over $F'$ and is the twisted form of the split group $\bbG$ by the homomorphism $\rho_G:\mathrm{Gal}(F'/F)\to\mathrm{Out}(\bbG)$.\par 
In this section we review some facts about conjugacy classes in $\fg(F)$. Most of the results here are well-known if $p=\mathrm{char}(k)$ is sufficiently large, but not so otherwise. We have to pay attention to the exact conditions on $p$ throughout. 

\subsection{Chevalley quotient and regular semisimple elements}\label{sec:chevalley-quotient}
The \emph{Chevalley base} for the Lie algebra $\fg$ is the invariant quotient of $\fg$ (viewed as an $F$-scheme) by the adjoint action of $G$, denoted $\fc=\fg//G:=\spec(F[\fg]^{G})$. It comes with a natural morphism 
\[\chi_G:\fg\to\fc\]
called the \emph{Chevalley morphism}. Similarly, for the split form $\bbG$ of $G$ we have its Chevalley morphism $\chi_\bbG:\bbg\to\bbc$ that is actually defined over $\bbZ$.\par 
The $\mathrm{Out}(\bbG)$-action and the adjoint action of $\bbG$ on $\bbg$ combine into an action of the semi-direct product $\bbG\rtimes\mathrm{Out}(\bbG)$ on $\bbg$, which then induces an action of $\mathrm{Out}(\bbG)$ on $\bbc$ making the morphism $\chi_\bbG$ equivariant under $\mathrm{Out}(\bbG)$. Consequently $\chi_G$ is the $\mathrm{Out}(\bbG)$-twist of $\chi_\bbG$ over $F'$ defined by the homomorphism $\rho_G:\mathrm{Gal}(F'/F)\to\mathrm{Out}(\bbG)$. \par 
Recall that an element in $\fg$ (resp. $\bbg$) is \emph{regular} if its centralizer in $G$ (resp. $\bbG$) has minimal possible dimension (or equivalently, its adjoint orbit has maximal possible dimension). 
The regular elements in $\bbg$ form an open subscheme $\bbg^{\mathrm{reg}}\subset\bbg$ that is $\bbG\rtimes\mathrm{Out}(\bbG)$-stable and descends to the open subscheme $\fg^{\mathrm{reg}}$ of regular elements in $\fg$. Then the Chevalley morphism $\chi_G$ is flat and its restriction to the regular open subscheme $\chi_G^{\mathrm{reg}}:\fg^{\mathrm{reg}}\to\fc$ is smooth. This follows by Galois descent from the corresponding statements for the split group $\bbG$. \par
Let $\bbg^{\mathrm{rs}}\subset\bbg^{\mathrm{reg}}$ be the open subscheme of regular semisimple elements. There is an open subscheme $\bbc^{\mathrm{rs}}\subset\bbc$ such that $\bbg^{\mathrm{rs}}=\chi_\bbG^{-1}(\bbc^{\mathrm{rs}})$. The complement of $\bbc^{\mathrm{rs}}$ in $\bbc$ is the vanishing loci of the discriminant function $\Delta_\bbg$. Let us recall its definition below. For this we view the split group $\bbG$ and related objects $\bbg,\bbc$ etc. to be defined over $\bbZ$.\par 
Let $P(X)\in \bbZ[\bbg][X]$ be the universal characteristic polynomial for adjoint actions of elements in $\bbg$. In other words, for any ring $R$ and any $\ga\in\bbg(R)$, viewed as a morphism $\ga:\spec R\to\bbg$, the pullback $\ga^*P(X)$ is the characteristic polynomial of the endomorphism $\mathrm{ad}(\ga)$ on the free $R$-module $\bbg(R)$. The characteristic polynomials are invariant under the action of $\bbG\rtimes\mathrm{Out}(\bbG)$ and therefore we have $P(X)\in\bbZ[\bbc]^{\mathrm{Out}(\bbG)}$, the set of $\mathrm{Out}(\bbG)$-invariant polynomials on $\bbc$ with $\bbZ$-coefficients. Moreover $P(X)$ is divisible by $X^{\bbr}$, where $\bbr$ is the rank of $\bbG$. The coefficient of the $X^{\bbr}$ term is the \emph{discriminant} $\Delta_\bbg\in\bbZ[\bbc]^{\mathrm{Out}(\bbG)}$ that we also view as an element in $\bbZ[\bbg]$. Then the regular semisimple locus $\bbc^{\mathrm{rs}}$ is the complement of the principal divisor $\Delta_{\bbg}=0$. \par
The open subschemes $\bbg^{\mathrm{rs}}$ and $\bbc^{\mathrm{rs}}$ are stable under the $\mathrm{Out}(\bbG)$-action and descend to open subschemes $\fg^{\mathrm{rs}}\subset\fg$ and $\fc^{\mathrm{rs}}\subset\fc$ of regular semisimple elements. Let $\Delta_\fg\in F[\fc]=(F'[\bbc])^{\mathrm{Gal}(F'/F)}$ be the image of $\Delta_\bbg$. Then $\fc^{\mathrm{rs}}$ is the complement of the principal divisor $\Delta_\fg=0$ on $\fc$.
\begin{lem}
    For any regular semisimple element $\ga\in\fg^{\mathrm{rs}}(F)$ we have 
    \[\Delta_\fg(\ga)=\det(\mathrm{ad}(\ga)\mid\fg(F)/\fg_\ga(F))\]
    where $\fg_\ga=\ker(\mathrm{ad}(\ga))=\mathrm{Lie}(G_\ga)$ is the centralizer of $\ga$ in the Lie algebra. More generally, for any semisimple element $\ga\in\ft(\overline{F})=\bbt(\overline{F})$ we have
    \[\Delta_\fg(\ga)=\prod_{\alpha\in R}d\alpha(\ga)\]
    where $R$ is the set of roots of $\bbT$ in $\bbg$ and $d\alpha:\bbt(\overline{F})\to\overline{F}$ is the differential of a root $\alpha\in R$. 
\end{lem}
\begin{proof}
    We can work over $\overline{F}$ and after conjugation we may assume that $\ga\in\ft(\overline{F})$ is regular semisimple. Then $\fg_\ga(\overline{F})=\ft(\overline{F})$ and the first equality follows. Then the second equality follows from the first for regular elements in $\ft(\overline{F})$. If $\ga\in\ft(\overline{F})$ is not regular, then both sides of the second equality are $0$ and we are done. 
\end{proof}
\begin{defn}\label{def:disc-Lie-alg}
    The \emph{discriminant valuation} of a regular semisimple element $\ga\in\fg^{\mathrm{rs}}(F)$ is the integer defined by
    \[d_\fg(\ga):=\mathrm{val}(\Delta_\fg(\ga))=\mathrm{val}\det(\mathrm{ad}(\ga)\mid\fg(F)/\fg_\ga(F)).\]
\end{defn}
Assume in the rest of this subsection that the extension $F'/F$ is tamely ramified of degree $e\ge1$ and identify $\mathrm{Gal}(F'/F)\cong\mu_e$. The construction in \S\ref{sec:group-model} provides natural finite type $\cO$-models for the $F$-schemes $\fg$, $\fg^{\mathrm{reg}}$, $\fc$ etc. We use the same symbol to denote the resulting $\cO$-schemes of finite type. Then we have 
\[\fg=(\mathrm{Res}_{\cO'/\cO}\bbg)^{\mu_e}=\mathrm{Lie}(\cG_0),\quad\fg^{\mathrm{reg}}=(\mathrm{Res}_{\cO'/\cO}\bbg^{\mathrm{reg}})^{\mu_e},\quad\fc=(\mathrm{Res}_{\cO'/\cO}\bbc)^{\mu_e}.\]
Then we can talk about the $\cO$-points $\fg(\cO),\fc(\cO)$ etc. In particular, from the definition we see that $\fg^{\mathrm{reg}}(\cO) = \fg(\cO)\cap\bbg^{\mathrm{reg}}(\cO')$.\par
For any closed point $\ga\in\bbg$, we have $\dim\bbG_\ga\ge\bbr$, the absolute rank of $\bbG$. The locus where equality holds is the open subscheme $\bbg^{\mathrm{reg}}\subset\bbg$. 
\begin{defn}
    Let $\bbI$ be the universal centralizer group scheme over $\bbg$ whose fiber over any point $\ga\in\bbg$ is the centralizer $\bbG_\ga$ of $\ga$ in $\bbG$. Let $\bbI^{\mathrm{reg}}$ be the restriction of $\bbI$ to $\bbg^{\mathrm{reg}}$. Similarly we let $I$ be the universal centralizer group scheme over $\fg$ and let $I^{\mathrm{reg}}$ be the restriction of $I$ to $\fg^{\mathrm{reg}}$. 
\end{defn}
Then $I^{\mathrm{reg}}$ is an open subscheme of the fixed point subscheme of $\mu_e$ on $(\mathrm{Res}_{\cO'/\cO}\bbI^{\mathrm{reg}})|_{\fg^\mathrm{reg}}$, the base change of $\mathrm{Res}_{\cO'/\cO}\bbI^{\mathrm{reg}}$ along the closed embedding $\fg^{\mathrm{reg}}\to\mathrm{Res}_{\cO'/\cO}\bbg^{\mathrm{reg}}$. 
\begin{prop}\label{prop:reg-centralizer}
    Suppose $G$ is tamely ramified and the residue characteristic $p$ is not a torsion prime for the root datum of $G$. We consider the $\cO$-schemes of finite type $\fg^{\mathrm{reg}}\subset\fg$ and $\fc$ defined as above. Then the morphism $\chi_G^{\mathrm{reg}}:\fg^{\mathrm{reg}}\to\fc$ (between $\cO$-schemes of finite type) is smooth and there exists a unique commutative group scheme $J$ over $\fc$, equipped with an isomorphism $(\chi_G^{\mathrm{reg}})^*J\cong I^{\mathrm{reg}}$ that extends uniquely to a homomorphism of group schemes $\chi^*J\to I$ over the $\fg$. 
\end{prop}
The group scheme $J$ in the Proposition is called the \emph{universal regular centralizer for $\fg$ (in $G$)}. In general we need more assumptions on either $p$ or the group $G$ to guarantee its smoothness. 
\begin{proof}
    When $G$ is split the smoothness of $\chi_G^{\mathrm{reg}}$ follows from \cite[Proposition 4.2.6]{BC22} and the other statements follow from \cite[Theorem 4.2.8]{BC22}. In particular all the statements hold for the split model $\bbG$ over $\cO$ and we get the regular centralizer group scheme $\bbJ$ over the $\cO$-scheme $\bbc$ that satisfies all the statements.\par 
    In general, we deduce the smoothness of $\mathrm{Res}_{\cO'/\cO}(\chi_\bbG^{\mathrm{reg}}):\mathrm{Res}_{\cO'/\cO}\bbg^{\mathrm{reg}}\to\mathrm{Res}_{\cO'/\cO}\bbc$ from the smoothness of $\chi_\bbG^{\mathrm{reg}}$ by the infinitesimal lifting criterion. Then we apply Lemma \ref{lem:fixed-point-locus-smooth} to the base change of $\mathrm{Res}_{\cO'/\cO}(\chi_\bbG^{\mathrm{reg}})$ along the closed subscheme $\fc\subset\mathrm{Res}_{\cO'/\cO}\bbc$ to get the smoothness of $\chi_G^{\mathrm{reg}}$. The group scheme $I$ is an open subgroup scheme of $(\mathrm{Res}_{\cO'/\cO}\bbI)^{\mu_e}$ and hence descends to an open sub group scheme $J$ of $(\mathrm{Res}_{\cO'/\cO}\bbJ)^{\mu_e}$. The existence of the isomorphism $(\chi^{\mathrm{reg}})^*J\cong I^{\mathrm{reg}}$ that extends uniquely to a homomorphism $\chi^*J\to I$ follows from the corresponding statements for $\bbJ$.
\end{proof}
\begin{cor}\label{cor:chi-surjection}
    Under the notations and assumptions in Proposition \ref{prop:reg-centralizer}, the Chevalley morphism $\chi_G$ restricts to a surjection $\fg^{\mathrm{reg}}(\cO)\to\fc(\cO)$. More generally, for any parahoric subgroup $\bfP\subset G(F)$, $\chi_G$ restricts to a surjection $\mathrm{Lie}(\bfP)\to\fc(\cO)$.
\end{cor}
\begin{proof}
    The first statement follows from the smoothness of $\chi_G^{\mathrm{reg}}$. For the second statement, after $G(F)$-conjugation we may assume that $\bfP$ contains the standard Iwahori subgroup $\bfI$ and thus it suffices to show surjectivity of $\mathrm{Lie}(\bfI)\to\fc(\cO)$. This follows from the fact that any element in $\fg(\cO)$ is $G(\cO)$-conjugate to an element in $\mathrm{Lie}(\bfI)$, since $G(\cO)$-conjugates of $\bfI$ corresponds bijectively to Borel subgroups of $\sG_0=\sG^{\mu_e,\circ}$. 
\end{proof}
\subsection{Topological Jordan decomposition for the Lie algebra}\label{sec:top-Jordan-Lie}

\begin{defn}\label{def:bounded-Lie-alg}
    An element $\ga\in\fg(\overline{F})=\bbg(\overline{F})$ is 
    \begin{itemize}
        \item \emph{bounded} if $\chi_G(\ga)\in\bbc(\overline{\cO})$;
        \item \emph{topologically nilpotent} if it is bounded and the image of $\chi_G(\ga)$ under the natural reduction map $\bbc(\overline{\cO})\to\bbc(k)$ equals to zero.
        \item \emph{strongly semisimple} if it is semisimple, bounded, and for any root $\alpha\in\Sigma$ and any $\ga'\in\ft(\overline{F})$ that is $G(\overline{F})$-conjugate to $\ga$ we have either $d\alpha(\ga')=0$ or $|d\alpha(\ga')|=1$
    \end{itemize}
\end{defn}
\begin{lem}\label{lem:top-nil}
    Suppose that $G$ is tamely ramified, the residue characteristic $p$ is good for $G$ and divides neither the order of $\pi_1(G^{\mathrm{der}})$ nor the order of $\pi_0(Z_G)$. Then an element $\ga\in\fg(F)$ is topologically nilpotent if and only if there exists a parahoric subgroup $\bfP\subset G(F)$ such that $\ga\in\mathrm{Lie}(\bfP_+)$.
\end{lem}
\begin{proof}
    This follows from results in \cite{AFV18}. Let us explain it using the notations and citations from \emph{loc. cit.} By Remark 16(a) and Remark 8, the assumption on $p$ implies that $p$ is ``$\fn^-$-good" in the sense of Definition 15(a). Then by Remark 16(b) we see that $p$ is ``$\fg$-good" in the sense of Definition 15(c). Moreover, the assumption on $G$ implies that it is ``not too wild" in the sense of Definition 17 (see also the paragraph following Definition 17). Then our result is precisely Remark 35.
\end{proof}
\begin{prop}\label{prop:top-jordan-Lie}
    Suppose that $G$ is tamely ramified, the residue characteristic $p$ is good for $G$ and divides neither the order of $\pi_1(G^{\mathrm{der}})$ nor the order of $\pi_0(Z_G)$. Let $\ga\in\fg(F)$ be a bounded semisimple element. 
    Then there exists a decomposition $\ga=\ga_0+\ga_1$ where $\ga_0,\ga_1\in\fg_\ga(F)$ satisfy the following conditions:
    \begin{itemize}
        \item $\ga_0$ is strongly semisimple in $\fg(F)$ and $G(F)$-conjugate to an element in $\ft(\cO)$;
        \item The centralizer $G_{\ga_0}$ is an $F$-Levi subgroup of $G$ and $\ga_1$ is a topologically nilpotent element in $\mathrm{Lie}(G_{\ga_0})$.
    \end{itemize}
     If moreover $\ga$ is regular semisimple in $\fg(F)$, then $\ga_1$ is regular semisimple in $\mathrm{Lie}(G_{\ga_0})$.
\end{prop}
\begin{proof}
    Let $F'\subset\overline{F}$ be the minimal Galois extension of $F$ over which $G$ becomes split. Let $\Gamma:=\mathrm{Gal}(\overline{F}/F)$, $\Gamma':=\mathrm{Gal}(\overline{F}/F')$ and $\Theta:=\mathrm{Gal}(F'/F)\cong\Gamma/\Gamma'$. Let $\rho_G:\Gamma\to\mathrm{Out}(\bbG)$ be the twisting homomorphism of $G$ as in \S\ref{sec:parahoric}. Then $\rho_G$ is the composition of the projection $\Gamma\epic\Theta$ and an embedding $\Theta\into\mathrm{Out}(\bbG)$. For each $\sigma\in\Gamma$, we let $\sigma_G$ denote the Galois action on $G(\overline{F})$ or $\fg(\overline{F})$ and simply use $\sigma(\cdot)$ to denote the Galois action with respect to the \emph{split form} $\bbG$. Then we have 
    \[\sigma_G(g)=\rho_G(\sigma)\sigma(g),\quad\sigma_G(x)=\rho_G(\sigma)\sigma(x),\quad\forall g\in G(\overline{F}), x\in\fg(\overline{F}).\]
    Let $T_\ga\subset G_\ga$ be a maximal $F$-torus. Then $T_\ga$ is also a maximal $F$-torus of $G$ since $\ga$ is semisimple. Let $g\in G(\overline{F})$ be an element such that $\mathrm{Ad}(g)^{-1}T_{\overline{F}}=T_{\ga,\overline{F}}$ (where $T$ is the maximal $F$-torus of $G$, which is part of the pinning of $G$ fixed in \S\ref{sec:parahoric}). Let $\ga'=\mathrm{ad}(g)\ga$. Since $\ga$ lies in the center of $\fg_\ga$ we have $\ga\in\mathrm{Lie}(T_\ga)$. Moreover since $\ga'$ is bounded we get that $\ga'\in\bbt(\overline{\cO})$. For any $\sigma\in\Gamma$ we have $g\cdot\sigma_G(g)^{-1}\in N_G(T)(\overline{F})$ and we let $w_\sigma$ denote its image in $\bbW$. Then we have
    \begin{equation}\label{eq:gamma'-Galois-action}
    \ga'=w_\sigma\cdot\sigma_G(\ga')=w_\sigma\cdot\rho_G(\sigma)\sigma(\ga'),\quad\quad\forall\sigma\in\Gamma.
    \end{equation}
    Recall that $\bbt$ is the split $\cO$-model of $\ft$. Let $\Bar{\ga}'\in\bbt(k)$ be the image of $\ga'$ under the reduction map $\ft(\overline{\cO})\to\bbt(k)$. Since $\chi_G(\ga)=\chi_G(\ga')\in\fc(\cO)$, we have $\chi_\bbG(\bar{\ga}')\in\bbc^{\Theta}(k)$.\par
    We claim that the natural map $\bbt^{\Theta}(k)\to\bbc^{\Theta}(k)$ is surjective. By the assumption that $G$ is tamely ramified, the order of $\Theta$ is invertible in $\cO$. So $\bbt^\Theta$ and $\bbc^\Theta$ are both smooth $\cO$-schemes by Lemma \ref{lem:fixed-point-locus-smooth}. In particular any element in $\bbc^{\Theta}(k)$ lifts to an element in $\bbc^{\Theta}(\cO)$ by \cite[Lemma 8.1.3]{KP23}. Since the morphism $\bbt^\Theta\to\bbc^\Theta$ is finite (hence proper), it suffices to show surjectivity of the map $\bbt^\Theta(\overline{F})\to\bbc^\Theta(\overline{F})$. Thus we are reduced to characteristic $0$ situation and the claim follows from \cite{Levy-MO}. Since $\bbW$ acts transitively on the fiber $\chi_\bbG^{-1}(\chi_\bbG(\bar{\ga}'))$, after conjugation by $N_G(T)(\overline{F})$ we may and do assume that $\bar{\ga}'\in\bbt^{\Theta}(k)$ so that $\rho_G(\sigma)(\bar{\ga}')=\bar{\ga}'$ for all $\sigma\in\Gamma$.\par 
    Let $\Sigma_1\subset\Sigma$ be the subset consisting of roots $\alpha$ such that $\mathrm{val}(d\alpha(\ga'))>0$. Then $\Sigma_1$ is the closed sub-system of $\Sigma$ consisting of the roots whose root hyperplane contains $\Bar{\ga}'$. Moreover, $\Sigma_1$ is $\Theta$-stable since $\bar{\ga}'$ is $\Theta$-invariant. Let $\bbW_1$ be the Weyl group of the root system $\Sigma_1$ and let $\bbH\subset \bbG$ be the split reductive subgroup containing $\bbT$ and having root system $\Sigma_1$. By assumption $p$ is good for $\bbG$ and does not divide the order of $\pi_1(\bbG^{\mathrm{der}})$, so $p$ is not a torsion prime for the root datum of $\bbG$. This implies that the centralizer $k$-group $\sG_{\bar{\ga}'}=\bbG_{k,\bar{\ga}'}$ is connected by Proposition \ref{prop:centralizer-Levi-connected}. Then by \cite[Lemma 3.7]{St-torsion} we get that $\bbW_1=\bbW_{\bar{\ga}'}$, the stabilizer of $\bar{\ga}'$ in $\bbW$. From equation \eqref{eq:gamma'-Galois-action} we deduce
    \[\bar{\ga}'=w_\sigma\cdot\rho_G(\sigma)(\bar{\ga}')=w_\sigma\cdot\bar{\ga}',\quad\forall\sigma\in\Gamma.\]
    Therefore the assignment $\sigma\mapsto w_\sigma$ defines a homomorphism
    \[\rho_{\bar{\ga}'}:\Gamma\to\bbW_{\bar{\ga}'}=\bbW_1.\]
    Let $\bbt_1\subset\bbt$ be the vanishing loci of $\{d\alpha,\alpha\in\Sigma_1\}$ defined by
    \[\bbt_1:=\spec\frac{\cO[\bbt]}{\langle d\alpha;\alpha\in\Sigma_1\rangle}=\spec\frac{\mathrm{Sym}_\cO(X^*(T)\otimes_\bbZ\cO)}{\langle \alpha,\alpha\in\Sigma_1\rangle}.\]
    Then we have $\Omega^1_{\bbt_1/\cO}=(X^*(T)/\bbZ\Sigma_1)\otimes_\bbZ\cO$. Our assumptions on $p$ imply that  $X^*(T)/\bbZ\Sigma_1$ does not have $p$-torsion. Indeed, $\bbZ\Sigma/\bbZ\Sigma_1$ is $p$-torsion free since $p$ is good for $G$; while the order of the torsion subgroup of $(X^*(T)/\bbZ\Sigma)$ equals to the order of $\pi_0(Z_\bbG)$, which is prime to $p$ by assumption. So $\Omega^1_{\bbt_1/\cO}$ is a free $\cO$-module of finite rank and hence $\bbt_1$ is a smooth $\cO$-scheme. Moreover $\bbt_1$ is preserved by the $\Theta$-action and we have $\bbt_1\subset\bbt^{\bbW_1}$. The fixed point locus $\bbt_1^{\Theta}$ is a smooth $\cO$-scheme by Lemma \ref{lem:fixed-point-locus-smooth}. By construction we have $\bar{\ga}'\in\bbt_1^{\Theta}(k)$ and then by \cite[Lemma 8.1.3]{KP23} there exists an element 
    \[\ga_0'\in\bbt_1^{\Theta}(\cO)\subset\ft(\cO)\] 
    whose image under the natural reduction map $\bbt_1(\cO)\to\bbt_1(k)$ equals to $\Bar{\ga}'$. Let $\ga_1':=\ga'-\ga_0'$. Then for any $\sigma\in\Gamma$ we have $w_\sigma\cdot\sigma_G(\ga_0')=\ga_0'$ and by \eqref{eq:gamma'-Galois-action} we get
    \[w_\sigma\cdot\sigma_G(\ga_1')=\ga_1'.\]
    Let $\ga_0:=\mathrm{ad}(g)^{-1}(\ga_0')$ and $\ga_1:=\mathrm{ad}(g)^{-1}(\ga_1')$. Then we have $\ga=\ga_0+\ga_1$ and the above equations imply that $\ga_0,\ga_1\in\fg_\ga(F)$. \par
    For any root $\alpha\in\Sigma_1$ we have $d\alpha(\ga_0')=0$ and for any $\alpha\in\Sigma\backslash\Sigma_1$ we have $d\alpha(\bar{\ga}')\in k^\times$ so that $|d\alpha(\ga_0')|=1$. This implies that $\ga_0'\in\ft(\cO)\subset\fg(F)$ is strongly semisimple and hence $\ga_0\in\fg(F)$ is strongly semisimple. On the other hand since the image of $\ga_1'$ under the reduction map $\bbt(\overline{\cO})\to\bbt(k)$ equals to $0$, it is a topologically nilpotent element in $\fg_{\ga_0'}(\overline{F})$. So $\ga_1$ is a topologically nilpotent element in $\fg_{\ga_0}(F)=\mathrm{Lie}(G_{\ga_0})$. \par
    By our construction $\ga_0$ is $G(\overline{F})$ conjugate to the element $\ga_0'\in\ft(\cO)$. Let $G_{\ga_0,\ga_0'}$ be the $F$-scheme classifying elements $g\in G$ such that $\mathrm{ad}(g)(\ga_0)=\ga_0'$. Then $G_{\ga_0,\ga_0'}$ is nonempty, hence a $G_{\ga_0}$-torsor over $F$. Since $G_{\ga_0}$ is connected reductive by Proposition \ref{prop:centralizer-Levi-connected}, we have $H^1(F,G_{\ga_0})=0$ by Steinberg's theorem and the fact that $F$ has cohomological dimension $\le1$. Thus $G_{\ga_0,\ga_0'}$ is a trivial $G_{\ga_0}$-torsor, which means that $\ga_0$ is $G(F)$-conjugate to $\ga_0'\in\ft(\cO)$. By Proposition \ref{prop:centralizer-Levi-connected},  $G_{\ga_0',\overline{F}}=\bbH_{\overline{F}}$ is a Levi subgroup of $G_{\overline{F}}$. Also, $G_{\ga_0'}$ contains the maximally split maximal $F$-torus $T$ of $G$ since $\ga_0'\in\ft(F)$. Therefore $G_{\ga_0'}$, and hence $G_{\ga_0}$, is an $F$-Levi subgroup of $G$.\par
    Finally suppose that $\ga$ is regular semisimple. Then $\ga'$ is also regular semisimple and hence $d\alpha(\ga_1')=d\alpha(\ga')\ne0$ for all $\alpha\in\Sigma_1$. So $\ga_1$ is regular semisimple in $\mathrm{Lie}(G_{\ga_0})$.
\end{proof}

\subsection{Invariant pairing and quasi-logarithm}\label{sec:pairing-log}
We review relevant facts from \cite[\S1.8]{KV} and generalize some results there to the case of tamely ramified groups. Let $G$ be an algebraic group over $F$ and let $\rho:G\to\mathrm{GL}(V)$ be an algebraic representation on a finite dimensional $F$-vector space $V$. Let $d\rho:\mathrm{Lie}(G)\to\mathrm{End}(V)$ be the induced representation of the Lie algebra. On $\mathrm{End}(V)$ we have the trace pairing defined by $(A,B):=\mathrm{Tr}(AB)$. This induces a $G_{\mathrm{ad}}$-invariant symmetric bilinear form on $\fg$ defined by $(X,Y)_\rho:=\mathrm{Tr}(d\rho(X)\circ d\rho(Y))$. Suppose that $(\cdot,\cdot)_\rho$ is nondegenerate. Then we can define the orthogonal projection $\mathrm{pr}_\fg:\mathrm{End}(V)\to\fg$ such that $\mathrm{pr}_\fg(A)\in\fg$ is the unique element satisfying $(\mathrm{pr}_\fg(A),X)_\rho=\mathrm{Tr}(A\circ d\rho(X))$ for all $X\in\fg$. Define a morphism $\Phi_\rho:G\to\fg$ by $\Phi_\rho(g):=\mathrm{pr}_\fg(\rho(g)-\mathrm{Id})$. Then $\Phi_\rho$ is a quasi-logarithm map in the sense of \cite{KV} whose definition is as follows:

\begin{defn}
    A quasi-logarithm map for an algebraic group $G$ with Lie algebra $\fg$ is a $G_\mathrm{ad}$-equivariant morphism $\Phi:G\to\fg$ such that $\Phi(1)=0$ and the differential $d\Phi_1$ of $\Phi$ at $1\in G$ is the identity map. Here the adjoint group $G_\mathrm{ad}$ acts on $G$ by conjugation and on $\fg$ by the adjoint representation. 
\end{defn}

\begin{lem-def}\label{def:qlog-over-O}
    Let $G$ be a reductive group over $F$ with Lie algebra $\fg$ and let $\Phi:G\to\fg$ be quasi-logarithm map. We say that $\Phi$ is \emph{defined over $\cO$} if the following equivalent conditions are satisfied:
    \begin{enumerate}
        \item For any parahoric subgroup $\bfP\subset G(F)$, $\Phi$ extends to a morphism of $\cO$-schemes $\cG_\bfP\to\mathrm{Lie}(\bfP)$;
        \item There exists a parahoric subgroup $\bfP\subset G(F)$ such that $\Phi$ extends to a morphism of $\cO$-schemes $\cG_\bfP\to\mathrm{Lie}(\bfP)$.
    \end{enumerate}
\end{lem-def}
\begin{proof}
    This is part of \cite[Lemma 1.8.7(b)]{KV}. Note that in \emph{loc. cit.} it is assumed that $G$ is split, but this assumption is not necessary for the current result. Let us review the arguments for (2)$\Rightarrow$(1).\par 
    Suppose that $\Phi$ extends to a morphism of $\cO$-group schemes $\cG_\bfP\to\mathrm{Lie}(\bfP)$ for a parahoric subgroup $\bfP\subset G(F)$. Let $\bfI\subset G(F)$ be an Iwahori subgroup contained in $\bfP$. Then $\Phi$ induces a quasi-logarithm map on the special fiber $\overline{\cG}_\bfP$ and the image of $\bfI$ in the special fiber $\overline{\cG}_\bfP$ is a Borel subgroup. Then by \emph{loc.cit.} Lemma 1.8.3 (a) (which holds also for non-reductive groups), we deduce that $\Phi(\bfI)\subset\mathrm{Lie}(\bfI)$. For any parahoric subgroup $\bfP'$ containing $\bfI$ we have $\bfP'=\cup_{x\in\bfP'}x\bfI x^{-1}$ and $\mathrm{Lie}(\bfP')=\cup_{x\in\bfP'}\mathrm{ad}(x)\bfP'$. By the $G$-equivariance of $\Phi$, we get that $\Phi(\bfP')\subset\mathrm{Lie}(\bfP')$ and that this holds for any parahoric subgroup $\bfP'$. Thus $\Phi$ extends to a morphism of $\cO$-schemes $\cG_{\bfP'}\to\mathrm{Lie}(\cG_{\bfP'})$ for any parahoric subgroup $\bfP'\subset G(F)$ by \cite[Proposition 1.7.6]{BT-II}.
\end{proof}
We are interested in extending the pairing $(\cdot,\cdot)_\rho$ and the map $\Phi_\rho$ to the parahoric group schemes of $G$. First we consider the case where $G$ is split.  
\begin{lem-def}\label{lem:quasi-log}
    Suppose $G$ is a split reductive group over $F$. Let $\rho:G\to\mathrm{GL}(V)$ be an algebraic representation such that the associated pairing $(\cdot,\cdot)_\rho$ on $\fg(F)$ defined as above is non-degenerate. We say that $(\cdot,\cdot)_\rho$ is \emph{non-degenerate over $\cO$} if the following equivalent conditions are satisfied:
    \begin{enumerate}
        \item for any parahoric subgroup $\bfP\subset G(F)$ we have 
        \[\mathrm{Lie}(\bfP_+)=\mathrm{Lie}(\bfP)^\perp:=\{Y\in\fg(F),(Y,\mathrm{Lie}(\bfP))_\rho\subset\varpi\cO\}\]
        \item Then there exists a parahoric subgroup $\bfP\subset G(F)$ such that $\mathrm{Lie}(\bfP_+)=\mathrm{Lie}(\bfP)^\perp$.
    \end{enumerate}
    In this case, the associated quasi-logarithm map $\Phi_\rho$ is defined over $\cO$. 
\end{lem-def}
This is \cite[Lemma 1.8.7(a)]{KV} and \cite[1.8.9]{KV}. 

\begin{lem}\label{lem:pairing-split-case}
    Let $G$ be a reductive group over $F$ whose adjoint group $G_{\mathrm{ad}}$ is simple and split. Suppose that the residue characteristic $p$ is good for $G$. Then there exists a split reductive group $G^\natural$ over $\cO$ equipped with an isomorphism between adjoint groups $G_{\mathrm{ad}}\cong G^\natural_{\mathrm{ad}}$, and a faithful algebraic representation $\rho:G^\natural\to\mathrm{GL}(V)$ (defined over $\cO$) satisfying the following conditions:
    \begin{enumerate}
        \item The trace pairing $(\cdot,\cdot)_\rho$ defined by $\rho$ is a perfect pairing on the $\cO$-module $\fg^\natural=\mathrm{Lie}(G^\natural)$.
        \item The natural homomorphism $\mathrm{Out}(G^\natural)\to\mathrm{Out}(G_{\mathrm{ad}})$ between outer automorphism groups is an isomorphism\footnote{Note that $\mathrm{Out}(G_{\mathrm{ad}})=\mathrm{Out}(G^{\mathrm{sc}})$ is isomorphic to the automorphism group of the Dynkin diagram of $G$. On the other hand, $\mathrm{Out}(G)$ is isomorphic to the automorphism group of the \emph{based root datum} of $G$, which maps naturally to the automorphism group of its Dynkin diagram, but the map may not be bijective. For example, the outer automorphism group of the split $\mathrm{SO}(8)$ is $\bbZ/2\bbZ$; while the outer automorphism group of its simply connected cover $\mathrm{Spin}(8)$ or its adjoint group $\mathrm{PSO}(8)$ is $S_3$, the symmetric group on $3$ letters.} and the trace pairing $(\cdot,\cdot)_\rho$ is $\mathrm{Out}(G_{\mathrm{ad}})$-invariant. In other words, for any $x,y\in\fg^\natural(F)$ and any $\theta\in\mathrm{Out}(G_{\mathrm{ad}})$, we have $(x,y)_\rho=(\theta(x),\theta(y))_\rho$.
        \item The residue characteristic $p$ divides neither the order of $\pi_1(G^{\natural,\mathrm{der}})$ \footnote{together with the assumption that $p$ is good for $G$, this implies that $p$ is not a torsion prime for the root datum of $G^\natural$.} nor the order of the component group of the center of $G^\natural$.
    \end{enumerate}
\end{lem}
\begin{proof}
    We take $G^\natural$ to be $\mathrm{GL}_n$ (resp. $\mathrm{SO}_{2n+1}$, $\mathrm{Sp}_{2n}$, $\mathrm{SO}_{2n}$) if $G$ is of type $\mathrm{A}_{n-1}$ for $n\ge1$ (resp. $\mathrm{B}_{n}$ for $n\ge2$, $\mathrm{C}_n$ for $n\ge3$, $\mathrm{D}_n$ for $n\ge4$) and let $\rho$ be the standard representation of $G^\natural$. If $G$ is exceptional, then we take $G^\natural=G_{\mathrm{ad}}$ and let $\rho$ be the adjoint representation of $G^\natural$. It is clear from the definition that $\mathrm{Out}(G^\natural)\cong\mathrm{Out}(G_{\mathrm{ad}})$ and that the representation $\rho$ naturally extends to the hyperspecial $\cO$-model of $G^\natural$ (that we still denote by $G^\natural$). By the proof of \cite[I,Lemma 5.3]{SS-conj}, the assumption that $p$ is good implies that the induced trace pairing $(\cdot,\cdot)_\rho$ remains perfect on the special fiber of the hyperspecial $\cO$-model. Therefore it is non-degenerate over the hyperspecial point and hence non-degenerate over $\cO$ by Lemma \ref{lem:quasi-log}.\par 
    Next we show that $(\cdot,\cdot)_\rho$ is $\mathrm{Out}(G_\mathrm{ad})$-invariant. If $G$ is exceptional, then $(\cdot,\cdot)_\rho$ is the Killing form and is easily checked to be invariant under any automorphism of the Lie algebra. If $G$ is of type $\mathrm{A}_{n-1}$, then any nontrivial outer automorphism of $\fg^\natural=\mathrm{Mat}_n$ is conjugate to $x\mapsto-{}^tx$, which clearly fixes the usual trace pairing. If $G$ is of type B or C, then $\mathrm{Out}(G_\mathrm{ad})$ is trivial and there is nothing to prove. If $G$ is of type $\mathrm{D}_n$ for $n\ge5$, then any nontrivial outer automorphism of $\fg^\natural=\mathfrak{so}(2n)$ has order $2$ and is induced by conjugation by an element in $\mathrm{O}_{2n}$ (which contains $G^\natural=\mathrm{SO}_{2n}$ as a subgroup of index $2$, see for example \cite[\S17.3.2]{Spr-LAG}), hence it fixes the trace pairing $(\cdot,\cdot)_\rho$ for the standard representation $\rho$ of $G^\natural=\mathrm{SO}_{2n}$. Thus we are left with the case where $G$ is of type $\mathrm{D}_4$. In this case $\mathrm{Out}(G_\mathrm{ad})=S_3$ and the action of a generator of order $2$ on $\fg^\natural=\mathfrak{so}(8)$ fixes the trace pairing for the same reason as in the previous case. Finally we refer to \cite[Appendix]{MW} for an explicit description of the action of an order 3 outer automorphism on $\fg^\natural=\mathfrak{so}(8)$, from which we easily check that it fixes the standard trace pairing\footnote{Actually the invariance of the standard trace pairing on $\mathfrak{so}(8)$ under the order 3 outer automorphism will not be used in this article}.
\end{proof}
Next we generalize the statements on regular centralizers and Kostant sections to tamely ramified groups. Following the notations in \S\ref{sec:group-model}, we consider a connected reductive group $G$ over $F$ that splits over a degree $e$ tamely ramified extension $F'/F$ where $p\nmid e$ and we assume that $G_{\mathrm{ad}}$ is absolutely simple and $p$ is good for $G$. Let $\bbG$ be the split form of $G$ and let $\rho_G:\mu_e\to\mathrm{Out}(\bbG)$ be the twisting data of $G$. Then we have the special integral model of $G$ for which we use the same notation $G:=(\mathrm{Res}_{\cO'/\cO}\bbG)^{\mu_e,\circ}$. Let $\bbG^\natural$ be the split reductive group over $\cO$ associated to $G^\natural$ from Lemma \ref{lem:pairing-split-case}. Let $\rho_{G_{\mathrm{ad}}}:\mu_e\to\mathrm{Out}(\bbG_{\mathrm{ad}})$ be the composition of $\rho_G$ with the natural homomorphism $\mathrm{Out}(\bbG)\to\mathrm{Out}(\bbG_{\mathrm{ad}})$. By condition (2) in Lemma \ref{lem:pairing-split-case} we can lift $\rho_{G_{\mathrm{ad}}}$ to a homomorphism $\mu_e\to\mathrm{Out}(\bbG^\natural)$, from which we obtain an outer twist $G^\natural$ of $\bbG^\natural$ that comes with an isomorphism of $F$-groups $G^\natural_{\mathrm{ad}}\cong G_{\mathrm{ad}}$. Moreover, we also get a special $\cO$-model $G^\natural:=(\mathrm{Res}_{\cO'/\cO}\bbG^\natural)^{\mu_e,\circ}$. Let $\fg^\natural=\mathrm{Lie}(G^\natural)=(\mathrm{Res}_{\cO'/\cO}\bbg^\natural)^{\mu_e}$ be the Lie algebra and let $\fc^\natural=(\mathrm{Res}_{\cO'/\cO}\bbc^\natural)^{\mu_e}$ be the Chevalley base for $G^\natural$, both defined over $\cO$. We have the Chevalley homomorphism $\chi_{G^\natural}:\fg^\natural\to\fc^\natural$ and universal centralizer group scheme $I^\natural$ over $\fg^\natural$. By condition (3) of Lemma \ref{lem:pairing-split-case}, the residue characteristic $p$ is not a torsion prime for the root datum of $G$. Then by Proposition \ref{prop:reg-centralizer}, we have the universal regular centralizer $J^\natural$ over $\fc^\natural$, together with an isomorphism $(\chi_{G^\natural}^{\mathrm{reg}})^*J^\natural\cong I^\natural|_{\fg^{\natural,\mathrm{reg}}}$. 

\begin{cor}\label{cor:Kostant-section}
    Let notations and assumptions be as above. In particular, $G_\mathrm{ad}$ is absolutely simple and tamely ramified, and $p$ is good for $G$. Then the universal regular centralizer $J^\natural$ (for $\fg^\natural$ in $G^\natural$) is smooth over the $\cO$-scheme $\fc^\natural$. Moreover, there exists a section $\kappa^\natural:\fc^\natural\to\fg^{\natural,\mathrm{reg}}$ of $\chi_{G^\natural}^{\mathrm{reg}}$ and an isomorphism of $\fc^\natural$-group schemes $(\kappa^\natural)^*I^\natural\cong J^\natural$, both defined over $\cO$.
\end{cor}
\begin{proof}
    Let $\bbJ^\natural$ be the universal regular centralizer of $\bbc^\natural$ defined for the split form $\bbG^\natural$ over $\cO$. 
    From the proof of Lemma \ref{lem:pairing-split-case}, we see that if $\bbG$ is of type $\mathrm{A}_{n-1}$ then $\bbG^\natural=\mathrm{GL}_n$ and $\bbJ^\natural$ can be written as a Weil restriction of $\bbG_m$ from a finite flat cover of $\bbc^\natural$ (defined by the universal characteristic polynomial) and hence smooth. If $\bbG$ is of other types, then the assumption on $p$ implies that $p$ does not divide the order of the algebraic fundamental group of $\bbG^\natural_{\mathrm{ad}}\cong \bbG_{\mathrm{ad}}$ and hence $\bbJ^\natural$ is smooth over $\bbc^\natural$ by \cite[Proposition 4.2.11]{BC22}. Then $J^\natural$ is also smooth since it is an open subgroup scheme of $(\mathrm{Res}_{\cO'/\cO}\bbJ^\natural)^{\mu_e}$ and the latter is smooth by Lemma \ref{lem:res} and Lemma \ref{lem:fixed-point-locus-smooth}.\par
    It remains to prove the existence of the section $\kappa^\natural$ (the isomorphism $(\kappa^\natural)^*I^\natural\cong J^\natural$ will then follow immediately). Recall that we have fixed a pinning $(\bbB,\bbT,(x_\alpha\in\bbg_\alpha(\cO))_{\alpha\in\Delta})$ of $\bbG$. This uniquely determines a pinning $(\bbB^\natural,\bbT^\natural,(x_\alpha\in\bbg^\natural_\alpha(\cO))_{\alpha\in\Delta})$ of $\bbG^\natural$ since we have a canonical identification of root spaces $\bbg_\alpha=\bbg^\natural_\alpha$. Then the action of $\mathrm{Out}(\bbG_\mathrm{ad})$ fixes $x_+:=\sum_{\alpha\in\Delta}x_\alpha$ and hence preserves the $\cO$-submodule $\mathfrak{s}^\natural:=\ker(\mathrm{ad}(x_+)\mid\bbg^\natural(\cO))$. By Lemma \ref{lem:pairing-split-case} there is a perfect pairing $(\cdot,\cdot)_\rho$ on $\bbg^\natural(\cO)$ that is invariant under $\bbG_{\mathrm{ad}}\rtimes\mathrm{Out}(\bbG_{\mathrm{ad}})$. 
    Under this pairing, $\mathfrak{s}^\natural$ is the orthogonal complement of $[x_+,\bbg^\natural(\cO)]$ in $\bbg^\natural(\cO)$ and hence 
    $x_++\mathfrak{s}^\natural$ is a section of $\chi_{\bbG^\natural}$ (see \cite[\S4.2.4, Remark 4.2.5]{BC22}). Since $x_++\mathfrak{s}^\natural$ is stable under the $\mathrm{Out}(G_\mathrm{ad})$-action, it induces a section on the twisted form $\fc^\natural\to\fg^\natural$ and we are done. 
\end{proof}
\begin{lem-def}\label{def:qlog-top-nilp}
    Let $G$ be a reductive group over $F$ with Lie algebra $\fg$ and let $\Phi:G\to\fg$ be a quasi-logarithm map. We say that $\Phi$ is \emph{bijective over topologically nilpotent elements} if it is defined over $\cO$ (in the sense of Definition \ref{def:qlog-over-O}) and satisfies any of the following equivalent conditions:
    \begin{enumerate}
        \item For any parahoric subgroup $\bfP\subset G(F)$, $\Phi$ restricts to a bijection $\bfP_{tu}\xrightarrow{\sim}\mathrm{Lie}(\bfP)_{tn}$, where $\bfP_{tu}$ (resp. $\mathrm{Lie}(\bfP)_{tn}$) is the set of topologically unipotent (resp. topologically nilpotent) elements in $\bfP$ (resp. $\mathrm{Lie}(\bfP)$)
        \item There exists a parahoric subgroup $\bfP\subset G(F)$ such that $\Phi$ restricts to a bijection $\bfP_{tu}\xrightarrow{\sim}\mathrm{Lie}(\bfP)_{tn}$.
    \end{enumerate}
    Moreover, when these conditions are satisfied, $\Phi$ restricts to a bijection $G(F)_{stu}\xrightarrow{\sim}\fg(F)_{tn}$ between the set of  \emph{strongly} topologically unipotent elements in $G(F)$ and the set of topologically nilpotent elements in $\fg(F)$. 
\end{lem-def}
\begin{proof}
    This follows from arguments in \cite[proof of Proposition 1.8.16]{KV}. In \emph{loc. cit.} it is assumed that $G$ is split but this assumption is not necessary for the current result. Let us review the arguments for (2)$\Rightarrow$(1). Note that by Lemma \ref{lem:top-nil}, our notion of ``topological nilpotent" agrees with the one used in \cite{KV}.\par 
    Suppose $\Phi$ restricts to a bijection $\bfP_{tu}\xrightarrow{\sim}\mathrm{Lie}(\bfP)_{tn}$. Let $\bfI\subset G(F)$ be an Iwahori subgroup contained in $\bfP$. Then $\bfI_{tu}=\bfP_{tu}\cap\bfI$ and $\mathrm{Lie}(\bfI)_{tn}=\mathrm{Lie}(\bfP)_{tn}\cap\mathrm{Lie}(\bfI)$ by \emph{loc. cit.} Lemma 1.8.15. Thus $\Phi$ restricts to a bijection $\bfI_{tu}\xrightarrow{\sim}\mathrm{Lie}(\bfI)_{tn}$. To prove statement (1) for any parahoric subgroup $\bfP'$, by $G(F)$-equivariance we may assume that $\bfI\subset\bfP'$. Then by \emph{loc.cit.} Lemma 1.8.15 we have $\bfP'_{tu}=\cup_{g\in\bfP'}g\bfI_{tu}g^{-1}$ and $\mathrm{Lie}(\bfP')_{tn}=\cup_{g\in\bfP'}\mathrm{ad}(g)\mathrm{Lie}(\bfI)_{tn}$, so $\Phi$ restricts to a bijection $\bfP'_{tu}\xrightarrow{\sim}\mathrm{Lie}(\bfP')_{tn}$. 
\end{proof}
\begin{lem}\label{lem:qlog-bij-top-nilp}
    Let $G$ be a tamely ramified reductive group over $F$ obtained as an outer twist of a split reductive group $\bbG$. Suppose there is an algebraic representation $\rho$ of $\bbG$ defined over $\cO$ such that the associated trace form $(\cdot,\cdot)_\rho$ on the Lie algebra $\bbg=\mathrm{Lie}(\bbG)$ is non-degenerate over $\cO$. Then there exists a quasi-logarithm map $\Phi:G\to\fg=\mathrm{Lie}(G)$ that is bijective over topologically nilpotent elements. 
\end{lem}
When $G$ is split, this is proved in \cite[Proposition 1.8.16]{KV}. As far as we can see, the arguments in \emph{loc.cit.} do not immediately generalize to the current setting. However, see \cite[Appendix C]{BKV-r-proj} for closely related results. 
\begin{proof}
    Let $F'/F$ be the minimal tamely ramified cyclic extension such that $G_{F'}$ is split and let $\cO'$ be the valuation ring of $F'$. Fix a uniformizer $\varpi'\in\cO'$ such that $(\varpi')^e=\varpi$ is a uniformizer of $\cO$, where $e=[F':F]$. Then there is an outer automorphism $\sigma$ of $\bbG$ that generates an action of the cyclic group $\mathrm{Gal}(F'/F)\cong\mu_e$ on $\bbG$, from which we get the twisted form $G$. We extend $G$ to a smooth group scheme over $\cO$ by taking the fiberwise identity component of $(\mathrm{Res}_{\cO'/\cO}\bbG)^{\mu_e}$. Then the Lie algebra of $G$ is $\fg=\mathrm{Res}_{\cO'/\cO}\bbg$ and $G(\cO)$ is a special parahoric subgroup of $G(F)$. \par 
    For each $i\ge1$ we denote $k[t]_i:=k[t]/t^i$ and in particular $\cO'/(\varpi')^e\cO'\cong k[t]_e$. 
    Let $\sG=G_k=(\mathrm{Res}_{k[t]_e/k}\bbG)^{\sigma,\circ}$ be the special fiber of $G$. Define a sequence of normal subgroups $\sU_0=\sU\supset\sU_1\supset\dotsm\supset\sU_{e-1}\supset\sU_e=1$ in $\sG$ by the congruence condition 
    \[\sU_i:=\ker(\sG\to\mathrm{Res}_{k[t]_{i+1}/k}\bbG),\quad\forall 0\le i\le e\]
    and let $\sG_i:=\sG/\sU_i$. Then $\sU:=\sU_0$ is the unipotent radical of $\sG$ and $\sG_0=\sG/\sU$ is the reductive quotient. There are natural isomorphisms $\sG_0\cong(\bbG_k)^{\sigma,\circ}$ and $\sG\cong\sU\rtimes\sG_0$.\par      
    The action of $\sigma$ on the Lie algebra $\bbg=\mathrm{Lie}(\bbG)$ leads to a decomposition of $\cO$-modules $\bbg=\oplus_{i=0}^{e-1}\bbg_i$ where $\bbg_i$ is the submodule on which $\sigma$ acts by the scalar $\zeta^i$; here $\zeta\in\cO$ is a primitive $e$-th root of unity. Therefore $\fg(\cO)=\bbg_0(\cO)\oplus\bigoplus_{i=1}^{e-1}(\varpi')^{i}\bbg_{e-i}(\cO)$. Moreover we have $\mathrm{Lie}(\sG_0)=\bbg_0(k)$ and 
    \[\mathrm{Lie}(\sG_i)=\bbg_0(k)\oplus\bigoplus_{j=1}^i\bbg_{e-j}(k)t^j,\quad\forall 1\le i\le e.\]
    Let $\rho:\bbG\to\mathrm{GL}(\bbV)$ be the given representation on a finite rank free $\cO$-module $\bbV$. Let $\Phi_\rho:\bbG\to\bbg$ be the associated quasi-logarithm map, which is a morphism of $\cO$-group schemes. Define a morphism between $\cO$-schemes $\Phi:G\to\fg$ to be the following composition
    \[\Phi:G\into\mathrm{Res}_{\cO'/\cO}\bbG\xrightarrow{\mathrm{Res}_{\cO'/\cO}\Phi_{\rho,\cO'}}\mathrm{Res}_{\cO'/\cO}\bbg\xrightarrow{p_\fg}\fg\]
    where $p_\fg=\frac{1}{e}\sum_{0\le i\le e-1}\sigma^i$ is the averaging operator on $\bbg(\cO')$, with $\sigma$ acting on both $\bbg$ and the coefficient ring $\cO'$. From the fact that $\Phi_\rho$ is a quasi-logarithm map we deduce that $\Phi$ is a quasi-logarithm map. For any $g\in\sG$ and $u\in\sU_i$ we have 
    \[\Phi_\rho(gu)-\Phi_\rho(g)=\mathrm{pr}_\bbg(\rho(gu)-\rho(g))\in t^{i+1}\bbg(k[t]_e)\]
    and after applying $p_\fg$ we get that $\Phi(gu)-\Phi(g)\in\mathrm{Lie}(\sU_i)$. Therefore $\Phi$ induces quasi-logarithm maps
    \[\Phi_i:\sG_i=\sG/\sU_i\to\mathrm{Lie}(\sG_i),\quad\forall 0\le i\le e.\]
    Let $\cU_i\subset\sG_i$ be the variety of unipotent elements and let $\cN_i\subset\mathrm{Lie}(\sG_i)$ be the variety of nilpotent elements. We claim that for any $0\le i\le e$, $\Phi_i$ restricts to an isomorphism $\cU_i\xrightarrow{\sim}\cN_i$. When $i=0$, $\sG_0$ is reductive and this follows from the version of Luna's \'etale slice theorem valid in all characteristic as in \cite[Theorem 6.2]{BR-Luna} (cf. \cite[proof of Corollary 9.3.4]{BR-Luna} and \cite[proof of Proposition 1.8.16]{KV}). Suppose now that for some $0\le i\le e-1$, the morphism $\Phi_i$ induces an isomorphism $\cU_i\xrightarrow{\sim}\cN_i$. We note that
    $\sU_i/\sU_{i+1}=\bbg_{e-i-1, k}$ for any $0\le i\le e-1$ and that an element in $\sG_{i+1}$ is unipotent if and only if its image in $\sG_i$ is unipotent. Therefore $\cU_{i+1}$ is a $\bbg_{e-i-1, k}$-torsor over $\cU_i$ and similarly $\cN_{i+1}$ is a $\bbg_{e-i-1, k}$-torsor over $\cN_i$. Fix an element $\ga\in\cU_{i+1}$ and we examine the restriction of $\Phi_{i+1}$ to the fiber of $\cU_{i+1}\to\cU_{i}$ through $\ga$. For any $x\in\bbg_{e-i-1}(k)$ we have 
    \begin{equation}\label{eq:quasi-log}
        \begin{split}
            \Phi_{i+1}(x\ga)&=p_\fg\circ\mathrm{pr}_{\bbg}(\rho(x)\rho(\ga)-1)\\
            &=p_\fg\circ\mathrm{pr}_{\bbg}((\rho(x)-1)\rho(\ga))+p_\fg\mathrm{pr}_{\bbg}(\rho(\ga)-1)\\
            &=p_\fg\circ\mathrm{pr}_{\bbg}(d\rho(x)\circ\rho(\ga))+\Bar{\Phi}(\ga)
        \end{split}
    \end{equation}
    where in the last equality we view $x$ as an element in the Lie algebra $\bbg(k)$ and $d\rho:\bbg\to\mathrm{End}(\bbV)$ is the differential of $\rho$. 
    Therefore the restriction of $\Phi_{i+1}$ to each fiber of $\cU_{i+1}\to\cU_{i}$ is given by a linear endomorphism on $\bbg_{e-i-1,k}$ and hence $\Phi_{i+1}$ is determined by a morphism of $k$-schemes $f_i:\cU_i\to\mathrm{End}(\bbg_{e-i-1, k})$ such that 
    \[f_i(\ga)(x)=p_\fg\circ\mathrm{pr}_{\bbg}(d\rho(x)\circ\rho(\ga)).\]
    Since $f_i(1)=\mathrm{Id}$ is the identity, there is a nonempty open subset $\Omega_i\subset\cU_i$ that maps into the open subscheme $\mathrm{GL}(\bbg_{e-i-1, k})$ of invertible linear maps. If two elements $\ga,\ga'\in\cU_{i}$ have the same image in $\cU_0\subset\sG_0$, then $\ga\equiv\ga'(\mod t)$ and we see from the definition that $f_i(\ga)=f_i(\ga')$. Thus $\Omega_i$ is the inverse image of a nonempty open subset $\Omega_0\subset\cU_0$ under the natural morphism $\cU_i\to\cU_0$. By the $\sG_0$-equivariance of $\Phi_i$ (note that $\sG_0$ is a subgroup of $\sG_i$), we see that $\Omega_0$ is an open $\sG_0$-stable subset that contains the identity. Since the identity is contained in the closure of any $\sG_0$-orbits on $\cU_0$, we must have $\Omega_0=\cU_0$ and therefore $f_i(\cU_i)\subset\mathrm{GL}(\bbg_{e-i-1, k})$. This implies that $\Phi_{i+1}$ restricts to an isomorphism of $k$-schemes $\cU_{i+1}\xrightarrow{\sim}\cN_{i+1}$ and we are done by induction.\par 
    In particular, when $i=e$ we have shown that the special fiber $\Bar{\Phi}$ of $\Phi$ restricts to an isomorphism $\cU\xrightarrow{\sim}\cN$ between the unipotent variety $\cU\subset\sG$ and the nilpotent variety $\cN\subset\mathrm{Lie}(\sG)$.  As a consequence $\Bar{\Phi}$ restricts to an \'etale morphism on an affine open neighborhood of $\cU$ in $\sG$. Then by the arguments in the proof of \cite[Proposition 1.8.16]{KV}, we deduce that $\Phi$ restricts to a bijection $G(\cO)_{tu}\xrightarrow{\sim}\fg(\cO)_{tn}$ and hence $\Phi$ is bijective over topologically nilpotent elements by Lemma \ref{def:qlog-top-nilp}. 
\end{proof}
\begin{prop}\label{prop:pairing}
    Suppose that the residue characteristic $p$ is good for $G$ and that $G_{\mathrm{ad}}$ is tamely ramified over $F$. Then there exists a reductive group $G^\natural$ over $F$, whose adjoint group is isomorphic to $G_{\mathrm{ad}}$, that satisfies the following conditions:
    \begin{enumerate}
        \item The residue characteristic $p$ divides neither the order of $\pi_1(G^{\natural,\mathrm{der}})$ nor the order of the component group of the center of $G^\natural$.
        \item The natural homomorphism $G^\natural(F)\to G^\natural_{\mathrm{ad}}(F)\cong G_{\mathrm{ad}}(F)$ restricts to a surjection on the set of \emph{strongly} topologically unipotent elements.
        \item There exists a quasi-logarithm map $\Phi:G^\natural\to\fg^\natural=\mathrm{Lie}(G^\natural)$ that is bijective over topologically nilpotent elements in the sense of Definition \ref{def:qlog-top-nilp}.
    \end{enumerate}
\end{prop}
\begin{proof}
    We may assume that $G$ is absolutely simple. When $G$ is split, we take $G^\natural$ as in Lemma \ref{lem:pairing-split-case}. Then $G^\natural$ extends to a split reductive group scheme over $\cO$ (which we denote by the same symbol). Moreover, there exists a faithful algebraic representation $\rho:G^\natural\to\mathrm{GL}(V)$ defined over $\cO$, from which we construct an invariant pairing $(\cdot,\cdot)_\rho$ that is non-degenerate over $\cO$. Then statement (3) follows by Lemma \ref{lem:qlog-bij-top-nilp}, statement (1) is clear from the definition of $G^\natural$ while (2) follows from (1) and \cite[Lemma 1.8.17]{KV}. (Note that in \cite{KV} ``topologically unipotent" means strongly topologically unipotent in our Definition \ref{def:top-unip}, see  \emph{loc.cit.} Lemma 1.8.15.)\par
    Next we study the remaining case where $G$ is quasi-split but non-split over $F$. There are the following possibilities for the type of the split form $\bbG$: $\mathrm{A}_n(n\ge1)$, $\mathrm{D}_n(n\ge4)$, $\mathrm{E}_6$. We have either $[F':F]=2$, or $[F':F]=3$ and $\bbG$ is either adjoint or simply connected of $\mathrm{D}_4$ type (in which case $G$ is a tamely ramified triality group of type ${}^3\mathrm{D}_4$).\par  
    We first consider the case where $G$ is not of type ${}^3\mathrm{D}_4$. Then $[F':F]=2$. Let $\bbG^\natural$ be the split reductive group scheme over $\cO$ associated to $\bbG$ in Lemma \ref{lem:pairing-split-case}. It comes with an algebraic representation whose associated trace pairing on $\bbg^\natural=\mathrm{Lie}(\bbG^\natural)$ is non-degenerate over $\cO$. From the definition of $\bbG^\natural$ we see that the outer automorphism $\sigma$ of order $2$ on $\bbG^\natural_{\mathrm{ad}}\cong\bbG_{\mathrm{ad}}$ lifts to $\bbG^\natural$ and we let $G^\natural$ be the fiberwise identity component of $(\mathrm{Res}_{\cO'/\cO}\bbG^\natural)^{\sigma}$ over $\cO$, where the action of $\sigma$ on the $\cO$-scheme $\mathrm{Res}_{\cO'/\cO}\bbG^\natural$ is the composition of the outer automorphism on $\bbG^\natural$ with the automorphism induced by the nontrivial automorphism of the $\cO$-algebra $\cO'$. Then we have an isomorphism $G^\natural_{\mathrm{ad}}\cong G_{\mathrm{ad}}$. Statement (3) follows from Lemma \ref{lem:qlog-bij-top-nilp} while statements (1), (2) are checked immediately.\par 

    Finally we consider the case where $G$ is of type ${}^3\mathrm{D}_4$ so that $[F':F]=3$ and hence $p>3$ by the assumption that $G$ is tamely ramified. We take $G^\natural=G_{\mathrm{ad}}$ and let $\rho:\bbG^\natural_{\mathrm{ad}}\to\mathrm{End}(\bbg_{\mathrm{ad}})$ be the adjoint representation. Since $p>3$, statements (1), (2) are clear and from \cite[I.4.8]{SS-conj} we deduce that the Killing form on $\bbg_{\mathrm{ad}}$ is non-degenerate over $\cO$. Then statement (3) follows from Lemma \ref{lem:qlog-bij-top-nilp}.
\end{proof}

\section{Generalized affine Springer fibers and orbital integrals}\label{sec:transporter}
In this section we introduce a general class of perfect subspaces of Witt vector affine flag varieties and prove their basic finiteness properties under suitable assumptions. As special cases, we get finiteness properties of the Witt-vector affine Springer fibers for the group and the Lie algebra. Then we will discuss the relation between these spaces and $p$-adic orbital integrals.\par 
The generalized affine Springer fibers that we study are defined in an analogous way to classical transporter spaces. Let us recall this notion. Let $H$ be a group acting on a set $X$. For any two subsets $\Omega,\Omega'\subset X$, the associated \emph{transporter space} is the subspace of $H$ consisting of elements $h\in H$ such that $h\Omega'\subset\Omega$. The affine analogue of these spaces occur quite often in the context of Harmonic analysis on reductive $p$-adic groups. \par 
In this section we let $F$ be as in \S\ref{sec:local-fields} and let $G$ be the connected quasi-split reductive group over $F$ with pinning as in \S\ref{sec:parahoric}.
\subsection{Loop groups and affine flag varieties}\label{sec:affine-flag}
Let $\bfP\subset G(F)$ be a parahoric subgroup and let $\cG_\bfP$ be the associated Bruhat-Tits group scheme. Let $L_\cO G$ be the Witt vector loop group and let $L_\cO^+\cG_\bfP$ be the Witt vector positive loop group, both of which are perfect ind-schemes over $k$. Recall that for any perfect $k$-algebra $R$, we have $L_\cO^+\cG_\bfP(R)=G(W_\cO(R))$ and $L_\cO G(R)=G(W_\cO(R)[\frac{1}{p}])$. See  
\cite{Zhu-mixed} for more details. Sometimes we will abuse notation and use $\bfP$ to denote $L_\cO^+\cG_\bfP$.\par
Let $\mathrm{Fl}_\bfP:=L_\cO G/L_\cO^+\cG_\bfP$ be the Witt vector affine (partial) flag variety for $\bfP$. By \cite{BS-Witt} we know that $\mathrm{Fl}_\bfP$ is representable by an inductive limit of perfections of projective varieties over $k$. 
\begin{lem}\label{lem:affine-flag-decompose}
    \begin{enumerate}
        \item Suppose $G=G_1\times G_2$ is a product of two reductive groups over $F$ and $\bfP=\bfP_1\times\bfP_2$ is a product of parahoric subgroups $\bfP_1\subset G_1(F)$ and $\bfP_2\subset G_2(F)$. Then we have a corresponding decomposition of the functors
        \[L_\cO G=L_\cO G_1\times L_\cO G_2,\quad L^+_\cO\cG_\bfP=L^+_\cO\cG_{\bfP_1}\times L^+_\cO\cG_{\bfP_2},\quad\mathrm{Fl}_\bfP=\mathrm{Fl}_{\bfP_1}\times\mathrm{Fl}_{\bfP_2}\]
        \item Suppose $F'/F$ is a finite extension and $\cO'$ is the valuation ring of $F'$. Let $G'$ be a reductive group over $F'$ and $\bfP'\subset G'(F)$ a parahoric subgroup. Let $G=\mathrm{Res}_{F'/F}G$ and $\bfP\subset G(F)$ the parahoric subgroup that equals to $\bfP'$. Then $\cG_\bfP=\mathrm{Res}_{\cO'/\cO}\cG_{\bfP'}$ and we have 
        \[L_\cO G=L_{\cO'}G',\quad L_\cO^+\cG_\bfP=L_{\cO'}^+\cG_{\bfP'},\quad\mathrm{Fl}_\bfP=\mathrm{Fl}_{\bfP'}.\]
    \end{enumerate}
\end{lem}
This follows immediately from the definitions. In the following, when the base field $F$ is clear from the context we will drop the subscript ``$\cO$" and simply write $LG,L^+\cG_\bfP$ etc.

\subsection{Generalized affine Springer fibers}\label{sec:generalization}
Let $\bfP\subset G(F)$ be a parahoric subgroup. Let $S$ be an $F$-scheme with a right action of $G$ and let $v\in S(F)$. Let $LS$ be the Witt-vector loop space and let $\Omega\subset LS$ be a $\bfP$-stable locally closed subspace.
\begin{defn}
    The \emph{generalized affine Springer fiber} for the triple $(\bfP,\Omega,v)$ is the subspace of $\mathrm{Fl}_\bfP$ defined by
    \[\mathrm{Fl}_{\bfP,v}^\Omega:=\{g\in G(F)/\bfP | v\cdot g\in\Omega\}\]
\end{defn}
Let $G_v$ be the $F$-group of stabilizer of $v$ in $G$ and let $LG_v$ be the corresponding loop group over $k$. Then $LG_v$ acts on $\mathrm{Fl}_{\bfP,v}^{\Omega}$ by left multiplication. When $\bfP=G(\cO)$ is the special parahoric, we denote $\mathrm{Gr}_v^\Omega:=\mathrm{Fl}_{G(\cO),v}^\Omega$.
\begin{lem}\label{lem:open-subspace}
    Let $\Omega'\subset\Omega$ be an open (resp. closed) subspace. Then $\mathrm{Fl}_{\bfP,v}^{\Omega'}$ is an open (resp. closed) subspace of $\mathrm{Fl}_{\bfP,v}^{\Omega}$. 
\end{lem}
\begin{proof}
    Consider the morphism $\mathrm{Orb}_v:LG\to LS$ defined by $\mathrm{Orb}_v(g)=v\cdot g^{-1}$. If $\Omega'\subset\Omega$ is an open (resp. closed) subspace, then $\mathrm{Orb}_v^{-1}(\Omega')$ is an open (resp. closed) subspace of $\mathrm{Orb}_v^{-1}(\Omega)$. Then we conclude by observing the natural isomorphisms
    \[\mathrm{Fl}_{\bfP,v}^{\Omega'}\cong\mathrm{Orb}_v^{-1}(\Omega')/\bfP,\quad\mathrm{Fl}_{\bfP,v}^{\Omega}\cong\mathrm{Orb}_v^{-1}(\Omega)/\bfP.\]
\end{proof}
Next we establish a finiteness property for generalized affine Springer fibers in many cases.
\begin{thm}\label{thm:finiteness}
Let $S$ be an affine scheme of finite type over $F$ with right $G$-action. Let $v\in S(F)$ be an element such that the stabilizer $G_v$ is a group of multiplicative type. Let $\bfP\subset G(F)$ be a parahoric subgroup and let $\Omega\subset LS$ be a $\bfP$-stable locally closed subspace that is bounded in the $\varpi$-adic topology. Then we have:
\begin{enumerate}
    \item[(1)] $\mathrm{Fl}_{\bfP,v}^\Omega$ is represented by a perfect scheme locally perfectly of finite type (see Definition \ref{def:perfect-finite-type}). 
    \item[(2)] If moreover $G$ is tamely ramified and $G_v$ is a maximal torus in $G$, then $\mathrm{Fl}_{\bfP,v}^\Omega$ is finite dimensional.
\end{enumerate}
\end{thm}
\begin{proof}
(1) First we assume that $G$ is split over $F$ and $\bfP=G(\cO)$ is a hyperspecial parahoric subgroup. Let $T\subset B$ be a maximal torus contained in a Borel subgroup of $G$ and let $U\subset B$ be the unipotent radical. Consider the ind-scheme
\[Z_v:=\{x\in LU/L^+U\mid v\cdot x\in \Omega\}.\]
Let $\La:=X_*(T)$ be the coweight lattice. From the Iwasawa decomposition \[G(F)=T(F)U(F)G(\cO)=\bigsqcup_{\la\in\La}\la(\varpi)U(F)G(\cO)\] 
we obtain a decomposition 
\begin{equation}\label{eq:decompose-Z-v}
    \mathrm{Gr}_v^\Omega=\bigsqcup_{\la\in\La}Z_{v\cdot\la(\varpi)}.
\end{equation}
Then it suffices to show that $Z_{v\cdot\la(\varpi)}$ is perfectly of finite type (and hence finite dimensional) for each $\la\in\La$. Without loss of generality it suffices to prove this for $Z_v$. \par
Choose a closed embedding $S\into\bbA^r$ into some $r$-dimensional affine space. Since $\Omega$ is bounded, its image in $L\bbA^r$ is contained in a closed subscheme $\varpi^{-m}L^+\bbA^r$ for $m$ sufficiently large. Compose this embedding with the orbit map $U\to S$ sending $g$ to $v\cdot g$, we obtain a morphism $f:U\into\bbA^r$ defined over $F$. By assumption we have $G_v\cap U=1$. Then the orbit map $U\to S$ is a closed embedding by \cite[Theorem 2]{Ros}. Hence $f$ is also a closed embedding. \par
Let $\{\alpha_1,\dotsc,\alpha_N\}=\Phi^+$ be an ordering of the positive roots of $G$. Then we have an isomorphism of schemes $\bbA^N\cong U$ under which the image of the $i$-th coordinate is the root subgroup $U_{\alpha_i}$ and any element $u\in U$ is expressed as a product $u=u_1\dotsm u_N$ where $u_i$ lies in the root subgroup $U_{\alpha_i}$. We let $v^{(0)}:=v$ and define inductively $v^{(i)}:=v^{(i-1)}\cdot u_i$. Then $v^{(N)}=v\cdot u$. Using the coordinates given by the embedding $S\into\bbA^r$ we write each $y\in S$ as $y=(y_1,\dotsc,y_r)$. Then the action map $S\times U\to S$ is given by $r$ polynomials $\{f_i(u,y),1\le i\le r\}$ in the variables $u=(u_1,\dotsc,u_N)$ and $y=(y_1,\dotsc,y_r)$. If $u$ represents a point in $Z_v$, then $v\cdot u\in\Omega$ so that the valuations of its coordinates $f_i(u,v)=f_i(0,\dotsc,0,u_N,v^{(N-1)})$ are bounded below by a fixed constant that is independent of $u$ (equivalently, the absolute value is bounded above). This implies that the valuation of $u_N$ and the coordinates of $v^{(N-1)}$ are bounded below by a constant $C_N$. Now the coordinates of $v^{(N-1)}$ are given by $v^{(N-1)}_i=f_i(0,\dotsc,0,u_{N-2},0,v^{(N-2)})$ and this leads to a lower bound on the valuation of $u_{N-2}$ and the coordinates of $v^{(N-2)}$ by a constant $C_{N-2}$. Repeating this procedure, we see that the valuations of all the coordinates $u_i$, $1\le i\le N$, are bounded below by a constant $C_0$ that only depends on $\Omega$, $v\in S$ and the chosen embedding $S\into\bbA^r$. This implies that $Z_v$ is representable by a finite dimensional perfect scheme perfectly of finite type and we are done with the special case.\par 
Now we consider the general case. By \cite[Proposition 1.4]{PR}, there exists a group scheme closed embedding $\cG_\bfP\into H:=\mathrm{GL}_n\times\bbG_m$ (defined over $\cO$) such that the quotient $H/\cG_\bfP$ is quasi-affine. (Although the setting of \emph{loc.cit.} is over equal-characteristic discrete valuation rings, the proof is valid also in the mixed-characteristic case.) Then by \cite[Proposition 1.20]{Zhu-mixed} the induced map $\mathrm{Fl}_{\bfP}\to\mathrm{Gr}_H$ is a locally closed embedding. Let $S_H:=S\times^GH$. Then $S_H$ is an affine scheme over $F$ by \cite[Proposition 9.7.8]{Alper}. We let $H$ acts on $S_H$ by right multiplication on the second factor. Let $w=(v,1)\in S_H$ and let $\Omega_H$ be the image of $\Omega\times L^+H$ under the natural map $LS\times LH\to LS_H$. Then one checks that $H_w\cong G_v$ and $\Omega_H$ is a bounded $H(\cO)$-stable subspace of $LS_H$. By construction, the image of $\mathrm{Fl}_{\bfP,v}^{\Omega}$ under the embedding $\mathrm{Fl}_\bfP\into\mathrm{Gr}_H$ lies in $\mathrm{Gr}_{w}^{\Omega_H}$, which is locally perfectly of finite type by the case we already proved. Thus $\mathrm{Fl}_{\bfP,v}^{\Omega}$ is locally perfectly of finite type.\par 
(2) First assume that $G$ is split and $\bfP=G(\cO)$ is hyperspecial. If $G_v$ is a \emph{split} maximal torus, which we may assume without loss of generality  that $G_v=T$. Then in the decomposition \eqref{eq:decompose-Z-v} above we have $v\cdot\la(\varpi)=v$ for all $\la\in\La$. Since we already proved that $Z_v$ is finite dimensional, then $\mathrm{Gr}_v^\Omega$ is finite dimensional. In general there is a finite Galois extension $F'/F$ such that $G_v$ splits over $F'$. Let $\cO'\subset F'$ be the valuation ring and let $G':=\mathrm{Res}_{\cO'/\cO}G_{\cO'}$. Then $G$ is a closed sub-group scheme of $G'$ by \cite[A.3.20]{KP23} and
since $G$ is a reductive group scheme over $\cO$, the quotient $G'/G$ is an affine scheme over $\cO$ by \cite[Corollary 9.7.7]{Alper}. Then by \cite[Proposition 1.20]{Zhu-mixed}, the induced map $\mathrm{Gr}_G\to\mathrm{Gr}_{G'}$ is a closed embedding. Let $S':=\mathrm{Res}_{F'/F}(S_{F'})$ and let $\Omega'\subset LS'$ be a bounded $L^+G'$-stable subspace containing $\Omega$. Then the map $\mathrm{Gr}_G\to\mathrm{Gr}_{G'}$ restricts to a locally closed embedding $\mathrm{Gr}_v^{\Omega}\to\mathrm{Gr}_{v}^{\Omega'}$. Since $\mathrm{Gr}_{v}^{\Omega'}$ is finite dimensional by what we already proved, $\mathrm{Gr}_v^{\Omega}$ is also finite dimensional.\par 
Next we assume $G$ is split and $\bfP=\bfI\subset G(\cO)$ is the standard Iwahori subgroup. Let $\Omega'$ be a bounded $L^+G$-stable subspace of $LS$ containing the $\bfI$-stable subspace $\Omega$. Then $\mathrm{Gr}^{\Omega'}_v$ is finite dimensional by the case already proved. Moreover, the fibers of the natural map $\mathrm{Fl}_{\bfI,v}^\Omega\to\mathrm{Gr}^{\Omega'}_v$ are subschemes of the (perfection of) flag variety $\sG/\sB$ over $k$ and hence $\mathrm{Fl}_{\bfI,v}$ is finite dimensional.\par 
Next we treat the case where $G$ is split and $\bfP$ is arbitrary. Without loss of generality we may assume that $\bfP\supset\bfI$. Since $\Omega$ is $\bfP$-stable, it is also $\bfI$-stable and hence $\mathrm{Fl}_{\bfI,v}^\Omega$ is finite dimensional by the case already proved. Observe that the natural map $\mathrm{Fl}_{\bfI,v}^\Omega\to\mathrm{Fl}_{\bfP,v}^\Omega$ is surjective with fibers all isomorphic to (the perfection of) a flag variety $\bfP/\bfI$. So $\mathrm{Fl}_{\bfP,v}^\Omega$ is finite dimensional. \par 
Finally we consider the general case where $G$ is tamely ramified. Let $F'/F$ be a finite tamely ramified Galois extension such that $G_{F'}$ is split and let $\cO'\subset F'$ be the valuation ring. There exists a parahoric subgroup $\bfP'\subset G(F')$ such that $\bfP=\bfP'\cap G(F)$ and $\cG_\bfP$ is the fiberwise identity component of $(\mathrm{Res}_{\cO'/\cO}\cG_{\bfP'})^{\mathrm{Gal}(F'/F)}$ (see, for example, \cite[Theorem 12.7.1(4)]{KP23}). In particular $\cG_\bfP$ is a closed subgroup scheme of $\mathrm{Res}_{\cO'/\cO}\cG_{\bfP'}$. Let $S':=\mathrm{Res}_{F'/F}(S_{F'})$ and let $\Omega'\subset LS'$ be a bounded $\bfP'$-stable subspace containing $\Omega$. Then we have a locally closed embedding $\mathrm{Fl}_{\bfP,v}^\Omega\into\mathrm{Fl}_{\bfP',v}^{\Omega'}$. In fact, one can show that the natural map $\mathrm{Fl}_\bfP\to\mathrm{Fl}_{\bfP'}$ is a closed embedding by an argument similar to \cite[proof of Proposition 8.1]{PZ}. Since $\mathrm{Fl}_{\bfP',v}^{\Omega'}$ is finite dimensional by the case already proved, we conclude that $\mathrm{Fl}_{\bfP,v}^\Omega$ is finite dimensional. 
\end{proof}
We remark that in \cite[Theorem 3.1]{BFN-Springer}, one finds a stronger version of the above result in a special case. \par

In this paper we are mainly interested in the following two special cases:
\begin{itemize}
    \item $S=\fg$ with adjoint action of $G$. Any $g\in G$ acts by sending $X\in\fg$ to $\mathrm{ad}(g)^{-1}X$; In this case we take $\Omega=\mathrm{Lie}(\bfP)$ to be the Lie algebra of a parahoric subgroup $\bfP$, which is an $\cO$-lattice in $\fg(F)$;
    \item $S=G$ with adjoint action of $G$. Any $g\in G$ acts by sending $x\in G$ to $g^{-1}xg$. In this case we will usually consider a parahoric subgroup $\bfP$ and take $\Omega=\widetilde{\bfP}$. 
\end{itemize}
These examples, together with their slight variants in which one allows another group $G'$ with the same adjoint group as $G$ to act on $G$ or $\fg$, will be studied in detail later. Before that, we explain the relation between generalized affine Springer fibers and orbital integrals.
\subsection{Orbital integrals}\label{sec:orbital-integrals}
In this subsection we assume that $k=\bbF_q$ is the finite field with $q$ elements where $q$ is a power of $p$. Then $F$ is a finite extension of $\bbQ_p$ with ring of integers $\cO$ and residue field $k$. For each integer $n\ge1$, let $k_n=\bbF_{q^n}$ be the unique degree $n$ extension of $k$, $\cO_n=W_\cO(k_n)$ and $F_n=\cO_n[\frac{1}{p}]$. Then $F_n/F$ is the unique unramified extension of degree $n$.\par 
We keep the notations from \S\ref{sec:generalization}. Let $f\in C_c^\infty(S(F))$ be a locally constant function with compact support and suppose that its support is contained in $\Omega$. Suppose we take $v\in S(F)$ such that the stabilizer $G_v(F)$ is unimodular. We fix Haar measures $dg$ (resp. $dg_v$) on $G(F)$ (resp. $G_v(F)$). They induce a $G(F)$-invariant quotient measure $d\dot{g}$ on $G_v(F)\setminus G(F)$.\par
The \emph{orbital integral of $f$ at $v$} is defined by
\[ O_v(f)=\int_{G_v(F)\setminus G(F)}f(v\cdot g)d\dot{g}\]
If we suppose moreover that $f$ is $\bfP$-invariant, then we can calculate it as
\[O_v(f)=\sum_{g\in G_v(F)\backslash\mathrm{Fl}_{\bfP,v}^\Omega(k)}\frac{f(v\cdot g)\vol(\bfP,dg)}{\vol(g\bfP g^{-1}\cap G_v(F),dg_v)}\]
where the sum is over the set of $G_v(F)$-orbits on $\mathrm{Fl}_{\bfP,v}^\Omega(k)$. \par 
We would like to relate $O_v(f)$ and certain weighted counting on the quotient stack $[LG_v\backslash\mathrm{Fl}_{\bfP,v}^\Omega]$. In most cases that we are interested in, including the case of the Witt vector affine Springer fibers, this quotient stack is essentially of finite type, which guarantee that the sum in the formula above is finite. To proceed further, we make more assumptions.
\begin{itemize}
    \item Assume that the stabilizer $G_v$ is a group scheme of multiplicative type (in particular, commutative). \par
    Let $\La_v:=X_*(G_{v,\Breve{F}})$ be the lattice of coweights of $G_v$ defined over $\Breve{F}$. Via the action of $\mathrm{Gal}(F^{\mathrm{ur}}/F)\cong\mathrm{Gal}(\Bar{k}/k)$, we view $\La_v$ as an \'etale group scheme over $k$. The uniformizer $\varpi$ induces an embedding of $k$-group schemes $\La_v\to LG_v$ 
    sending $\la\in\La_v$ to $\la(\varpi)\in LG_v$. Then for any $n\ge1$, $\La_v(k_n)$ is a discrete cocompact subgroup of $LG_v(k_n)=G_v(F_n)$.
    \item Replacing $\La_v$ by a finite index subgroup if necessary, we may assume that $\La_v$ acts freely on $\mathrm{Fl}_{\bfP,v}^\Omega$.  
    \item Finally, we assume moreover that the quotient $\La_v\backslash\mathrm{Fl}_{\bfP,v}$ is an algebraic space perfectly of finite type over $k$. 
\end{itemize}

In most examples that we are interested in, the function $f$ has geometric origin. To explain this let us continue to make more assumptions. 
\begin{itemize}
    \item Suppose $\Omega$ is the set of $k$-points of a $L^+\cG_\bfP$-stable perfect sub-indscheme, which we still denote by $\Omega$, of the loop space $LS$.
    \item Suppose there is an object $\cF\in D_c^b([L^+\cG_\bfP\backslash\Omega],\overline{\bbQ}_\ell)$ in the bounded derived category of $L^+\cG_\bfP$-equivariant $\ell$-adic sheaf on $\Omega$, where $\ell\ne p$ is another prime, which gives rise to $f$ by Grothendieck's sheaf-function dictionary. In other words, for any $x\in\Omega(k)\subset LS(k)=S(F)$, we have $f(x)=\mathrm{Tr}(\mathrm{Frob}_x,\cF_{\Bar{x}})$ where $\mathrm{Frob}_x$ is the geometric Frobenius at $x$. 
\end{itemize}
The sheaf $\cF$ (more precisely, complex of $\ell$-adic sheaves) contains much more information than the function $f$. For instance, by taking traces of geometric Frobenius on stalks of $\cF$ we get a family of functions: $f_n$ on $LS(k_n)=S(F_n)$ for each integer $n\ge1$, among which $f=f_1$ is only one of them. Denote $\bfP_n:=\cG_\bfP(\cO_n)$. Then we have
\begin{equation}
    \begin{split}
        O_v(f_n)&=\sum_{g\in G_v(F_n)\setminus\mathrm{Fl}_{\bfP,v}^\Omega(k_n)}\frac{f_n(v\cdot g)\vol(\bfP_n,dg)}{\vol(g\bfP_n g^{-1}\cap G_v(F_n),dg_v)}\\
        &=\frac{\vol(\bfP_n,dg)}{\vol(\La_v(k_n)\backslash G_v(F_n),dg_v)}\sum_{g\in\La_v(k_n)\backslash\mathrm{Fl}_{\bfP,v}^\Omega(k_n)}f_n(v\cdot g)
    \end{split}
\end{equation}
To proceed further and make this expression more geometric, one is confronted with the discrepancy between the set theoretic quotient 
\[\La_v(k_n)\backslash\mathrm{Fl}_{\bfP,v}^\Omega(k_n)\] 
and the set of $k_n$-points of the geometric quotient 
\[(\La_v\backslash\mathrm{Fl}_{\bfP,v}^\Omega)(k_n).\]
There difference is measured by the cohomology group $H^1(k_n,\La_v)$. Any $\xi\in H^1(k_n,\La_v)$ defines a twisted form $\mathrm{Fl}_{\bfP,v}^{\Omega,\xi}$ of $\mathrm{Fl}_{\bfP,v}^\Omega$ and we have the decomposition:
\[(\La_v\backslash\mathrm{Fl}_{\bfP,v}^\Omega)(k_n)=\bigsqcup_{\xi\in H^1(k_n,\La_v)}\La_v(k_n)\backslash\mathrm{Fl}_{\bfP,v}^{\Omega,\xi}(k_n).\]
Here we have used that $\La_v$ is commutative so that it does not have nontrivial inner forms. The nonempty pieces in this decomposition are those $\xi$ that lie in the subset
\[H^1(F_n,\La_v\to G):=\ker(H^1(k_n,\La_v)\cong H^1(F_n,\La_v)\to H^1(F_n,G(\Breve{F})))\]
Each such cohomology class $\xi$ can be represented by a cocycle of the form $\sigma^n\mapsto\sigma^n(g)g^{-1}$ where $\sigma\in\mathrm{Gal}(\Bar{k}/k)$ is the arithmetic Frobenius and $g\in G(\Breve{F})$. Then $g\bfP_n$ is a point in $\mathrm{Fl}_{\bfP,v}^{\Omega,\xi}(k_n)$. Let $v_\xi:=v\cdot g$. Then we have a canonical isomorphism $\mathrm{Fl}_{\bfP,v}^{\Omega,\xi}\cong\mathrm{Fl}_{\bfP,v_\xi}^\Omega$, from which we deduce that
\[\sum_{x\in(\La_v\backslash\mathrm{Fl}_{\bfP,v}^\Omega)(k_n)}\mathrm{Tr}(\mathrm{Frob}_x,\cF_{\Bar{x}})=\frac{\vol(\La_v(k_n)\backslash G_v(F_n),dg_v)}{\vol(\bfP_n,dg)}\sum_{\xi\in H^1(F_n,\La_v\to G)}O_{v_\xi}(f_n)\]
We note that there is a natural surjection from the set $H^1(F_n,\La_v\to G)$ to the set of $G(F_n)$-orbits inside the $G(\Breve{F})$-orbits of $v$. So one can define a suitable notion of stable orbital integral, which would be a certain multiple of the right hand side above.\par 
In particular, when $n$ is sufficiently large so that $H^1(k_n,\La_v)=1$, we have
\[O_v(f_n)=\frac{\vol(\bfP_n,dg)}{\vol(\La_v(k_n)\backslash G_v(F_n),dg_v)}\sum_{x\in(\La_v\backslash\mathrm{Fl}_{\bfP,v}^\Omega)(k_n)}\mathrm{Tr}(\mathrm{Frob}_x,\cF_{\Bar{x}}).\]
Here are some typical examples in which all the assumptions above are satisfied.
\begin{ex}
    Let $S=G$ with the adjoint action of $G$. For simplicity we assume that $G$ is split over $F$ and $\bfP\subset G(F)$ is a standard parahoric subgroup. Recall that we have the decomposition of $G(F)$ into $\bfP$-double cosets 
    \[G(F)=\bigsqcup_{w\in W_\bfP\backslash\widetilde{W}/W_\bfP}\bfP\dot{w}\bfP\]
    where $\widetilde{W}$ is the extended affine Weyl group of $G(F)$ and $W_\bfP$ is the Weyl group of the reductive quotient of $\bfP$. For any $w\in\widetilde{W}$ we use $\dot{w}\in G(F)$ to denote a representative of $w$. 
    We could take $\Omega$ to be a finite union of these double cosets and $\cF$ to be the constant sheaf $\overline{\bbQ}_\ell$. The resulting spaces have been studied in \cite{He2023affine}. In fact, the Witt vector affine Springer fibers for the groups, which will be introduced in \S\ref{sec:ASF-group-intro} and studied in detail in \S\ref{sec:dim-group}, are special cases of these spaces where $\Omega$ is a union of double cosets for certain \emph{length $0$} elements $w$.\par 
    In particular, when $\bfP=G(\cO)$ is a hyperspecial we get the Cartan decomposition labeled by the set of dominant coweights of a maximal split torus. In this case the equal-characteristic version of these spaces have been studied in \cite{Bou}, \cite{BC} and \cite{Chi}, using global methods which do not apply in the mixed-characteristic setting. 
\end{ex}
\begin{ex}
    Let $S=\fg$ with the adjoint action of $G$. We take $\Omega=\mathrm{Lie}(\bfP)$ and the sheaf $\cF$ to be the constant sheaf $\overline{\bbQ}_\ell$. Then we get the Witt vector affine Springer fibers for the Lie algebras to be introduced in \S\ref{sec:ASF-Lie-intro} and studied in detail in \S\ref{sec:dim-Lie-alg}.
\end{ex}

\section{Basic properties of Witt-vector affine Springer fibers}\label{sec:main-thm}
In this section we introduce the Witt-vector affine Springer fibers for the group and the Lie algebra, state the main theorems and prove the non-emptiness criteria. During the reduction steps in the proof of the main theorems, it will become necessary to slightly generalize the usual definitions in \S\ref{sec:intro}. We introduce the more general setup first. 

\subsection{The setup}\label{sec:setup-adjoint-pair}
Let $(G',G)$ be a pair of reductive groups over $F$ equipped with an isomorphism $G'_{\mathrm{ad}}\cong G_{\mathrm{ad}}$ between their adjoint groups. Then $G'$ acts on $G$ by conjugation and on $\fg=\mathrm{Lie}(G)$ through the adjoint representation. Let $\bfP'\subset G'(F)$ and $\bfP\subset G(F)$ be parahoric subgroups that have the same image in their adjoint groups. Then $\bfP'$ and $\bfP$ are uniquely determined by each other and $\bfP$ is preserved by the adjoint action of $\bfP'$. Let $\widetilde{\bfP}$ (resp. $\widetilde{\bfP'}$) be the normalizer of $\bfP$ (resp. $\bfP'$) in $G(F)$ (resp. $G'(F)$). 

\subsection{Witt-vector affine Springer fibers for the Lie algebras}\label{sec:ASF-Lie-intro}
We first introduce the Lie algebra version of the Witt-vector affine Springer fibers. We maintain the assumptions and notations in \S\ref{sec:setup-adjoint-pair}.\par 
\begin{defn}
    The \emph{Witt-vector affine Springer fiber in $\mathrm{Fl}_{\bfP'}=LG'/\bfP'$} for an element $\ga\in\fg(F)$ is the closed perfect sub-indscheme $X_{\bfP',\ga}^\fg$ of $\mathrm{Fl}_{\bfP'}=LG'/\bfP'$ whose set of $k$-points is 
\[X_{\bfP',\ga}^\fg(k)=\{g\in G'(F)/\bfP'| \mathrm{ad}(g)^{-1}\ga\in\mathrm{Lie}(\bfP)\}.\]
\end{defn}
We will simplify notations in the following situations. When $G=G'$ and $\bfP=\bfP'$ we will simply denote $X_{\bfP,\ga}:=X_{\bfP,\ga}^\fg$. If $\bfP=\bfI$ is an Iwahori subgroup of $G(F)$, we denote $Y_\ga:=X_{\bfI,\ga}^{\fg}$. If moreover $G$ is tamely ramified and $\bfP=G(\cO)$ is the special parahoric subgroup defined in \S\ref{sec:group-model}, we denote $X_\ga:=X_{G(\cO),\ga}^{\fg}$.\par
First let us make some initial reductions in the study of the spaces $X_{\bfP',\ga}^\fg$.\par 
Let $G_1:=G^{\mathrm{sc}}\cong (G')^{\mathrm{sc}}$ be the simply connected cover of the derived group of $G$ (and also $G'$) and let $\bfP_1$ be the inverse image of $\bfP$ in $G_1(F)$. Let $\fg_1:=\mathrm{Lie}(G_1)$ be the Lie algebra of $G_1$, which is also canonically isomorphic to the Lie algebra of the adjoint group of $G$ (and also $G'$). Let $\ga\in\fg(F)^{\mathrm{rs}}$ be a regular semisimple element and let $\ga_1\in\fg_1(F)$ be its image under the natural map $\fg\to\fg_1\cong\mathrm{Lie}(G_{\mathrm{ad}})$. Then we have the Witt-vector affine Springer fiber $X_{\bfP_1,\ga_1}=X_{\bfP_1,\ga_1}^{\fg_1}$ in $\mathrm{Fl}_{\bfP_1}=LG_1/\bfP_1$. For each element $n\in\widetilde{\bfP'}/\bfP'$, we define a map 
\[i_n:X_{\bfP_1,\ga_1}\to X_{\bfP',\ga}^{\fg}\]
sending $g_1\bfP_1$ to $\pi'(g_1)\dot{n}\bfP'$, where $\pi':G_1\cong(G')^\mathrm{sc}\to G'$ is the natural homomorphism and $\dot{n}\in\bfP'$ is a representative of $n$.
\begin{prop}\label{prop:reduction-Lie-alg}
    With notations and assumptions as above, the map $i_n$ is a closed embedding for each element $n\in\widetilde{\bfP'}/\bfP'$. Altogether they induce an isomorphism
    \[\bigsqcup_{n\in\widetilde{\bfP'}/\bfP'}i_n:\bigsqcup_{n\in\widetilde{\bfP'}/\bfP'}X_{\bfP_1,\ga_1}\cong X_{\bfP',\ga}^{\fg}.\]
\end{prop}
\begin{proof}
    The morphism $i_n$ is the restriction to $X_{\bfP_1,\ga_1}$ of the map $\mathrm{Fl}_{\bfP_1}=LG_1/\bfP_1\to\mathrm{Fl}_{\bfP'}=LG'/\bfP'$ defined by $g_1\bfP_1\mapsto\pi'(g_1)\dot{n}\bfP'$, the latter of which is an isomorphism onto a connected component of the target by \cite[Proposition 1.21]{Zhu-mixed}. Therefore each $i_n$ is a closed embedding. Clearly the image of the different embeddings are disjoint.\par  
    It remains to show surjectivity on $k$-points. For any $g\in G'(F)$ representing a closed point of $X_{\bfP',\ga}^\fg$, we choose an element $h\in G_1(F)$ such that the inverse image of $\mathrm{ad}(g)\bfP$ in $G_1(F)$ equals to $\mathrm{ad}(h)\bfP_1$. Then from the condition $\ga\in\mathrm{ad}(g)\mathrm{Lie}(\bfP)$ we deduce that $\ga_1\in\mathrm{ad}(h)\mathrm{Lie}(\bfP_1)$ and hence $h\bfP_1\in X_{\bfP_1,\ga_1}$. On the other hand we have $\mathrm{ad}(\pi'(g_1))\bfP'=\mathrm{ad}(g)\bfP'$ and hence $\pi'(h)^{-1}g\in\widetilde{\bfP'}$. Let $n\in\widetilde{\bfP'}/\bfP'$ be the class of $\pi'(h)^{-1}g$. Therefore $g\bfP'=i_n(h\bfP_1)$ and we are done.
\end{proof}

\begin{thm}\label{thm:nonempty-Lie}
    Suppose $G$ is essentially tamely ramified and the residue characteristic $p$ is not a torsion prime for $G$ (for example, this holds if $p$ is good for $G$). Then the space $X_{\bfP',\ga}^\fg$ is nonempty if and only if $\ga\in\fg(F)$ is bounded in the sense of Definition \ref{def:bounded-Lie-alg}. Furthermore, in this case it is finite dimensional if and only if $\ga$ is bounded and regular semisimple.
\end{thm}
\begin{proof}
    By Proposition \ref{prop:reduction-Lie-alg}, we may assume that $G$ is tamely ramified, semisimple simply connected. Then $p$ is not a torsion prime for the root datum of $G$ and the first statement follows from Corollary \ref{cor:chi-surjection}. It remains to prove the second statement.\par
    If $\ga\in\fg^{\mathrm{rs}}(F)$, then $X_{\bfP,\ga}$ is finite dimensional by Theorem~\ref{thm:finiteness}. Conversely if $\ga\in\fg(F)$ is not regular semisimple, then the centralizer $G_\ga$ contains a nontrivial unipotent subgroup $U$ and the orbit of its loop group $LU$ on $X_{\bfP,\ga}$ is infinite dimensional.
\end{proof}
Now we state the main theorem on the dimension of $X_{\bfP',\ga}^\fg$ for a bounded regular semisimple element $\ga$. The formulae involve two numerical invariants of the conjugacy class of $\ga$: its discriminant valuation $d_\fg(\ga)$ and its \emph{Artin conductor} $\mathrm{Art}_\ga:=\mathrm{Art}(G_\ga)$.
\begin{thm}\label{thm:main-Lie-algebra-case}
    Let $(G',G)$ be a pair of reductive groups over $F$ equipped with an isomorphism $G'_{\mathrm{ad}}\cong G_{\mathrm{ad}}$ between their adjoint groups. Suppose that $G$ is essentially tamely ramified (see Definition \ref{def:ess-tame-ram}) and the residue characteristic of $F$ is not bad for $G$. Let $\bfP'\subset G'(F)$ and $\bfP\subset G(F)$ be parahoric subgroups whose images in $G'_{\mathrm{ad}}(F)\cong G_{\mathrm{ad}}(F)$ coincide. For any bounded regular semisimple element $\ga\in\fg^{\mathrm{rs}}(F)$, the Witt-vector affine Springer fiber $X_{\bfP',\ga}^\fg$ is represented by a finite dimensional perfect scheme locally perfectly of finite type (see Definition \ref{def:perfect-finite-type}). Its dimension satisfies
    \[\dim X_{\bfP',\ga}^\fg\le\frac{1}{2}(d_\fg(\ga)+\mathrm{Art}(G)-\mathrm{Art}_\ga),\]
    and if $\bfP$ is contained in a special parahoric subgroup of $G(F)$, then equality holds:
    \[\dim X_{\bfP',\ga}^\fg=\frac{1}{2}(d_\fg(\ga)+\mathrm{Art}(G)-\mathrm{Art}_\ga).\]
    Furthermore, if $\bfP=\bfI$ is an Iwahori subgroup, then $X_{\bfI',\ga}^\fg$ is equi-dimensional. 
\end{thm}
\begin{rem}\label{rem:general-parahoric}
    We expect that the dimension formulae should hold for any parahoric subgroup $\bfP$. In the equal-characteristic case, when the residue field is $\bbC$, the dimension formulae for general $\bfP$ can be proved by the same methods as in \cite[\S4]{KL}, using the description of the representations of affine Weyl groups on the homology of the affine flag varieties. If this description also holds for the Witt vector affine flag varieties, then the dimension formulae for general $\bfP$ would follow. On the other hand, we do not know if the equi-dimensionality property holds for general parahoric subgroup $\bfP$. In the equal characteristic case, the equi-dimensionality property can be proved for hyperspecial parahoric subgroups by using the geometry of Hitchin fibrations, see \cite[3.3.1, 4.16.2]{Ngo10}. 
\end{rem}
The proof of Theorem \ref{thm:main-Lie-algebra-case} will be presented in \S\ref{sec:dim-Lie-alg}. Here we make some initial reductions.
\begin{cor}\label{cor:reduction-Lie-alg}
    Let $(G,G',\bfP,\bfP')$ be as in \S\ref{sec:setup-adjoint-pair}. Let $(H,H')$ be another pair of reductive groups over $F$ with an isomorphism between the adjoint groups $H_{\mathrm{ad}}\cong H'_{\mathrm{ad}}$ and let $\bfP_H\subset H(F)$, $\bfP_H'\subset H'(F)$ be parahoric subgroups whose images in the adjoint group coincide. Suppose there is an isomorphism $G_\mathrm{ad}\cong H_\mathrm{ad}$ such that the image of $\bfP$ and $\bfP_H$ in the adjoint group coincide. Let $\fh:=\mathrm{Lie}(H)$ be the Lie algebra of $H$. Let $\ga\in\fg(F)$ and $\ga_H\in\fh(F)$ be bounded regular semisimple elements whose image in $\mathrm{Lie}(G_{\mathrm{ad}})\cong\mathrm{Lie}(H_{\mathrm{ad}})$ coincide. Then Theorem \ref{thm:main-Lie-algebra-case} holds form $X_{\bfP',\ga}^\fg$ if and only if it holds for $X_{\bfP_H',\ga_H}^\fh$. 
\end{cor}
\begin{proof}
    By the assumptions and Proposition \ref{prop:reduction-Lie-alg}, both spaces are disjoint unions of the same affine Springer fibers for the common derived group of $G$ and $H$ and hence have the same dimension. On the other hand, the invariant $d_\fg(\ga)$ depends only the image of $\ga$ in the adjoint Lie algebra and the difference $\mathrm{Art}(G)-\mathrm{Art}_\ga$ does not change when we replace $G$ by its adjoint group (and replace $\ga$ by its image in the adjoint Lie algebra). 
\end{proof}
Next we show that Theorem \ref{thm:main-Lie-algebra-case} can be reduced to the case of topologically nilpotent elements. We will employ this reduction procedure when comparing the group version and Lie algebra version of Witt-vector affine Springer fibers (cf. \S\ref{sec:finish-proof-group}). \par
Thanks to Corollary \ref{cor:reduction-Lie-alg}, we may assume that $G=G'$ and $\bfP=\bfP'$ and we have the freedom of replacing $G$ by any reductive group whose adjoint group is isomorphic to $G_{\mathrm{ad}}$. Therefore by Proposition \ref{prop:pairing} we may assume that the assumptions in Proposition \ref{prop:top-jordan-Lie} are satisfied. In particular $G$ is tamely ramified. Then a bounded regular semisimple element $\ga\in\fg(F)^{\mathrm{rs}}$ admits a topological Jordan decomposition $\ga=\ga_0+\ga_1$ where $\ga_0\in\fg(F)$ is strongly semisimple and $\ga_1\in\fg(F)$ is topologically nilpotent. By Proposition \ref{prop:top-jordan-Lie}, after $G(F)$-conjugation we may and do assume that $\ga_0\in\ft(\cO)$. Let $H=G_{\ga_0}$ be the centralizer of $\ga_0$ and let $\fh=\fg_{\ga_0}$ be its Lie algebra. Then $H$ is an $F$-Levi subgroup of $G$ containing the maximal torus $T$, and $\ga_1\in\fh(F)^{\mathrm{rs}}$ is bounded and regular semisimple. Let $\bfP_H:=\bfP\cap H(F)$. Then $\bfP_H$ is a parahoric subgroup of $H(F)$.
\begin{prop}\label{prop:HC-descent-Lie-alg}
    With notations as above, there is a canonical isomorphism of perfect schemes $X_{\bfP_H,\ga_1}\cong X_{\bfP,\ga}$. 
\end{prop}
\begin{proof}
    Since $\ga_0\in\ft(\cO)$ lies in the center of $\fh(F)$, we see that the natural embedding $\mathrm{Fl}_{\bfP_H}\to\mathrm{Fl}_\bfP$ induces an embedding $X_{\bfP_H,\ga_1}\to X_{\bfP,\ga}$. It remains to show that it is surjective.\par 
    Let $g\bfP$ be a point in $X_{\bfP,\ga}$ so that $\ga\in\mathrm{ad}(g)\mathrm{Lie}\bfP$. Suppose that $\ga_0\notin\mathrm{ad}(g)\mathrm{Lie}\bfP$, then $g\bfP\ne h\bfP$ for any $h\in H(F)$ and by \cite[Lemma 2.3.3]{KM} we get that $\ga\notin\mathrm{ad}(g)\mathrm{Lie}\bfP$. This is a contradiction and hence we must have $\ga_0\in\mathrm{ad}(g)\mathrm{Lie}\bfP$. Then we deduce that $g\bfP=h\bfP$ for some $h\in H(F)$ by \emph{loc. cit.} Theorem 2.3.1. Then we get that $\ga\in\mathrm{ad}(g)\mathrm{Lie}\bfP\cap\fh=\mathrm{ad}(h)\bfP_H$ and since $\ga_0\in\mathrm{ad}(h)\mathrm{Lie}\bfP_H$ we deduce that $\ga_1\in\mathrm{ad}(g)\mathrm{Lie}\bfP_H$. Therefore $g\bfP=h\bfP$ lies in the image of $X_{\bfP_H,\ga_1}\to X_{\bfP,\ga}$ and we are done.
\end{proof}
\subsection{Witt-vector affine Springer fibers for the groups}\label{sec:ASF-group-intro}
We proceed to introduce the group version of Witt vector affine Springer fibers. Keep the assumptions and notations in \S\ref{sec:setup-adjoint-pair}.
\begin{defn}
The \emph{Witt vector affine Springer fiber in $\mathrm{Fl}_{\bfP'}=LG'/\bfP'$ for an element $\ga\in G(F)$} is the closed sub-indscheme of $\mathrm{Fl}_{\bfP'}$ whose set of $k$-points is
\[X^{G}_{\bfP',\ga}(k)=\{g\in G'(F)/\bfP' | \mathrm{ad}(g)^{-1}(\ga)\in\widetilde{\bfP}\}.\]
More precisely, for any perfect $k$-algebra $R$, $X^{G}_{\bfP',\ga}(R)$ is the set of isomorphism classes of triples $(\cE',\phi,\iota)$ in which 
\begin{itemize}
    \item $\cE'$ is a $\cG_{\bfP'}$-torsor on $\spec W(R)$, where $\cG_{\bfP'}$ is the $\cO$-model of $G'$ whose set of $\cO$-points is $\bfP'$;
    \item $\phi$ is a $W(R)$-point of the $W(R)$-scheme $\cE'\times^{\cG_{\bfP'}}\cG_{\widetilde{\bfP}}$ where $\cG_{\widetilde{\bfP}}$ is the $\cO$-model of $G$ whose set of $\cO$-points is $\widetilde{\bfP}$. Here we form the contracted product using the action of $\cG_{\bfP'}$ on $\cG_{\widetilde{\bfP}}$ that uniquely extends the adjoint action of $G'$ on $G$;
    \item $\iota$ is an isomorphism between the restriction of the pair $(\cE,\phi)$ to $\spec W(R)[\frac{1}{p}]$ and $(\cE_0',\ga)$, where $\cE_0'$ is the trivial $G'$-torsor on $\spec W(R)[\frac{1}{p}]$.
\end{itemize}
\end{defn}
Although we are mainly interested in the case where $G=G'$ and $\bfP=\bfP'$, the current slightly more general definition will be convenient in certain reduction steps in the proof of the main theorems. Two special cases are of particular importance. If $G=G'$ and $\bfP=\bfI$ is the standard Iwahori subgroup of $G(F)$, we denote $Y_\ga:=X^G_{\bfI,\ga}$. If $G=G'$ is tamely ramified and $\bfP=G(\cO)$ is the special parahoric subgroup defined in \S\ref{sec:group-model}, we denote $X_\ga:=X^G_{G(\cO),\ga}$. In the following, it will be clear from the context whether $\ga$ is an element of the Lie algebra or the group, so the conflict of notations in the two cases should not cause confusion.\par 
First we establish a non-emptiness criteria.
\begin{thm}\label{thm:nonempty-group}
    For any $\ga\in G(F)$, the Witt vector affine Springer fiber $X^{G}_{\bfP',\ga}$ is nonempty if and only if $\ga$ is bounded mod center and $\kappa_G(\ga)\in\pi_0(\widetilde{\bfP})$. In particular, when $G=G'$ and $\bfP=\bfI$, we have $Y_\ga=X_{\bfI,\ga}$ is nonempty if and only if $\ga$ is bounded mod center. Furthermore, suppose that $G$ is essentially tamely ramified (see Definition \ref{def:ess-tame-ram}) and $X^{G}_{\bfP',\ga}$ is nonempty, then $X^{G}_{\bfP',\ga}$ is finite dimensional if and only if $\ga$ is regular semisimple. 
\end{thm}
\begin{proof}
    First we show the non-emptiness criterion. The necessity is clear since any element in $\widetilde{\bfP}$ is bounded mod center. It remains to show sufficiency. We may assume that $\bfP\supset\bfI$ is standard. If $\ga\in G(F)$ is bounded mod center, then there exists $g\in G(F)$ such that $g^{-1}\ga g\in\widetilde{\bfI}$ by Lemma \ref{lem:bounded}. After multiplying by an element in $\widetilde{\bfI}$ we may assume that $g$ is the image of an element $g_1\in\mathrm{G}^{\mathrm{sc}}(F)$ under the natural homomorphism $\mathrm{G}^{\mathrm{sc}}(F)\to G(F)$ (cf. \cite[Definition 7.4.1, 7.4.5]{KP23}). Let $g'\in G'(F)$ be the image of $g_1$ under the natural homomorphism $\mathrm{G}^{\mathrm{sc}}(F)\to G'(F)$. Then we have $\mathrm{ad}(g')^{-1}(\ga)\in\widetilde{\bfI}$. By assumption we have $\kappa_G(\mathrm{ad}(g')^{-1}(\ga))\in\pi_0(\widetilde{\bfP})$ and hence there exists $x\in\widetilde{\bfP}$ such that $\kappa_G(x)=\kappa_G(\mathrm{ad}(g')^{-1}(\ga))$. After multiplying by an element of $\bfP$ we may assume that $x\in\widetilde{\bfI}\cap\widetilde{\bfP}$. Consequently we have $x^{-1}\mathrm{ad}(g')^{-1}(\ga)\in\bfI\subset\widetilde{\bfP}$ and then we get that $\mathrm{ad}(g')^{-1}(\ga)\in\widetilde{\bfP}$. This shows that $g'\bfP'$ is a point in $X^G_{\bfP',\ga}$ and in particular $X^G_{\bfP',\ga}$ is nonempty.\par 
    Now we prove the last statement. Suppose $G$ is essentially tamely ramified and $X^{G}_{\bfP',\ga}$ is nonempty. By Proposition \ref{prop:HC-descent} we may assume that $G$ is tamely ramified. (Note that this proposition only depends on the first part of the current theorem that we already proved, so there is no circular reasoning.) If $\ga\in G^{\mathrm{rs}}(F)$ is regular semisimple, then $X^{G}_{\bfP',\ga}$ is finite dimensional by Theorem \ref{thm:finiteness}. If $\ga$ is not regular semisimple, then there exists a nontrivial unipotent subgroup $U\subset G_\ga$ and the $LU$-orbits on $X^{G}_{\bfP',\ga}$ are infinite dimensional. 
\end{proof}

Now we can state the main theorem on the dimension formula for $X^{G}_{\bfP',\ga}$. It will involve the following numerical invariants of the conjugacy class of $\ga\in G^{\mathrm{rs}}(F)$: 
\begin{itemize}
    \item the discriminant valuation $d_G(\ga)$ (see Definition \ref{def:disc-group});
    \item the \emph{Artin conductor} of $\ga$, defined by $\mathrm{Art}_\ga:=\mathrm{Art}(G_\ga)$ (see Definition \ref{def:conductor-group});
    \item the Kottwitz invariant $\kappa_G(\ga)$ and its defect $\mathrm{def}(\kappa_G(\ga))$ (see Definition \ref{def:defect-Kott-invariant}).
\end{itemize}

\begin{thm}\label{thm:main-group-case}
Let $(G',G)$ be a pair of reductive groups over $F$ equipped with an isomorphism $G'_{\mathrm{ad}}\cong G_{\mathrm{ad}}$ between their adjoint groups. Suppose that $G$ is essentially tamely ramified (see Definition \ref{def:ess-tame-ram}). Let $\bfP'\subset G'(F)$ and $\bfP\subset G(F)$ be parahoric subgroups whose images in $G'_{\mathrm{ad}}(F)\cong G_{\mathrm{ad}}(F)$ coincide. Let $\ga\in G(F)^{\mathrm{rs}}$ be a regular semisimple element that is bounded mod center. Then the affine Springer fiber $X^{G}_{\bfP',\ga}$ is represented by a finite dimensional perfect scheme locally perfectly of finite type (see Definition \ref{def:perfect-finite-type}). Suppose moreover that the residue characteristic of $F$ is not bad for $G$. Then we have
    \[\dim X^{G}_{\bfP',\ga}\le\frac{1}{2}(d_G(\ga)+\mathrm{def}(\kappa_G(\ga))+\mathrm{Art}(G)-\mathrm{Art}_\ga).\]
and if $\bfP$ is contained in a special parahoric subgroup, then equality holds:
    \[\dim X^{G}_{\bfP',\ga}=\frac{1}{2}(d_G(\ga)+\mathrm{def}(\kappa_G(\ga))+\mathrm{Art}(G)-\mathrm{Art}_\ga).\]
Furthermore, if $\bfP=\bfI$ is an Iwahori subgroup, then $X_{\bfI',\ga}^G$ is equi-dimensional.
\end{thm}
This will be established by reducing to similar statements for Lie algebras that we introduced in the previous subsection. Similar to Lie algebra case, we expect that the equality should hold for any reductive group with no assumption on residue characteristic and any parahoric subgroup $\bfP$, but we do not know if $X_{\bfP',\ga}^{G}$ is equi-dimensional in general (cf. Remark \ref{rem:general-parahoric}).\par 
To finish this subsection, we establish the equi-dimensionality property in some special cases, which generalizes similar results of \cite{KL} in the equal characteristic setting. 
\begin{prop}\label{prop:equi-dim-group}
    Let $G$ be a reductive group over $F$ and let $G'=G^{\mathrm{sc}}$. 
    Let $\bfI\subset G(F)$ be an Iwahori subgroup and let $\bfI'\subset G'(F)$ be the inverse image of $\bfI$. Let $\ga\in G^{\mathrm{rs}}(F)$ be a regular semisimple element that is strongly topologically unipotent. Then the Witt-vector affine Springer fiber $X_{\bfI',\ga}=X_{\bfI',\ga}^G$ in $LG'/\bfI'$ is connected and equi-dimensional.   
\end{prop}
\begin{proof}
    We follow the arguments in \cite[\S4]{KL}. For a parahoric subgroup $\bfP\subset G(F)$ containing $\bfI$ whose inverse image in $G'(F)$ is denoted by $\bfP'$, we consider the space 
    \[X_{\bfP',\ga}^+:=\{g\in G'(F)/\bfP'_+ | \mathrm{ad}(g)^{-1}(\ga)\in\bfI_+\}\]
    on which $\sB_{\bfP'}=\bfI'/\bfP'_+$ acts by right multiplication. Since $\bfI_{tu}=\bfI_+$ and $\ga$ is topologically unipotent, for each $g\bfI'\in X_{\bfI',\ga}$ (i.e. $g^{-1}\ga g\in\bfI$), we have automatically $g^{-1}\ga g\in\bfI_+$. Hence $X_{\bfP',\ga}^+$ is a $\sB_{\bfP'}$-torsor over $X_{\bfI',\ga}$. Consider the associated adjoint bundle on $X_{\bfI',\ga}$ defined by 
    \[\cV_{\bfP',\ga}:=X_{\bfP',\ga}^+\times^{B_{\bfP'}}\mathrm{Lie}(\bfI_+/\bfP_+),\] 
    i.e. taking the quotient of $X_{\bfP',\ga}^+\times\mathrm{Lie}(\bfI_+/\bfP_+)$ by the diagonal action of $\sB_{\bfP'}=\bfI'/\bfP'_+$.\par
    Suppose that $\bfI_+/\bfP_+$ has dimension $1$, then $\cV_{\bfP',\ga}$ is a line bundle. We fix an isomorphism $\bfI_+/\bfP_+\cong\mathrm{Lie}(\bfI_+/\bfP_+)$. Then the natural map $X_{\bfP',\ga}^+\to\mathrm{Lie}(\bfI_+/\bfP_+)$ sending $g\bfP'_+$ to the image of $g^{-1}\ga g$ in $\mathrm{Lie}(\bfI_+/\bfP_+)$ descends to a global section of $\cV_{\bfP',\ga}$ which we denote by $v_{\bfP',\ga}$. Let $Z_{\bfP',\ga}\subset X_{\bfI',\ga}$ be the vanishing loci of $v_{\bfP',\ga}$. Then we have 
    \begin{equation}\label{eq:vanishing-loci}
        Z_{\bfP',\ga}=\pi_{\bfP'}^{-1}(\pi_{\bfP'}(Z_{\bfP',\ga}))=\pi_{\bfP',\ga}^{-1}(\pi_{\bfP',\ga}(Z_{\bfP',\ga}))
    \end{equation}
    where $\pi_{\bfP',\ga}:X_{\bfI',\ga}\to X_{\bfP',\ga}$ is the restriction of the natural projection $\pi_{\bfP'}:\mathrm{Fl}_{\bfI'}=LG'/\bfI'\to\mathrm{Fl}_{\bfP'}=LG'/\bfP'$. 
    Indeed, the second equality follows from the first one and the first one follows from the fact that $g^{-1}\ga g$ lies in $\bfP_+$ if and only if the (classical) Springer fiber of the image of $g^{-1}\ga g$ in $\sG_\bfP=\bfP/\bfP_+$ equals to the whole flag variety $\sG_\bfP/\sB_\bfP=\bfP/\bfI=\bfP'/\bfI'$. In other words, $Z_{\bfP',\ga}$ is the union of the fibers of $\pi_{\bfP'}$ that are completely contained in $X_{\bfI',\ga}$.\par 
    For each simple affine root $\alpha$ with corresponding simple reflection $s_\alpha\in\widetilde{W}$, let $\bfP'_\alpha=\bfI'\cup\bfI' s_\alpha\bfI'$ be the corresponding parahoric subgroup of $G'(F)$ and let $\bfP'_{\alpha,+}$ be its unipotent radical. The fibers of the natural projection $\pi_\alpha:\mathrm{Fl}_{\bfI'}\to\mathrm{Fl}_{\bfP'_\alpha}$ are all isomorphic to the perfections of $\bbP^1$ and we call them lines of type $\alpha$.\par 
    Let $g\bfI',g'\bfI'\in X_{\bfI',\ga}$ be two points in relative position $w\in\widetilde{W}$ so that $g^{-1}g'\in\bfI' w\bfI'$. Let $w=s_1\dotsm s_r$ be the reduced expression in simple reflections. For each $1\le i\le r$ let $\alpha_i$ be the simple affine root corresponding to the simple reflection $s_i$ and fix a representative $\dot{s}_i\in G'(F)$ of $s_i$. Then we may assume that $g^{-1}g'=\dot{s}_1\dotsm\dot{s}_r$. We claim that there exists a unique sequence of points $\{g_i\bfI',0\le i\le r\}$ in the affine flag variety $\mathrm{Fl}_{\bfI'}=LG'/\bfI'$ satisfying
    \begin{itemize} 
        \item $g_0=g$ and $g_n=g'=g\dot{s_1}\dotsm\dot{s_r}$,
        \item $g_{i-1}^{-1}g_i\in\bfI' s_i\bfI'$ for all $1\le i\le r$.
    \end{itemize}
    Indeed the sequence $g_1:=g_0\dot{s}_1, g_2:=g_1\dot{s}_2,\dotsm, g_r:=g_{r-1}\dot{s}_r$ clearly satisfy the above conditions and the uniqueness follows from a general property of Tits system, see for example \cite[5.1.3(h)]{Kumar}. Since $G'$ is simply connected we may identify $\mathrm{Fl}_{\bfI'}$ as a connected component of the affine flag variety for $G$ so that $\ga\in G(F)^0$ acts on it by left multiplication. Then the sequence $\{\gamma g_i\bfI',0\le i\le r\}$ also satisfies the above conditions since $\gamma g_0\bfI'=g_0\bfI'$ and $\gamma g_r\bfI'=g_r\bfI'$ by the assumption that $g_0\bfI',g_r\bfI'\in X_{\bfI',\ga}$. Hence by uniqueness we get $\gamma g_i\bfI'=g_i\bfI'$, i.e. $g_i\bfI'\in X_{\bfI',\ga}$ for all $0\le i\le r$. For each $1\le i\le r$, $\ell_i:=g_{i-1}\bfP_{\alpha_i}/\bfI$ is a line of type $\alpha_i$ passing through $g_{i-1}\bfI$ and $g_i\bfI$. The intersection $\ell_i\cap X_{\bfI',\ga}$ is the perfection of a classical Springer fiber for a nilpotent element of $\mathfrak{sl}_2$ and contains at least $2$ distinct points $g_{i-1}\bfI',g_i\bfI'$ as we have just seen. Since the Springer fiber for a nilpotent element of $\mathfrak{sl}_2$ is either a single point or the whole 
    flag variety $\bbP^1$, we get $\ell_i\subset X_{\bfI',\ga}$ for all $1\le i\le r$. Thus we have shown that any two points in $X_{\bfI',\ga}$ can be connected by a sequence of lines of type $\alpha$ in $X_{\bfI',\ga}$ for various affine simple roots $\alpha$. In particular, $X_{\bfI',\ga}$ is connected. \par 
    To prove the equi-dimensionality of $X_{\bfI',\ga}$, it suffices to show that for any irreducible component $Z\subset X_{\bfI',\ga}$ of dimension $d=\dim X_{\bfI',\ga}$  and any line $\ell_\alpha\subset X_{\bfI',\ga}$ of type $\alpha$ (for some simple affine root $\alpha$) such that $\ell_\alpha\cap Z\ne\varnothing$ and $\ell\nsubseteq Z$, there exists an irreducible component $Z'\subset X_{\bfI',\ga}$ of dimension $d$ such that $\ell_\alpha\subset Z'$. To prove the claim we consider the line bundle $\cL_\alpha:=\cV_{\bfP'_\alpha,\ga}$ on $X_{\bfI',\ga}$ constructed as above. By \eqref{eq:vanishing-loci} we have $\ell_\alpha\subset Z_{\bfP'_\alpha,\ga}$ and hence $Z$ is not contained in $Z_{\bfP'_\alpha,\ga}$ since $\ell_\alpha\nsubseteq Z$. Then the intersection $Z\cap Z_{\bfP'_\alpha,\ga}$ has pure dimension $d-1$. Let $Z_1$ be an irreducible component of $Z\cap Z_{\bfP'_\alpha,\ga}$ that has nonempty intersection with $\ell_\alpha$. Then $Z':=\pi_\alpha^{-1}(Z_1)$ is contained in $Z_{\bfP'_\alpha,\ga}$ by \eqref{eq:vanishing-loci}. Hence $Z'\subset X_{\bfI',\ga}$ and $\dim Z'=\dim Z_1+1=d$. Therefore $Z'$ is an irreducible component of $X_{\bfI',\ga}$ containing $\ell_\alpha$ and we are done.
\end{proof}

\subsection{Comparison between the group case and the Lie algebra case}
We will relate the affine Springer fibers for the group and the Lie algebra by a quasi-logarithm map. The comparison is based on the following simple observation:
\begin{lem}\label{lem:compare-group-Lie}
    Let $G$ be a reductive group over $F$ with Lie algebra $\fg$ and let $\bfP\subset G(F)$ be a parahoric subgroup. Suppose that there exists a $\bfP$-equivariant bijection $\Phi:\bfP_{tu}\xrightarrow{\sim}\mathrm{Lie}(\bfP)_{\mathrm{tn}}$, where
    \begin{itemize}
        \item $\bfP_{tu}$ is the set of topologically unipotent elements in $\bfP$ on which $\bfP$ acts by conjugation;
        \item $\mathrm{Lie}(\bfP)_{\mathrm{tn}}$ is the set of topologically nilpotent elements in the Lie algebra $\mathrm{Lie}(\bfP)$, on which $\bfP$ acts by the adjoint representation.
    \end{itemize}
    Then for any \emph{strongly} topologically unipotent element $\ga\in G(F)$ we have an isomorphism between the centralizers $G_\delta\cong G_{\Phi(\delta)}$ and an identification between Witt vector affine Springer fibers $X_{\bfP,\gamma}^G=X_{\bfP,\Phi(\gamma)}^\fg$, i.e. they are the same subspace of the Witt-vector affine partial flag variety $\mathrm{Fl}_\bfP$.
\end{lem}
\begin{proof}
    Since both spaces are closed in the perfect ind-scheme $\mathrm{Fl}_\bfP$, it suffices to check that they have the same set of $k$-points. For any $g\in G(F)$ we have the following equivalences
    \[g^{-1}\ga g\in\bfP\Leftrightarrow g^{-1}\ga g\in\bfP_{tu}\Leftrightarrow\mathrm{ad}(g)^{-1}\Phi(\ga)\in\mathrm{Lie}(\bfP)_{tn}\Leftrightarrow\mathrm{ad}(g)^{-1}\Phi(\ga)\in\mathrm{Lie}(\bfP)\]
    where the middle equivalence follows from the assumptions. This shows that $X_{\bfP,\ga}^G=X_{\bfP,\Phi(\ga)}^\fg$.\par 
    After $G(F)$-conjugation we may assume that $\ga\in\bfP_{tu}$ and then $\Phi$ induces an isomorphism between the centralizers $G_\delta\cong G_{\Phi(\delta)}$.
\end{proof}
One also expects that the discriminant valuation functions $d_G$ on $G(F)$ and $d_\fg$ on $\fg(F)$ should be compatible with a quasi-logarithm map. This is indeed the case, at least for strongly topologically unipotent elements. Although there might be a straightforward proof of this fact, we will deduce it from Lemma \ref{lem:compare-group-Lie} and special cases of the main theorems later (see Corollary \ref{cor:disc-group-Lie-alg}).

\section{Dimension of affine Springer fibers for the Lie algebras}\label{sec:dim-Lie-alg}
In this section we establish the dimension formula of the Witt-vector affine Springer fibers for the Lie algebras and finish the proof of Theorem \ref{thm:main-Lie-algebra-case}.

\subsection{Further assumptions on the group}\label{sec:more-assumption}
By Corollary \ref{cor:reduction-Lie-alg} we may assume that $G=G'$ and we have the freedom of replacing $G$ by any reductive group over $F$ whose adjoint group is isomorphic to $G_{\mathrm{ad}}$. Meanwhile, the adjoint group $G_\mathrm{ad}$ is a product of $F$-simple groups and each simple factor of $G_\mathrm{ad}$ is the Weil restriction of an absolutely simple group (i.e. whose base change to $\overline{F}$ is simple). Let us examine the effect of Weil restrictions on the numerical invariants in Theorem \ref{thm:main-Lie-algebra-case}. Suppose $G=\mathrm{Res}_{\widetilde{F}/F}\widetilde{G}$ where $\widetilde{F}/F$ is a finite field extension. From the definition we easily see that the discriminant valuation $d_\fg(\ga)$ does not change if we replace $F$ (resp. $G$) by $\widetilde{F}$ (resp. $\widetilde{G}$). By Proposition \ref{prop:Artin-conductor-Res}, the same is true for the difference $\mathrm{Art}(G)-\mathrm{Art}(G_\ga)$. Also it is clear that all these invariants are additive when we decompose $G$ into products of simple factors. Then by Lemma \ref{lem:affine-flag-decompose} we may and do assume that $G_\mathrm{ad}$ is absolutely simple. Furthermore, we assume that $G$ splits over a \emph{tamely ramified} extension $F'/F$ of degree $e\ge1$, $p\nmid e$. We use the notations and definitions from \S\ref{sec:parahoric} and \S\ref{sec:group-model}. In particular, $G$ is an outer twist of its split form $\bbG$ defined by a homomorphism $\rho_G:\mu_e\to\mathrm{Out}(\bbG)$. 

By Lemma \ref{lem:pairing-split-case} and Corollary \ref{cor:Kostant-section}, after replacing $G$ by a group $G^\natural$ with the same adjoint group, we may assume that all of the following conditions are satisfied:
\begin{itemize}
    \item $G$ is a tamely ramified, $G_{\mathrm{ad}}$ is absolutely simple and $p$ is good for $G$.
    \item There exists a $\bbG\rtimes\mathrm{Out}(\bbG)$-invariant perfect symmetric bilinear form $(\cdot,\cdot)$ on the $\cO$-module $\bbg(\cO)$.
    \item The regular centralizer $J$ for $G$ is smooth over the $\cO$-scheme $\fc=(\mathrm{Res}_{\cO'/\cO}\bbc)^{\mu_e}$ and there exists a section $\kappa:\fc\to\fg$ (of the Chevalley morphism $\chi_G$) between $\cO$-schemes.
\end{itemize}
We will impose these assumptions until the end of \S\ref{sec:dim-Lie-alg}.\par
Let $\ga\in\fg(F)$ be a bounded and regular semisimple element. By Corollary \ref{cor:chi-surjection}, $\ga$ is $G(\overline{F})$-conjugate to an element in $\fg^{\mathrm{reg}}(\cO)$ since $\chi(\ga)\in\fc(\cO)$. Since $G_\ga$ is a torus over $F$ we have $H^1(F,G_\ga)=1$ by Steinberg's theorem. This implies that $\ga$ is $G(F)$-conjugate to an element in $\fg^{\mathrm{reg}}(\cO)$.\par 
We assume for the rest of this section that $\ga\in\fg^{\mathrm{reg}}(\cO)\cap\fg^{\mathrm{rs}}(F)$. Let $G_\ga$ be the $\cO$-group scheme of centralizers of $\ga$. Let $a:=\chi(\ga)\in\fc(F)^{\mathrm{rs}}\cap\fc(\cO)$. Let $J_a$ be the pullback of the regular centralizer $J$ along $a$, viewed as a morphism $\spec(\cO)\to\fc$. By Proposition \ref{prop:reg-centralizer} we have a canonical isomorphism $G_\ga\cong J_a$ and then by the assumptions above $G_\ga$ is a smooth commutative group scheme over $\cO$ whose generic fiber is a torus. Then we have the Witt-vector loop group $LG_\ga$, the positive loop group $L^+G_\ga$ and their quotient 
\[P_\ga:=\mathrm{Gr}_{G_\ga}=LG_\ga/L^+G_\ga.\]

\subsection{Centralizer action}
The loop group $LG_\ga$ acts on $X^\fg_{\bfP,\ga}$ and we will show that the action factors through the quotient $P_\ga$. 
\begin{lem}
    The positive loop group $L^+G_\ga$ acts trivially on $X^\fg_{\bfP,\ga}$. 
\end{lem}
\begin{proof}
    We may assume that $\bfP\supset\bfI$. Since the natural morphism $Y_\ga=X_{\bfI,\ga}^\fg\to X_{\bfP,\ga}^\fg$ is surjective and equivariant for the $LG_\ga$-action, it suffices to prove the claim for $Y_\ga$. \par 
    Let $g\bfI\in Y_\ga$ so that $\delta:=\mathrm{ad}(g)^{-1}\ga\in\mathrm{Lie}(\bfI)\subset\bbg(\cO')$. Since $\ga\in\fg^{\mathrm{reg}}(\cO)$, the restriction of the homomorphism $\chi^*\bbJ\to\bbI$ to $\delta\in\bbg(\cO')$ induces the inclusion 
    \[\mathrm{ad}(g)^{-1}G_\ga(\cO)\subset\bbG_{\delta}(\cO')\subset\bbG(\cO').\]
    We need to show furthermore that $\mathrm{ad}(g)^{-1}G_\ga(\cO)\subset\bfI$. Let $\bar{\delta}\in\bbg(k)$ be the image of $\delta$
    under the reduction map $\bbg(\cO')\to\bbg(k)$. It suffices to show that the image of $\bbG_\ga(k)$ under $\mathrm{ad}(g)^{-1}$ lies in $\bbB(k)$. This follows from \cite[Lemma 2.3.1]{Yun-GS}, which says that the restriction of the homomorphism $\chi^*\bbJ\to\bbI$ (see Proposition \ref{prop:reg-centralizer}) to $\mathrm{Lie}(\bbB)$ factors through the universal centralizer of elements of $\mathrm{Lie}(\bbB)$ in $\bbB$.    
\end{proof}
As a consequence, the action of $LG_\ga$ on $X_{\bfP,\ga}$ factors through the quotient $P_\ga=LG_\ga/L^+G_\ga$.\par 
Let us describe the $LG_\ga$-orbits on $X_{\bfP,\ga}$ in more detail. For any $\ga\in\mathrm{Lie}(\bfP)\cap\fg(F)^{\mathrm{rs}}$. Let $\cG_{\bfP,\ga}$ be the group smoothening of the schematic closure of $G_\ga$ in the Bruhat-Tits group scheme $\cG_\bfP$. Then $\cG_{\bfP,\ga}$ is a smooth commutative group scheme over $\cO$ whose generic fiber is isomorphic to the torus $G_\ga$ and whose group of $\cO$-points is $\cG_{\bfP,\ga}(\cO)=G_\ga(F)\cap\bfP$. Let $\cG_\ga^\dagger$ be the finite type Neron model of $G_\ga$. Then $\cG_\ga^\dagger(\cO)$ is the unique maximal bounded subgroup of $G_\ga(F)$ and the quotient $LG_\ga/L^+\cG_\ga^\dagger$ is discrete.\par We have natural inclusions of $\cO$-lattices in $\mathrm{Lie}(G_\ga)=\fg_\ga(F)$: 
\[\mathrm{Lie}(\cG_{\bfP,\ga})\subset\fg_\ga(F)\cap\mathrm{Lie}(\bfP)\subset\mathrm{Lie}(\cG_\ga^\dagger)\]
in which the first inclusion is an equality if and only if the schemetic closure of $G_\ga(F)$ in $\cG_\bfP$ is smooth. Indeed,  by \cite[Lemma A.2.3]{KP23} the intersection $\fg_\ga(F)\cap\mathrm{Lie}(\bfP)$ is the Lie algebra of the schematic closure of $G_\ga$ in $\cG_\bfP$.

\begin{lem}\label{lem:orbit-dim}
     For any regular semisimple element $\ga\in\fg(F)^{\mathrm{rs}}$ and any $g\in X_{\bfP,\ga}$, let $\ga_g:=\mathrm{ad}(g)^{-1}\ga\in\mathrm{Lie}\bfP$. Then the $LG_\ga$-orbit of $g\in X_{\bfP,\ga}$ is isomorphic to $\mathrm{Gr}_{\cG_{\bfP,\ga_g}}=LG_\ga/L^+\cG_{\bfP,\ga_g}$, and is represented by a perfect $k$-scheme locally perfectly of finite type of dimension
    \[\dim\mathrm{Gr}_{\cG_{\bfP,\ga_g}}=\mathrm{length}_{\cO}\left(\frac{\mathrm{Lie}(\cG^\dagger_{\ga_g})}{\mathrm{Lie}(\cG_{\bfP,\ga_g})}\right).\]
\end{lem}
\begin{proof}
    To ease notation we may and do assume that $g=1$ so that $\ga=\ga_g\in\mathrm{Lie}\bfP$. Note that $\pi_0(LG_\ga/L^+\cG^\dagger_{\ga})$ is a free abelian group of finite rank (which is the free part of $\pi_0(LG_\ga)$). Therefore it is enough to study the $L^+\cG^\dagger_{\ga}$-orbit of $1$ and the equality follows from \cite[Lemma 2.6(c)]{LLM}.
\end{proof}
Our next goal is to study the $P_\ga$-orbits on $X_{\bfP,\ga}$ of maximal dimension.

\subsection{The regular locus}
Let $\bfP\subset G(F)$ be a parahoric subgroup. After $G(F)$-conjugation we may and do assume that $\bfI\subset\bfP$. Let $\sG_\bfP$ be the reductive quotient of the special fiber of the associated Bruhat-Tits group scheme $\cG_\bfP$. The image of $\bfI$ in $\sG_\bfP$ under the reduction morphism is a Borel subgroup that we denote by $\sB_\bfP$. Let $\fg_\bfP=\mathrm{Lie}(\sG_\bfP)$ and $\fb_\bfP=\mathrm{Lie}(\sB_\bfP)$ be the Lie algebras. Let $\fg_\bfP^{\mathrm{reg}}\subset\fg_\bfP$ be the open subset of regular elements, i.e. those whose centralizer has minimal possible dimension. Let $\mathrm{Lie}(\bfP)^{\mathrm{reg}}$ be the inverse image of $\fg_\bfP^{\mathrm{reg}}$ under the natural reduction morphism. Then it is an open subscheme of $\mathrm{Lie}(\bfP)$. 
\begin{defn}
    The \emph{regular locus} of the Witt-vector affine Springer fiber $X_{\bfP,\ga}$ is the subspace defined by
    \[X_{\bfP,\ga}^{\mathrm{reg}}=\{g\in LG/\bfP|\mathrm{ad}(g)^{-1}(\ga)\in\mathrm{Lie}(\bfP)^{\mathrm{reg}}\}.\]
\end{defn}
By Lemma~\ref{lem:open-subspace} we see that $X_{\bfP,\ga}^{\mathrm{reg}}$ is an open subspace of $X_{\bfP,\ga}$. We expect that $X_{\bfP,\ga}^{\mathrm{reg}}$ is always nonempty and we can see this immediately when $\bfP=\bfP_0=G(\cO)$ is the special parahoric subgroup fixed in \S\ref{sec:group-model}, in which case we simply denote 
\[X_\ga:=X_{G(\cO),\ga}^\fg=\{g\in G(F)/G(\cO)|\mathrm{ad}(g)^{-1}(\ga)\in\fg(\cO)\}.\]
By the assumptions in \S\ref{sec:more-assumption} we have $\ga\in\fg^{\mathrm{reg}}(\cO)$ so that the base point $1\in\mathrm{Gr}_G$ lies in $X_\ga^{\mathrm{reg}}$ and in particular $X_\ga^{\mathrm{reg}}$ is non-empty.\par 
Recall that $a:=\chi_G(\ga)$. In the following, we let $\fg(\cO)_a:=\chi_G^{-1}(a)\cap\fg(\cO)$ and $\fg^{\mathrm{reg}}(\cO)_a:=\chi_G^{-1}(a)\cap\fg^{\mathrm{reg}}(\cO)$. The following result is a coarse space version of \cite[Lemma 2.4]{BFN-Springer} and \cite[Lemma 7.5]{EKO}.
\begin{lem}\label{lem:orbits-bijection}
    There is a natural bijection between the following sets
    \[\{G_\ga(F)\text{ orbits on }X_\ga\}\longleftrightarrow\{G(\cO)\text{ orbits on }\fg(\cO)_a\}\]
     which restricts to a bijection between their subsets
     \[\{G_\ga(F)\text{ orbits on }X_\ga^{\mathrm{reg}}\}\longleftrightarrow\{G(\cO)\text{ orbits on }\fg^{\mathrm{reg}}(\cO)_a\}.\]
\end{lem}
\begin{proof}
    The map is defined by sending the $G_\ga(F)$-orbit of $g\in X_\ga$ to the $G(\cO)$-orbit of $\mathrm{ad}(g)^{-1}(\ga)$, which clearly is injective. For any $\ga'\in\fg(\cO)_a^{\mathrm{reg}}$, since $\ga',\ga\in\fg(F)^{\mathrm{rs}}$ are both generically regular semisimple and $\chi_G(\ga)=\chi_G(\ga')$, there exists $g\in G(F)$ such that $\mathrm{ad}(g)^{-1}(\ga)=\ga'$. Then $g G(\cO)\in X_\ga$ and its $G_\ga(F)$-orbit is mapped to the $G(\cO)$-orbit of $\ga'$. This shows surjectivity. The second statement follows directly from the definition of regular locus. 
\end{proof}

\begin{lem}\label{lem:orbits-cohomology}
    There is a bijection between the set of $\bfP_0^\dagger=\bbG(\cO')^{\mu_e}$-orbits in $\fg^{\mathrm{reg}}(\cO)_a$ and the set
    \[\ker(H^1(\mu_e,\bbG_\ga(\cO'))\to H^1(\mu_e,\bbG(\cO'))).\]
    Here $\bbG_\ga$ is the centralizer group scheme (over $\cO'$) of $\ga\in\fg^{\mathrm{reg}}(\cO)\subset\bbg^{\mathrm{reg}}(\cO')$.\par 
    If $G$ is split so that $\bfP_0=G(\cO)$ is hyperspecial, then $\fg^{\mathrm{reg}}(\cO)_\ga$ consists of a single $G(\cO)$-orbit. 
\end{lem}
\begin{proof}
    This essentially follows from the smoothness of the universal regular centralizer $\bbJ$ over $\bbc$ and the fact that $[\bbg^{\mathrm{reg}}/\bbG]$ is a $B\bbJ$-gerbe neutralized by a Kostant section. Let us explain in more concrete terms. \par
    Let $\bbJ_a$ be the pull-back of $\bbJ$ along $a=\chi_G(\ga)\in\fc(\cO)\subset\bbc(\cO')$, viewed as a morphism $\spec\cO'\to\bbc$.
    By Proposition \ref{prop:reg-centralizer} we have $\bbG_\ga\cong\bbJ_a$ and hence $\bbG_\ga$ is smooth over $\cO'$ by the assumptions in \S\ref{sec:more-assumption}.\par
    Let $\bbg^{\mathrm{reg}}_a$ be the $\cO'$-scheme defined by the fiber product 
    \[\xymatrix{
    \bbg^{\mathrm{reg}}_a\ar[r]\ar[d] & \bbg^{\mathrm{reg}}\ar[d]^{\chi_\bbG} \\
    \spec\cO'\ar[r]^{a} & \bbc
    }\]
    Then the morphism of $\cO'$-schemes $f_\ga:\bbG\to\bbg_a^{\mathrm{reg}}$ sending $g$ to $\mathrm{ad}(g)^{-1}\ga$ induces an isomorphism $\bbG/\bbG_\ga\cong\bbg_a^{\mathrm{reg}}$. Since the special fiber $\bbg_a^{\mathrm{reg}}(k)$ consists of a single $\bbG(k)$-orbit, for any $\ga'\in\bbg_a^{\mathrm{reg}}(\cO')$ there exists $\Bar{g}\in\bbG(k)$ such that $\mathrm{ad}(\Bar{g})(\bar{\ga})=\bar{\ga}'$, where $\Bar{\ga}$ (resp. $\Bar{\ga}'$) denotes the reduction of $\ga$ (resp. $\ga'$) in $\bbg^{\mathrm{reg}}(k)$. Since $\bbG_\ga$ is smooth, the morphism $f_\ga$ is smooth and therefore we can lift $\Bar{g}\in\bbG(k)$ to $g\in\bbG(\cO')$ such that $\mathrm{ad}(g)(\ga)=\ga'$.\par
    Now suppose $\ga'\in\fg^{\mathrm{reg}}(\cO)_a=\bbg_a^{\mathrm{reg}}(\cO')\cap\fg(F)$ and we find as above an element $g\in\bbG(\cO')$ such that $\mathrm{ad}(g)^{-1}(\ga)=\ga'$. Let $\sigma\in\mu_e$ be a generator. Then $\sigma(\ga')=\ga'$ and $\sigma(\ga)=\ga$, from which we get that $g\sigma(g)^{-1}\in\bbG_{\ga}(\cO')$. This defines a cocycle of $\mu_e$ in $\bbG_\ga(\cO')$ that becomes a coboundary in $\bbG(\cO')$. Clearly $g\sigma(g)^{-1}$ only depends on the $\bbG(\cO')^{\mu_e}$-orbit of $\ga'$. For another choice $g'\in\bbG(\cO')$ such that $\mathrm{ad}(g')^{-1}(\ga)=\ga'$ we have $g'g^{-1}\in\bbG_\ga(\cO')$ and hence $g\sigma(g)^{-1}$ and $g'\sigma(g')^{-1}$ differ by a coboundary in $\bbG_\ga(\cO')$. Thus the cohomology class of $g\sigma(g)^{-1}$ is well-fined and we obtain a map from the set of $\bbG(\cO')^{\mu_e}$-orbits in $\fg^{\mathrm{reg}}(\cO)_a$ to the cohomology set
    \[\ker(H^1(\mu_e,\bbG_\ga(\cO'))\to H^1(\mu_e,\bbG(\cO'))).\]
    Then one checks immediately that this map is bijective. When $G$ is split we have $e=1$ and $\bfP_0^\dagger=\bfP_0=G(\cO)$, so the cohomology set is trivial.
\end{proof}

\begin{cor}\label{cor:reg-locus-description}
    The set of $LG_\ga$-orbits in $X_\ga^{\mathrm{reg}}$ is finite nonempty and admits a surjection to the set
    \[\ker(H^1(\mu_e,\bbG_\ga(\cO'))\to H^1(\mu_e,\bbG(\cO'))).\]
    Moreover, each $LG_\ga$-orbit in $X_\ga^{\mathrm{reg}}$ is a $P_\ga$-torsor and is open in $X_\ga$. \par 
    If $G$ is split so that $G(\cO)$ is hyperspecial, $X_\ga^{\mathrm{reg}}$ is a $P_\ga$-torsor (and hence form a single $LG_\ga$-orbit). 
\end{cor}
\begin{proof}
    The first statement follows by combinning Lemma \ref{lem:orbits-bijection} and Lemma \ref{lem:orbits-cohomology}. Now we prove the second statement. Note that $G_\ga(\cO)=\bbG_\ga(\cO')\cap G(\cO)$.
    Let $g\in G(F)$ whose coset $gG(\cO)$ determines a point in $X_\ga^{\mathrm{reg}}$. Then $\ga':=\mathrm{ad}(g)^{-1}\ga$ lies in $\fg^{\mathrm{reg}}(\cO)$. To show that the $LG_\ga$-orbit through $gG(\cO)$ is a torsor under $P_\ga$, we need to prove that the stabilizer of $g G(\cO)$ in $G_\ga(F)$ coincides with $G_\ga(\cO)$. By definition, this stablizer equals to $G_\ga(F)\cap gG(\cO)g^{-1}$. We note that $\mathrm{ad}(g)^{-1}$ induces a canonical isomorphism $\bbG_\ga(\cO')\cong\bbG_{\ga'}(\cO')=\bbG_{\ga'}(F')\cap\bbG(\cO')$. Intersecting with $G(F)$ we see that $\mathrm{ad}(g)^{-1}G_\ga(\cO)=G_{\ga'}(F)\cap G(\cO)$ and hence
    \[G_\ga(\cO)=\mathrm{ad}(g)(G_{\ga'}(F)\cap G(\cO))=G_\ga(F)\cap gG(\cO)g^{-1}\] 
    as desired. Since $X_\ga^{\mathrm{reg}}$ is open and a finite union of $LG_\ga$-orbits, each such orbit is also open. When $G$ is split there is only one such orbit. 
\end{proof}
Let us provide two simple examples of ramified groups to illustrate the description of regular orbits. 
\begin{ex}
    Suppose $p>2$ and let $F'/F$ be a tamely ramified quadratic extension. We consider the associated special unitary group $G=\mathrm{SU}_n$ whose split form is $\bbG=\mathrm{SL}_n$.  Let $\varpi'\in\cO'$ be a uniformizer with $\varpi:=(\varpi')^2\in\cO$. For a square matrix $A$ we denote ${}^\iota A:=J\;{}^tAJ^{-1}$ where $J$ is the square matrix whose anti-diagonal entries are $1$ and all other entries are $0$. The group $G$ is the twisted form of $\bbG$ associated to the involution $\theta(g):={}^\iota g^{-1}$ of $\bbG$. This is an outer automorphism if $n>2$, and if $n=2$ then $\theta(g)=\mathrm{ad}\bm{1&0\\0&-1}(g)$ is an inner automorphism. Let $\sigma\in\mathrm{Gal}(F'/F)$ be the nontrivial element and let $\sigma_G$ denotes the Galois action (with respect to $G$) on $G(F')=\mathrm{SL}_n(F')$. It is given by the formula
    \[\sigma_G(g)=\theta(\bar{g})={}^\iota\bar{g}^{-1},\quad\forall g\in\mathrm{SL}_n(F')\]
    where $\bar{g}$ is obtained from the matrix $g$ by applying $\sigma$ on each entry. We have $\cG_0=\cG_0^\dagger$ and for any $\cO$-algebra $R$ with $R':=R\otimes_\cO\cO'$, 
    \[\cG_0(R)=\{g\in\mathrm{SL}_n(R'),{}^\iota\Bar{g}g=1\}.\]
    The reductive quotient of the special fiber of $\cG_0$ is $\bbG^\theta\cong\mathrm{SO}_n$. The Lie algebra of $\bbG$ decomposes into eigen-spaces of $\theta$ as $\bbg=\bbg_0\oplus\bbg_1$ where $\bbg_0$ (resp. $\bbg_1$) consists of trace zero anti-symmetric (resp. symmetric, with respect to the bilinear form defined by $J$) matrices. Then we have $\fg(\cO)=\bbg_0(\cO)\oplus\varpi'\bbg_1(\cO)$ and $\fg^{\mathrm{reg}}(\cO)=(\bbg_0\cap\bbg^{\mathrm{reg}})(\cO)\oplus\varpi'\bbg_1(\cO)$. \par 
    Take an element $\ga\in\fg^{\mathrm{reg}}(\cO)\cap\ft(F)$. Then $G_\ga=\cT$ as group scheme over $\cO$ and after base change to $\cO'$ we have $\bbG_\ga=\bbT$. An element $t=\mathrm{diag}(t_1,\dotsc,t_n)\in\bbT(\cO')$ with $\prod_{i=1}^nt_i=1$ defines a cocycle in $Z^1(\sigma_G,\bbT(\cO'))$ if $t\sigma_G(t)=1$, or equivalently $t_i=\bar{t}_{n+1-i}$ for all $1\le i\le n$. It is a co-boundary if there exists $s=\mathrm{diag}(s_1,\dotsc,s_n)\in\bbT(\cO')$ with $\prod_{i=1}^ns_i=1$ such that $s\cdot\sigma_G(s)^{-1}=t$, or equivalently $t_i=s_i\bar{s}_{n+1-i}$ for all $1\le i\le n$. When $n=2m+1$ is odd we have $H^1(\sigma_G,\bbT(\cO'))=1$. Indeed, for a one cocycle represented by $t\in\bbT(\cO')$ as above, one can take $s_{m+1}=(\bar{t}_1\dotsm\bar{t}_m)^{-1}$, $s_1=\dotsm=s_m=1$ and for $i\ge m+2$, $s_i:=\bar{t}_{n+1-i}$. Then one checks that $\prod_{i=1}^{n}s_i=1$ and $t=s\cdot\sigma_G(s)^{-1}$. \par
    Now assume that $n=2m$ is even. Suppose $t=(t_1,\dotsc,t_m,\bar{t}_m,\dotsc,\bar{t}_1)\in\bbT(\cO')$ represents a cocycle in $Z^1(\sigma_G,\bbT(\cO'))$. Let $u:=t_1\dotsm t_m\in(\cO')^\times$. Then we have $u\bar{u}=1$ (so $u\equiv\pm1\mod\varpi'$) and one checks that $t$ is a co-boundary if and only if there exists $v\in(\cO')^\times$ such that $u=v\bar{v}^{-1}$, if and only if $u\equiv1(\mod\varpi')$. Indeed, suppose that such $v$ exists, let $s_1=\dotsm=s_{m-1}=1$, $s_m=v$, $s_{m+1}=\bar{t}_m/\bar{v}$, and $s_i=\bar{t}_{2m+1-i}$ for all $i\ge m+2$. Then we have $t_i=s_i\bar{s}_{n+1-i}$ for all $1\le i\le n$. Thus if $n=2m$ we have $H^1(\sigma_G,\bbT(\cO'))=\bbZ/2\bbZ$. Now suppose moreover that $n=2m$ where $m$ is odd. Suppose the element $t=\mathrm{diag}(t_1,\dotsc,t_m,\bar{t}_m,\dotsc,\bar{t}_1)\in\bbT(\cO')$ as above represents a nontrivial cohomology class. Then $u=t_1\dotsm t_m\equiv-1(\mod\varpi')$ and there exists $v\in(\cO')^\times$ with $u=-v\bar{v}^{-1}$. Take an element $\alpha\in(\cO')^\times$ with $\alpha^m=v^{-1}$. Let $g=\bm{0&B\\C&0}$ be the $n\times n$ matrix where $B=\alpha\cdot\mathrm{diag(t_1,\dotsc,t_m)}$ and $C=\bar{\alpha}^{-1}\cdot\mathrm{Id}_m$. Then one verifies that $g\in\bbG(\cO')=\mathrm{SL}_n(\cO')$ and $g\cdot\sigma_G(g)^{-1}=t$. This shows that $\ker(H^1(\sigma_G,\bbT(\cO'))\to H^1(\sigma_G,\bbG(\cO')))=\bbZ/2\bbZ$. \par 
    The case where $m=1$ and $n=2$ is exotic since we have $G=\mathrm{SU}_2\cong\mathrm{SL}_2$. In fact, conjugation by $\bm{0&\varpi'\\1&0}$ identifies $G(F)$ with $\mathrm{SL}_2(F)$ and takes $\bfP_0=\cG_0(\cO)$ to an Iwahori subgroup of $\mathrm{SL}_2(F)$. So $\bfP_0$ is not a special parahoric in this case (there is no contradiction with \S\ref{sec:group-model} since $F'$ is not the \emph{minimal} extension of $F$ that splits $G$ in this exotic case). However the discussion above is still valid. Consider the element $\ga=\bm{1&0\\0&-1}\in\fg^{\mathrm{reg}}(\cO)$ with $a=\chi_G(\ga)=-1\in\fc(\cO)$. Then there are two $\bfP_0$ orbits in $\fg^{\mathrm{reg}}(\cO)_a$ and it is easy to see that $\ga$ and $-\ga$ belong to different $\bfP_0$-orbits. 
\end{ex}
\begin{ex}
    Let $F'/F$ be as above and take $G=\mathrm{U}_2$ whose split form is $\bbG=\mathrm{GL}_2$. We still take $\ga=\bm{1&0\\0&-1}$. This time $\bfP_0$ has index two in $\bfP_0^\dagger$ and $\bfP_0^\dagger$ acts transitively on $\fg^{\mathrm{reg}}(\cO)_a$ but $\bfP_0$ has two orbits.
\end{ex}

\subsection{Dimension of the regular locus}                                 
From Corollary \ref{cor:reg-locus-description} and Theorem \ref{thm:finiteness}, we deduce that $P_\ga$ is a finite dimensional commutative $k$-group. In this section we will compute its dimension. \par
Let $G_\ga^\dagger$ be the finite type N\'eron model of the generic fiber of $G_\ga$. Then $G_\ga^\dagger(\cO)$ is the maximal bounded subgroup of $G_\ga(F)$. The $\cO$-modules $\mathrm{Lie}(G_\ga)$ (resp. $\mathrm{Lie}(G_\ga^\dagger)$) are identified with $\cO$-lattices in $\fg_\ga(F)$ that we denote by $\fg_\ga$ (resp. $\fg_\ga^\dagger$).

\begin{thm}\label{thm:dim-reg}
    With notations as above, we have 
    \[\dim X_\ga^{\mathrm{reg}}=\mathrm{length}_\cO(\fg_\ga^\dagger/\fg_\ga)=\frac{1}{2}(d_\fg(\ga)+\mathrm{Art}(G)-\mathrm{Art}_\ga).\]
\end{thm}
By Lemma \ref{lem:orbit-dim} we have
\[\dim X_\ga^{\mathrm{reg}}=\mathrm{length}_\cO\mathrm{Lie}(G_\ga^\dagger)/\mathrm{Lie}(G_\ga)=\mathrm{length}_\cO(\fg_\ga^\dagger/\fg_\ga).\]
It remains to compute the length of the $\cO$-module $\fg_\ga^\dagger/\fg_\ga$. In the equal-characteristic setting, this computation is done in \cite[Appendix]{OY} (which generalizes results in \cite{KL} and \cite{Be96}) and we will follow the arguments therein.\par 
Since $G_\ga$ is smooth over $\cO$ and \emph{a fortiori} flat over $\cO$, it is the schematic closure of its generic fiber in $G=\cG_0$ and hence $\fg_\ga=\fg_\ga(F)\cap\fg(\cO)$ by \cite[Lemma A.2.3]{KP23}. Let $\fp(F):=[\ga,\fg(F)]$ be the image of the linear endomorphism $\mathrm{ad}(\ga)$ on $\fg(F)$. Since $\ga$ is semisimple, we have $\fp(F)\cap\fg_\ga(F)=0$ and hence a decomposition $\fg(F)=\fg_\ga(F)\oplus\fp(F)$.

Recall from our assumption in \S\ref{sec:more-assumption} that there is a $\bbG(F)\rtimes\mathrm{Out}(\bbG)$-invariant non-degenerate symmetric $F$-bilinear form $(\cdot,\cdot)$ on the split Lie algebra $\bbg(F)$ that is nondegenerate over $\cO$. Extending scalars we get an $F'$-bilinear form on $\bbg(F')$ that is $\bbG(F')\rtimes\mathrm{Out}(\bbG)$-invariant. Thus its restriction to $\fg(F)=\bbg(F')^{\mu_e}$ is an $G(F)$-invariant $F$-bilinear form that takes values in $F$. In particular for any $x,y\in\fg(F)$ we have
\[(\mathrm{ad}(\ga)(x),y)=-(x,\mathrm{ad}(\ga)(y))\]
which implies that $\fg_\ga(F)$ and $\fp(F)$ are the orthogonal complements of each other and hence $(\cdot,\cdot)$ restricts to non-degenerate symmetric bilinear forms on $\fg_\ga(F)$ and $\fp(F)$. Let $V$ denote any of the $F$-vector spaces $\fg(F)$, $\fg_\ga(F)$ or $\fp(F)$. For an $\cO$-lattice $L\subset V$, let $L^\vee:=\{v\in V, (v,L)\subset\cO\}$ be its dual lattice. \par
We have a sequence of inclusions $\fg_\ga\subset\fg_\ga^\dagger\subset\fg_\ga^{\dagger,\vee}\subset\fg_\ga^\vee$ and hence
\begin{equation}\label{eq:length-Lie-alg}
    \mathrm{length}_\cO(\fg_\ga^\dagger/\fg_\ga)=\frac{1}{2}\mathrm{length}_\cO(\fg_\ga^\vee/\fg_\ga)-\frac{1}{2}\mathrm{length}_\cO(\fg_\ga^{\dagger,\vee}/\fg_\ga^\dagger)
\end{equation}

The action of $\rho_G(\mu_e)$ on the $\cO$-module $\bbg$ induces a decomposition $\bbg=\oplus_{i=0}^{e-1}\bbg_i$ where $\bbg_i$ is the $\cO$-submodule on which $\zeta\in\mu_e$ acts by scalar multiplication by $\zeta^i$. Then we have
\[\fg(\cO)=\bbg_0(\cO)\oplus\bigoplus_{i=1}^{e-1}(\varpi')^{e-i}\bbg_i(\cO)\]
and hence $\fg(\cO)^\vee=\bigoplus_{i=0}^{e-1}(\varpi')^{-i}\bbg_i(\cO)$ by the fact that $(\cdot,\cdot)$ is $\mathrm{Out}(\bbG)$-invariant and non-degenerate over $\cO'$.\par
For each $0\le j\le e-1$, let $\fg(\cO)_{\le j}:=\fg(\cO)+\sum_{1\le i\le j}(\varpi')^{-i}\bbg_i(\cO)$. Then we get a chain of $\cO$-modules
\[\fg(\cO)=\fg(\cO)_{\le0}\subset\fg(\cO)_{\le1}\subset\dotsm\subset\fg(\cO)_{\le e-1}=\fg(\cO)^\vee\]
whose successive quotients are $\fg(\cO)_{\le j}/\fg(\cO)_{\le j-1}=\bbg_j(k)$. From the definition we see that each $\fg(\cO)_{\le j}$ is preserved by adjoint action of elements in $\fg(\cO)$. We record the following result from \cite[Appendix]{OY}.
\begin{lem}\label{lem:lattice-saturated}
For each $0\le j\le e-1$, the $\cO$-lattice $[\ga,\fg(\cO)_{\le j}]$ in $\fp(F)$ is saturated in $\fg(\cO)_{\le j}$. In other words, we have $[\ga,\fg(\cO)_{\le j}]=\fg(\cO)_{\le j}\cap\fp(F)$. As a special case we get that 
\[[\ga,\fg(\cO)^\vee]=\fg(\cO)^\vee\cap\fp(F).\]
\end{lem}
\begin{proof}
    Let $Q$ be the cokernel of the $\cO$-linear endomorphism $\mathrm{ad}(\ga)$ of $\fg(\cO)_{\le j}$. We need to show that $Q$ is torsion free, which would follow from the inequality $\dim_k Q\otimes_\cO k\le\dim_F Q\otimes_\cO F$ (the fact that $Q$ is a finitely generated $\cO$-module implies the reverse inequality). Since $\ga\in\fg(F)$ is regular semisimple, we have $\dim_F Q\otimes_\cO F=\dim_F\fg_\ga(F)=\bbr$, which is the absolute rank of $\fg$ (i.e. the rank of the split form $\bbg$), we
    are reduced to showing that $\dim_k Q\otimes_\cO k\le\bbr$.\par
    We can write $\ga=\ga_0+\sum_{i=1}^{e-1}(\varpi')^{e-i}\ga_i$ where $\ga_i\in\bbg_i(\cO)$ for each $0\le i\le e-1$. Let $\Bar{\ga}$ be the image of $\ga$ under the reduction map 
    \[\fg(\cO)\to\fg(k)=\bbg_0(k)\oplus\bigoplus_{i=1}^{e-1}\bbg_i(k)t^i\subset\bbg(\cO'/\varpi\cO')=\bbg(k[t]/t^e)\]
    where $t$ denotes the image of $\varpi'\in\cO'$ in $\cO'/\varpi\cO'$. Let $\Bar{\ga}_0\in\bbg_0(k)$ be the image of $\ga_0$ under the reduction map $\bbg_0(\cO)\to\bbg_0(k)$. Since $\ga\in\fg(\cO)\cap\bbg^{\mathrm{reg}}(\cO')$ we have $\Bar{\ga}_0\in\bbg_0(k)\cap\bbg^{\mathrm{reg}}(k)$. \par 
    Consider the $k$-linear endomorphism $\mathrm{ad}(\Bar{\ga})$ on $\fg(\cO)_{\le j}/\varpi\fg(\cO)_{\le j}$. We have $\dim_k(Q\otimes_\cO k)=\dim_k\coker(\mathrm{ad}(\Bar{\ga}))=\dim_k\ker(\mathrm{ad}(\Bar{\ga}))$ and it suffices to show that $\dim_k\ker(\Bar{\ga})\le\bbr$.\par 
    Consider the following chain of $\cO$-modules
    \[\fg(\cO)_{\le j}\supset\fg(\cO)_{\le j-1}\supset\dotsm\supset\fg(\cO)_{\le 0}=\fg(\cO)\supset\varpi\fg(\cO)_{\le e-1}\supset\dotsm\supset\varpi\fg(\cO)_{\le j}\]
    whose successive quotients are 
    \[\bbg_j(k), \bbg_{j-1}(k),\dotsc,\bbg_0(k),\dotsc,\bbg_{j+1}(k).\]
    This induces a filtration on the $k$-vector space $\fg(\cO)_{\le j}/\varpi\fg(\cO)_{\le j}$ preserved by $\mathrm{ad}(\Bar{\ga})$, which induces the $k$-linear map $\mathrm{ad}(\Bar{\ga}_0)$ on the associated graded space $\bbg(k)=\oplus_{i=0}^{e-1}\bbg_i(k)$. Since $\Bar{\ga}_0\in\bbg^{\mathrm{reg}}(k)$, we get the desired inequality
    \[\dim_k\ker(\mathrm{ad}(\Bar{\ga})|\fg(\cO)_{\le j}/\varpi\fg(\cO)_{\le j})\le\dim_k\ker(\mathrm{ad}(\Bar{\ga}_0)|\bbg(k))=\bbr.\]
\end{proof}
Consider the $\cO$-lattice $\fp:=[\ga,\fg(\cO)]$ in $\fp(F)$. We have $[\ga,\fp]\subset\fp$ and hence $[\ga,\fp^\vee]\subset\fp^\vee$. On the other hand, the dual lattice is described by    
\begin{equation}\label{eq:p-dual}
    \begin{split}
        \fp^\vee &=\{x\in\fp(F),(x,[\ga,y])\in\cO,\;\forall y\in\fg(\cO)\}\\
        &=\{x\in\fp(F),([\ga,x],y)\in\cO,\;\forall y\in\fg(\cO)\}\\
        &=\{x\in\fp(F),[\ga,x]\in\fg(\cO)^\vee\}.
    \end{split}
\end{equation}
Then we get that
\[[\ga,\fp^\vee]\subset\fg(\cO)^\vee\cap\fp(F)=[\ga,\fg(\cO)^\vee]\subset[\ga,\fp^\vee]\]
where the first inclusion follows from \eqref{eq:p-dual}, the middle equality follows from Lemma \ref{lem:lattice-saturated} and the last inclusion is seen as follows: for any $x\in\fg(\cO)^\vee$, let $x=x_1+x_2$ be its orthogonal decomposition where $x_1\in\fg_\ga(F)$ and $x_2\in\fp(F)$; then $[\ga,x_2]=[\ga,x]\in[\ga,\fg(\cO)^\vee]\subset\fg(\cO)^\vee$ and by \eqref{eq:p-dual} we have $x_2\in\fp^\vee$ so that $[\ga,x]=[\ga,x_2]\in[\ga,\fp^\vee]$.\par
Therefore we deduce that
\[\fp=[\ga,\fg(\cO)]\subset[\ga,\fg(\cO)^\vee]=[\ga,\fp^\vee]\subset\fp^\vee\]
and we can compute:
\begin{equation}
    \begin{split}
        \mathrm{length}_\cO(\fp^\vee/\fp)&=\mathrm{length}_\cO(\fp^\vee/[\ga,\fp^\vee])+\mathrm{length}_\cO([\ga,\fp^\vee]/\fp)\\
        &=\mathrm{length}_\cO(\fp^\vee/[\ga,\fp^\vee])+\mathrm{length}_\cO([\ga,\fg(\cO)^\vee]/[\ga,\fg(\cO)])
    \end{split}
\end{equation}
We have $\mathrm{length}_\cO(\fp^\vee/[\ga,\fp^\vee])=d_\fg(\ga)$. Indeed, by the definition of $d_\fg(\ga)$ this equality is even true if $\fp^\vee$ is replaced by any $\cO$-lattice in $\fp(F)$. Then from the exact sequence
\[0\to\frac{\fg_\ga(F)\cap\fg(\cO)^\vee}{\fg_\ga(F)\cap\fg(\cO)}\to\frac{\fg(\cO)^\vee}{\fg(\cO)}\xrightarrow{\mathrm{ad}(\ga)}\frac{[\ga,\fg(\cO)^\vee]}{[\ga,\fg(\cO)]}\to0\]
we get that 
\[\mathrm{length}_\cO(\fp^\vee/\fp)=d_\fg(\ga)+\mathrm{length}_\cO(\fg(\cO)^\vee/\fg(\cO))-\mathrm{length}_\cO\left(\frac{\fg_\ga(F)\cap\fg(\cO)^\vee}{\fg_\ga(F)\cap\fg(\cO)}\right).\]
For each $1\le j\le e-1$, we have the following diagram in which the rows are exact
\[\xymatrix{
0\ar[r]&\fg(\cO)_{\le j-1}\ar[r]\ar[d]^{\mathrm{ad}(\ga)} &\fg(\cO)_{\le j}\ar[r]\ar[d]^{\mathrm{ad}(\ga)} &\bbg_j(k)\ar[d]^{\mathrm{ad}(\Bar{\ga}_0)}\ar[r] &0\\
0\ar[r]&\fg(\cO)_{\le j-1}\ar[r] &\fg(\cO)_{\le j}\ar[r] & \bbg_j(k)\ar[r]&0
}\]
Since $\coker(\mathrm{ad}(\ga))$ is torsion free by Lemma \ref{lem:lattice-saturated}, the connecting homomorphism in the associated long exact sequence vanishes and hence we get a short exact sequence
\[0\to\fg_\ga(F)\cap\fg(\cO)_{\le j-1}\to\fg_\ga(F)\cap\fg(\cO)_{\le j}\to\ker(\mathrm{ad}(\Bar{\ga}_0)|\bbg_j(k))\to0.\]
This implies that 
\begin{equation}
    \begin{split}
        \mathrm{length}_\cO\left(\frac{\fg_\ga(F)\cap\fg(\cO)^\vee}{\fg_\ga(F)\cap\fg(\cO)}\right)&=\sum_{j=1}^{e-1}\dim_k \ker(\mathrm{ad}(\Bar{\ga}_0)|\bbg_j(k))\\
        &=\dim_k\ker(\mathrm{ad}(\Bar{\ga}_0)|\bbg(k))- \dim_k\ker(\mathrm{ad}(\Bar{\ga}_0|\bbg_0(k)) =\bbr-r=\mathrm{Art}(G)
    \end{split}
\end{equation}
where we use the fact that $\Bar{\ga}_0\in\bbg^{\mathrm{reg}}(k)$ for the second equality.\par  
Then we get
\[\mathrm{length}_\cO(\fp^\vee/\fp)=d_\fg(\ga)+\mathrm{length}_\cO(\fg(\cO)^\vee/\fg(\cO))-\mathrm{Art}(G).\]
Now consider the following chain of $\cO$-lattices in $\fg(F)$:
\[\fg_\ga\oplus\fp\subset\fg(\cO)\subset\fg(\cO)^\vee\subset(\fg_\ga\oplus\fp)^\vee=\fg_\ga^\vee\oplus\fp^\vee.\]
We note that $\mathrm{ad}(\ga)$ induces an isomorphism $\fg(\cO)/(\fg_\ga\oplus\fp)\cong\fp/\mathrm{ad}(\ga)(\fp)$ and hence
\[\mathrm{length}_\cO\left(\frac{(\fg_\ga\oplus\fp)^\vee}{\fg(\cO)^\vee}\right)=\mathrm{length}_\cO\left(\frac{\fg(\cO)}{\fg_\ga\oplus\fp}\right)=\mathrm{length}_\cO(\fp/\mathrm{ad}(\ga)(\fp))=d_\fg(\ga).\]
From the above discussions we get that
\begin{equation}\label{eq:length-g-gamma}
    \begin{split}
\mathrm{length}_\cO(\fg_\ga^\vee/\fg_\ga)&=2d_\fg(\ga)+\mathrm{length}_\cO(\fg(\cO)^\vee/\fg(\cO))-\mathrm{length}_\cO(\fp^\vee/\fp)\\
&=d_\fg(\ga)+\mathrm{Art}(G).
    \end{split}
\end{equation}
Finally by Corollary \ref{cor:length-artin-conductor} we have 
\[\mathrm{length}_\cO(\fg_\ga^{\dagger,\vee}/\fg_\ga^\dagger)=\mathrm{Art}_\ga\]

and when combined with \eqref{eq:length-Lie-alg} and \eqref{eq:length-g-gamma}, we get
\[\mathrm{length}_\cO(\fg_\ga^\dagger/\fg_\ga)=\frac{1}{2}(d_\fg(\ga)+\mathrm{Art}(G)-\mathrm{Art}_\ga).\]
Note that the assumptions in Corollary \ref{cor:length-artin-conductor} are satisfied since the restriction of the bilinear from on the split Lie algebra $\bbg(F)$ to any split Cartan subalgebra induces a $\bbW$-invariant perfect pairing on the $\cO$-module $\La\otimes_\bbZ\cO$, where $\La=X_*(G_{\ga,\overline{F}})$ is the coweight lattice.  

\subsection{Finishing the proof of main theorem in Lie algebra case}
\begin{prop}\label{prop:equidim}
Suppose $G$ is essentially tamely ramified and the residue characteristic is good for $G$. For any regular semisimple element $\ga\in\fg(F)^{\mathrm{rs}}$, the Witt-vector affine Springer fiber $Y_\ga$ is equidimensional. If moreover $G$ is semisimple simply connected and $\ga$ is topologically nilpotent, then $Y_\ga$ is connected.
\end{prop}
\begin{proof}
    By Proposition \ref{prop:reduction-Lie-alg} and Proposition \ref{prop:HC-descent-Lie-alg} we may assume that $G$ is tamely ramified semisimple simply connected and $\ga$ is topologically nilpotent. By Lemma \ref{lem:quasi-log} and Proposition \ref{prop:pairing}, there exists a reductive group $G^\natural$ over $F$ with an isomorphism $G^\natural_{\mathrm{ad}}\cong G_{\mathrm{ad}}$ and a quasi-logarithm map $\Phi:G^\natural\to\mathrm{Lie}(G^\natural)$ that is bijective on topologically nilpotent elements in the sense of Definition \ref{def:qlog-top-nilp}. Since $G$ is semisimple, we can view $\ga$ as an element in $\mathrm{Lie}(G^\natural)$ via the natural embedding $\fg\into\mathrm{Lie}(G^\natural)$. Let $\delta^\natural\in G^\natural(F)$ be the regular semisimple strongly topologically unipotent element such that $\ga=\Phi(\delta^\natural)$ and let $\delta\in G_{\mathrm{ad}}(F)=G^\natural_{\mathrm{ad}}(F)$ be the image of $\delta^\natural$. Let $\bfI_{\mathrm{ad}}$ be the image of $\bfI$ in $G_{\mathrm{ad}}(F)$. Then by Lemma \ref{lem:compare-group-Lie} we have 
    \[Y_\ga=X_{\bfI,\delta}^{G_{\mathrm{ad}}}=\{g\in G(F)/\bfI\mid g^{-1}\delta g\in\bfI_{\mathrm{ad}}\}\]
    and we conclude by Proposition \ref{prop:equi-dim-group}. 
\end{proof}

\begin{cor}\label{cor:dim-eq-reg}
    For any $\ga\in\fg(F)^{\mathrm{rs}}$ and any parahoric subgroup $\bfP\subset G(F)$, we have $\dim X_{\bfP,\ga}\le\dim Y_\ga$ and equality holds if $\bfP$ is contained in a special parahoric subgroup. Moreover, in the case where $\bfP=G(\cO)$ is a special parahoric subgroup we have
    \[\dim X_\ga=\dim Y_\ga=\dim X_\ga^{\mathrm{reg}}.\] 
\end{cor}
\begin{proof}
    We may assume that $\bfP\supset\bfI$ is a standard parahoric subgroup. Consider the natural map $\pi_{\bfP,\ga}:Y_\ga\to X_{\bfP,\ga}$. For each $g\bfP\in X_\ga$, let $\bar{\ga}_g$ be the image of $\mathrm{ad}(g)^{-1}(\ga)$ under the reduction homomorphism $\mathrm{Lie}(\bfP)\to\mathrm{Lie}(\sG_\bfP)$. The conjugacy class of $\bar{\ga}_g$ is well defined and the fiber $\pi_{\bfP,\ga}^{-1}(g\bfP)$ is isomorphic to the perfection of the classical Springer fiber $\cB_{\bar{\ga}_g}$ of $\bar{\ga}_g$ in the flag variety of $\sG_\bfP$. Since the classical Springer fiber $\cB_{\bar{\ga}_g}$ is always non-empty and is $0$-dimensional if and only if $\bar{\ga}_g$ is regular, we see that $\pi_{\bfP,\ga}$ is surjective and is quasi-finite above the regular open subset $X_{\bfP,\ga}^{\mathrm{reg}}$. In particular we have
    \[\dim X_{\bfP,\ga}^{\mathrm{reg}}\le\dim X_{\bfP,\ga}\le\dim Y_\ga.\]
    When $\bfP=G(\cO)$ is a special parahoric subgroup, the regular locus $X_{\ga}^{\mathrm{reg}}$ is nonempty (since by the assumptions in \S\ref{sec:more-assumption} we have $\ga\in\fg^{\mathrm{reg}}(\cO)$ so that $1\in X_\ga^{\mathrm{reg}}$) and then by Proposition \ref{prop:equidim} we deduce that $\dim X_\ga=\dim Y_\ga=\dim X_\ga^{\mathrm{reg}}$.\par 
    More generally if the parahoric subgroup $\bfP$ is contained in a special parahoric subgroup $\bfP_0=G(\cO)$, the natural $\pi_{\bfP_0,\ga}:Y_\ga\to X_\ga$ factors into a composition $Y_\ga\xrightarrow{\pi_{\bfP,\ga}} X_{\bfP,\ga}\to X_\ga$ from which we deduce that $\dim X_{\bfP,\ga}=\dim X_\ga=\dim Y_\ga$. 
\end{proof}
Finally, by the remark at the beginning of this section, Theorem \ref{thm:main-Lie-algebra-case} follows from Theorem \ref{thm:dim-reg}, Proposition \ref{prop:equidim} and Corollary \ref{cor:dim-eq-reg}.

\section{Dimension of affine Springer fibers for the groups}\label{sec:dim-group}
In this section we will establish the dimension formula of the Witt vector affine Springer fiber for the groups, and thus finish the proof of Theorem \ref{thm:main-group-case}. After a series of reductions, this will eventually follow from Theorem \ref{thm:main-Lie-algebra-case}, the main result in the case of Lie algebras.
\subsection{The case of hyperbolic conjugacy classes}
We first consider the special case of a split regular semisimple conjugacy class that can be studied more directly. In the equal-characteristic setting, we have proved a more general result in \cite[Corollary 3.5.3]{Chi}, the argument
in \emph{loc. cit.} applies also in the current mixed-characteristic case. Here we give a more straightforward argument in the current setting, which is similar in spirit to the proof of \cite[\S5,Proposition 1]{KL}.\par 
We assume that $G$ is a split reductive group over $F$, $T\subset G$ is a split maximal torus over $F$ and $B=TU$ is a Borel subgroup of $G$ containing $T$, with unipotent radical $U$. Let $P=MN$ be a standard parabolic subgroup of $G$ with Levi $M$ and unipotent radical $N$. Then $B\subset P$ and $T\subset M$. Following \cite[\S5.3]{GHKR}, or \cite[\S3.4]{Chi}, we have a decreasing filtration of $N$ by normal subgroups $P$:
\[N=N[1]\supsetneq N[2]\supsetneq\dotsm\supsetneq N[l]\supsetneq N[l+1]=1\]
such that: 
\begin{itemize}
    \item the successive quotients $N\langle i\rangle=N[i]/N[i+1]$ are abelian;
    \item the commutator satisfies $[N,N[i]]\subset N[i+1]$ for all $1\le i\le l$. 
\end{itemize}
The second point can be deduced from, for example, \cite[\S3, Lemma 15]{Ste-Chevalley}. Each $N\langle i\rangle$ is a product of certain root groups $U_\alpha$ for roots $\alpha$ of $T$ occurring in $\mathrm{Lie}(N)$. In particular we have $N\cong\prod_{i=1}^l N\langle i\rangle$ as a scheme. Moreover, for each $1\le i\le l$ the adjoint action of $M$ on $N$ restricts to an action on $N\langle i\rangle$ and there is a natural $M$-equivariant isomorphism of group schemes $N\langle i\rangle\cong\mathrm{Lie}(N\langle i\rangle)$.\par 
Let $\bfP\subset G(F)$ be a parahoric subgroup corresponding to a facet in the apartment of $T$. Then we have 
\[N(F)\cap\bfP=\prod_{i=1}^{l}N\langle i\rangle(F)\cap\bfP.\]
Let $\cN$ be the $\cO$-model of $N$ such that $\cN(\cO)=N(F)\cap\bfP$. Then we get $\cO$-models $\cN[i]$ (resp. $\cN\langle i\rangle$) for $N[i]$ (resp. $N\langle i\rangle$) such that $\cN[i](\cO)=N[i](F)\cap\bfP$ and $\cN\langle i\rangle(\cO)=N\langle i\rangle(F)\cap\bfP$. We have the associated Witt vector affine Grassmannians $LN/L^+\cN$, $LN[i]/L^+\cN[i]$ and $LN\langle i\rangle/L^+\cN\langle i\rangle$ that are perfect indschemes. 
\par   
For each element $\ga\in M(F)$ define
\[r_N(\ga):=\mathrm{val}(\det(1-\mathrm{ad}(\ga)\mid\mathrm{Lie}(N)))\in\bbZ\cup\{\infty\}\]
and for each $1\le i\le l$,
\[r_i(\ga):=\mathrm{val}(\det(1-\mathrm{ad}(\ga)\mid\mathrm{Lie}(N\langle i\rangle)))\in\bbZ\cup\{\infty\}.\]
Then we have $r_N(\ga)=\sum_{i=1}^l r_i(\ga)$. 
\begin{prop}\label{prop:Levi-reduction}
    With notations as above, let $\ga\in M(F)$ be an element such that $\mathrm{ad}(\ga)\cN(\cO)\subset\cN(\cO)$ and $r_N(\ga)<\infty$. Then the subspace
    \[Z:=\{u\in LN/L^+\cN\mid u^{-1}\ga u\ga^{-1}\in L^+\cN\}\]
    is represented by a perfect scheme perfectly of finite type (see Definition \ref{def:perfect-finite-type}) of dimension $r_N(\ga)$.
\end{prop}
\begin{proof}
    Define a \emph{right} action of $N$ on itself by requiring that any $u\in N$ acts by $x*u:=u^{-1}x\ga u\ga^{-1}$ for any $x\in N$. This also defines an action of the subgroups $N[i]$ and $N\langle i\rangle$ on themselves.\par 
    
    Define a morphism of indschemes $f_\ga:LN\to LN$ by \[f_\ga(u):=u^{-1}\ga u\ga^{-1}=1*u.\] 
    Then $f_\ga$ preserves the sub-indschemes $LN[i]$ and $LN\langle i\rangle$ and restricts to morphisms
    \[f_\ga^{[i]}:LN[i]\to LN[i],\quad f_\ga^{\langle i\rangle}:LN\langle i\rangle\to LN\langle i\rangle.\]
    By assumption we have inclusions 
    \[f_\ga(L^+\cN)\subset L^+\cN,\quad f_\ga^{[i]}(L^+\cN[i])\subset L^+\cN[i],\quad f_\ga^{\langle i\rangle}(L^+\cN\langle i\rangle)\subset L^+\cN\langle i\rangle.\]
    For each $1\le i
    \le l$, we consider the perfect indschemes
    \[Z^{[i]}:=(f_\ga^{[i]})^{-1}(L^+\cN[i])/L^+\cN[i],\quad Z^{\langle i\rangle}:=(f_\ga^{\langle i\rangle})^{-1}(L^+\cN\langle i\rangle)/L^+\cN\langle i\rangle.\]
    In particular when $i=1$ we get $Z^{[1]}=Z=f_\ga^{-1}(L^+\cN)/L^+\cN$. \par 
    More generally, for each $x\in N(F)$ we consider the perfect indscheme
    \[Z_x^{[i]}:=\{u\in LN[i]/L^+\cN[i],x*u\in L^+\cN\}.\]
    For any $u\in Z_x^{[i]}$ we have
    \[x*u=u^{-1}xu\cdot(1*u)=[u^{-1},x]\cdot x\cdot(1*u)\]
    where $[u^{-1},x]:=u^{-1}xux^{-1}$ is the commutator. Note that since $u\in LN[i]$, we have $[u^{-1},x]\in LN[i+1]$. So when $x\in L^+\cN\cdot LN[i+1]$, the condition $x*u\in L^+\cN$ implies that $1*u\in L^+\cN\cdot LN[i+1]$ and hence $1*\Bar{u}\in Z^{\langle i\rangle}$, where $\Bar{u}$ is the image of $u$ in $LN\langle i\rangle$. Thus the map $u\mapsto \Bar{u}$ induces a map 
    \[\pi_{i,x}:Z_x^{[i]}\to Z^{\langle i\rangle}.\]
    \textbf{Claim:} For any $1\le i\le l$ and any $x\in L^+\cN\cdot LN[i+1]$, the map $\pi_{i,x}$ above is surjective. Moreover, for any $u_i\in LN[i]$ with $\Bar{u}_i\in Z^{\langle i\rangle}$ we have $\pi_{i,x}^{-1}(u_i)=Z^{[i+1]}_{x*u_i}$.\par
    \begin{proof}[Proof of Claim]
    We prove the surjectivity by descending induction on $i$. When $i=l$ this is clear since $Z_{x}^{[l]}=Z^{[l]}=Z^{\langle l\rangle}$. Suppose the claim is true for $i+1$. For any $u_i\in LN[i]$ representing an element $\Bar{u}_i\in Z^{\langle i\rangle}$, we have $1*u_i\in L^+\cN[i] LN[i+1]=LN[i+1]L^+\cN[i]$. We want to find an element $u_{i+1}\in LN[i+1]$ such that $u_iu_{i+1}\in Z_{x}^{[i]}$. In other words, we look for $u_{i+1}$ such that 
    \[x*u_iu_{i+1}=u_{i+1}^{-1}u_i^{-1}xu_iu_{i+1}(1*u_iu_{i+1})=u_{i+1}^{-1}u_i^{-1}xu_i(1*u_i)\ga u_{i+1}\ga^{-1}\in L^+\cN.\]
    Write $1*u_i=yz$ with $y\in L^+\cN[i]$ and $z\in LN[i+1]$, then we have 
    \[x':=x*u_i=[u_i^{-1},x]xyz\in L^+\cN\cdot LN[i+1].\]
    Thus we can write $x'=y'z'$ with $y'\in L^+\cN$ and $z'\in LN[i+1]$. Then for any $u'\in LN[i+1]$ we have
    \[x'*u'=[(u')^{-1},x']y'z'(1*u').\]
    Since $\ga$ is regular semisimple, the action map $N[i+1](F)\to N[i+1](F)$ sending $u\mapsto 1*u$ is injective. So it defines a transitive action of $N[i+1](F)$ on itself and we can find $u'\in LN[i+1]$ such that $1*u'=(z')^{-1}$. Then we have $x'*u'\in LN[i+2]L^+\cN$. By induction hypothesis the space $Z^{[i+1]}_{x'*u'}$ is nonempty, so we can find $v\in LN[i+2]\subset LN[i+1]$ such that $x'*(u'v)\in L^+\cN$. Taking $u_{i+1}=u'v$ we finish the proof of the surjectivity of $\pi_{i,x}$. The description of its fiber is immediate from definition.
    \end{proof}
    Now we  analyze the structure of $Z=Z^{[1]}$. Let $x_1=1$. Take any $u_1\in Z$ and let $x_2:=1*u_1\in L^+\cN$, then we have $\pi_{1,x_1}^{-1}(\bar{u}_1)=Z_{x_2}^{[2]}$. Take any $u_2\in Z_{x_2}^{[2]}$ and let $x_3:=x_2*u_2\in L^+\cN$, then we have $\pi_{2,x_2}^{-1}=Z^{[3]}_{x_3}$. Continue in this way we get a descending sequence of spaces $\pi_{i,x_i}^{-1}(\bar{u}_i)=Z_{x_{i+1}}^{[i+1]}\subset LN[i+1]/L^+\cN[i+1], 1\le i\le l$, where the last one $Z_{x_{l+1}}^{[l+1]}$ is a single point. Note that for each $1\le i\le l$, $Z^{\langle i\rangle}$ is the perfection of an affine space of dimension $r_i(\ga)$. Then we deduce by descending induction on $i$ that $Z=Z^{[1]}$ is a perfect scheme perfectly of finite type of dimension $r_N(\ga)=\sum_{i=1}^l r_i(\ga)$. 
\end{proof}
\begin{cor}\label{cor:dim-split-case}
    Let $G$ be a split reductive group scheme over $\cO$ and let $T\subset G$ be a split maximal torus defined over $\cO$. For any bounded regular semisimple element $\ga\in T(\cO)\cap G^{\mathrm{rs}}(F)$, the Witt vector affine Springer fiber in the affine Grassmannian
    \[X_\ga=X_{G(\cO),\ga}^{G}=\{g\in LG/L^+G\mid g^{-1}\ga g\in L^+G\}\]
    is represented by a perfect scheme locally perfectly of finite type (see Definition \ref{def:perfect-finite-type}) of dimension 
    \[\dim X_\ga=\frac{1}{2}d_G(\ga).\]
\end{cor}
\begin{proof}
    Let $B$ be a Borel subgroup of $G$ containing $T$ and let $U$ be its unipotent radical, both defined over $\cO$. By the Iwasawa decomposition, $X_\ga$ can be decomposed into a disjoint union of $X_*(T)=T(F)/T(\cO)$ translates of 
    \[Z=\{u\in LU/L^+U,u^{-1}\ga u\ga^{-1}\in L^+U\}.\]
    By Proposition \ref{prop:Levi-reduction}, it is a perfectly of finite type perfect scheme of dimension 
    \[r_U(\ga)=\mathrm{val}(\det(1-\mathrm{ad}(\ga)\mid\mathrm{Lie}(U)))=\sum_{\alpha\in\Phi^+}\mathrm{val}(1-\alpha(\ga))\]
    where $\Phi^+$ is the set of positive roots. Since $\ga\in T(\cO)$ we have $\mathrm{val}(1-\alpha(\ga))=\mathrm{val}(1-\alpha(\ga)^{-1})$ for any root $\alpha$. This implies that $r_U(\ga)=\frac{1}{2}d_G(\ga)$ and we are done. 
\end{proof}

\begin{cor}\label{cor:disc-group-Lie-alg}
    Let $G$ be a reductive group over $F$. Suppose there exists a quasi-logarithm map $\Phi:G\to\fg$ that is bijective over topologically nilpotent elements. Then for any strongly topologically unipotent and regular semisimple element $u\in G^{\mathrm{rs}}(F)$, we have $d_G(\gamma)=d_\fg(\Phi(\gamma))$.
\end{cor}
\begin{proof}
     When computing the invariants $d_G(\gamma)$ and $d_\fg(\Phi(\ga))$ we may pass to a finite extension and hence assume that $G$ is split and $\ga$ is contained in a split maximal torus. By Lemma \ref{lem:compare-group-Lie} we have an isomorphism of the centralizers $G_\gamma\cong G_{\Phi(\gamma)}$, which is a split maximal torus. Also, there is an identification of affine Springer fibers in the affine Grassmannians $X_\gamma\cong X_{\Phi(\gamma)}$. Then the equality follows from Theorem~\ref{thm:main-Lie-algebra-case} and Corollary \ref{cor:dim-split-case}.
\end{proof}

\subsection{Harish-Chandra descent}\label{sec:HC-descent}
In this section we reduce Theorem \ref{thm:main-group-case}
to the case of topologically unipotent elements. We follow the method of \cite[Appendix B]{BV21} in the equal-characteristic setting. \par
We maintain the notations and assumptions in Theorem \ref{thm:main-group-case} and assume that the affine Springer fiber $X^{G'}_{\bfP',\ga}$ is non-empty. Then by Theorem~\ref{thm:nonempty-group}, the element $\ga$ is bounded modulo center and after $G(F)$-conjugation we may assume that $\ga\in\widetilde{\bfI}$ by Lemma \ref{lem:bounded}. Let $\ga=su=us$ be the topological Jordan decomposition where $s\in\widetilde{\bfI}$ is strongly semisimple mod center and $u\in\bfI$ is topologically unipotent mod center.\par 
Let $H':=(G_s'^\circ)^{\mathrm{sc}}$ be the simply connected cover of the identity component $G_s'^\circ$ of the centralizer of $s$ in $G'$ and let $\pi':H'\to G'$ be the natural homomorphism. Let $\bfP'_s:=G_s'(F)\cap\bfP'$ and $\bfP_{H'}:=(\pi')^{-1}(\bfP'_s)$. Then $\bfP_{H'}$ is a parahoric subgroup of $H'(F)$. Let $H:=H'_{\mathrm{ad}}=G_{s,\mathrm{ad}}^\circ$ be the common adjoint group of $H'$ and $G_s^\circ$. Let $\bfP_H$ be the image of $\bfP\cap G_s(F)$ in $H(F)$. Then $\bfP_H$ is also the image of $\bfP'_H$ in $H(F)$ and is a parahoric subgroup of $H(F)$. Let $\widetilde{\bfP}_H$ be the normalizer of $\bfP_H$ in $H(F)$. Let $u_0$ be the image of $u$ in $H(F)$. For any $g\in G_s'(F)$, we have $\mathrm{ad}(g)^{-1}u\in\widetilde{\bfP}$ if and only if $\mathrm{ad}(g)^{-1}u_0\in\widetilde{\bfP}_H$. We will relate the Witt vector affine Springer fiber $X^{G}_{\bfP',\ga}$ to the Witt vector affine Springer fiber $X^{H'}_{\bfP_{H},u_0}$, whose set of $k$-points is: 
\[X^{H}_{\bfP_{H'},u_0}(k)=\{h\in H'(F)/\bfP_{H'}\mid \mathrm{ad}(h)^{-1}u_0\in\widetilde{\bfP}_H\}.\]
Let $\widetilde{\bfP}'_s$ be the normalizer of $\pi'(\bfP_{H'})$ in $G_s'(F)$. Then we have $\bfP'_s\subset\widetilde{\bfP'}_s$ and the quotient group $\widetilde{\bfP}'_s/\bfP'_s$ is discrete.
\begin{prop}\label{prop:HC-descent}
    Let the notations and assumptions be as above. For any $n\in\widetilde{\bfP}_s/\bfP_s$ we consider the morphism $i_n:X^{H}_{\bfP_{H'},u_0}\to X^{G}_{\bfP',\ga}$ defined by sending $h\bfP_{H'}$ to $\pi'(h)n\bfP'$. Then the disjoint union of the $i_n$'s defines an isomorphism of perfect schemes
    \[i:=\bigsqcup_{n\in\widetilde{\bfP}'_s/\bfP'_s}i_n:\bigsqcup_{n\in\widetilde{\bfP}'_s/\bfP'_s}X^{H}_{\bfP_{H'},u_0}\to X^{G}_{\bfP',\ga}.\]
    Moreover we have $d_G(\ga)=d_H(u_0)$ and 
    \[\mathrm{Art}(G)+\mathrm{def}(\kappa_G(\ga))-\mathrm{Art}(G_\ga)=\mathrm{Art}(H)+\mathrm{def}(\kappa_H(u_0))-\mathrm{Art}(H_{u_0}).\]
\end{prop}
\begin{proof}
    First we show injectivity. Suppose there are $n_1,n_2\in\widetilde{\bfP}'_s/\bfP'_s$ with $i_{n_1}(h_1\bfP_H)=i_{n_2}(h_2\bfP_H)$ so that $\pi'(h_1)n_1\bfP'=\pi'(h_2)n_2\bfP'$. Then we get $\pi'(h_1^{-1}h_2)\in \pi'(H'(F))\cap n_1\bfP' n_2^{-1}$ which implies that $n_1=n_2$ (using the Kottwitz homomorphism) and $h_1\bfP_{H'}=h_2\bfP_{H'}$.\par 
    Next we show surjectivity. Let $g\bfP'\in X^{G}_{\bfP',\ga}$ so that $g^{-1}\ga g\in\widetilde{\bfP}$. Then we have $g^{-1}sg\in\widetilde{\bfP}$ and $g^{-1}ug\in\widetilde{\bfP}$. By 
    \cite[Proposition 2.33, Corollary 2.37]{Spice08}, this implies that $g\in G'_s(F)$ and therefore
    \[X^{G}_{\bfP',\ga}=\{g\in G_s'(F)/\bfP'_s | g^{-1}ug\in\widetilde{\bfP}\cap G_s(F)\}=\{g\in G_s'(F)/\bfP'_s | g^{-1}u_0g\in\widetilde{\bfP}_H\}.\]
    Take an element $h\in H'(F)$ such that $h\bfP_{H'}h^{-1}=(\pi')^{-1}(g\bfP' g^{-1})$ and let $n:=\pi'(h)^{-1}g$. Then we get that $n\in\widetilde{\bfP}'_s$ and $g\bfP'=i_n(h\bfP_{H'})$. \par
    It remains to relate the numerical invariants of $\ga\in G(F)$ and $u_0\in H(F)$. Note that $G_\ga$ is a maximal torus of both $G_s^0$ and $G$. Let $\Phi_G(G_\ga)$ be the set of roots of $G$ (over $\overline{F}$) with respect to $G_\gamma$. Then the set of roots of $G_s^\circ$ with respects to $G_\ga$ is $\Phi_{G_s^\circ}(G_\ga)=\{\alpha\in\Phi_G(G_\ga),\alpha(s)=1\}$. We have
    \begin{equation}
    \begin{split}
        d_G(\ga)&=\mathrm{val}\det(1-\mathrm{ad}(\ga)|\fg(F)/\fg_\ga(F))=\sum_{\alpha\in\Phi_G(G_\ga)}\mathrm{val}(1-\alpha(\ga))\\
        &=\sum_{\alpha\in\Phi_{G_s^\circ}(G_\ga)}\mathrm{val}(1-\alpha(\ga))=\sum_{\alpha\in\Phi_{G_s^\circ}(G_\ga)}\mathrm{val}(1-\alpha(u))=d_{G_s^\circ}(u)
    \end{split}
    \end{equation}
    since for $\alpha\in\Phi_G(G_\ga)$ such that $\alpha(s)\ne1$, we have $\mathrm{val}(1-\alpha(\ga))=\mathrm{val}(1-\alpha(s)\alpha(u))=0$ because $\mathrm{val}(\alpha(u)-1)\ge0$. By definition we see that $d_{G_s^\circ}(u)=d_H(u_0)$ and therefore $d_G(\ga)=d_H(u_0)$.\par 
    Finally let us prove the last equality involving Artin conductors and Kottwitz invariants. By construction we have $G_\ga=(G_s^\circ)_u$ and $H_{u_0}\cong(G_s^\circ)_u/Z(G_s^\circ)$. Let $S\subset G_s^\circ$ be a maximal $F$-split sub-torus. Let $S_G$ be a maximal $F$-split torus in $G$ containing $S$ and let $T$ be the centralizer of $S_G$ in $G$. Recall our assumption that $G$ is essentially tamely ramified. We first claim that
    \begin{equation}\label{eq:artin-conductor-difference}
        \mathrm{Art}(G_s^\circ)-\mathrm{Art}(G)=\dim S_G-\dim S=\mathrm{def}(\kappa_G(s)).
    \end{equation}
    To show this we make the following observations: 
    \begin{itemize}
        \item All invariants in \eqref{eq:artin-conductor-difference} do not change if we mod out a central subgroup of $G$, so we may assume that $G$ is adjoint;
        \item All invariants in \eqref{eq:artin-conductor-difference} are additive when we decompose $G$ into products of smaller groups, so we may assume further that $G$ is the Weil restriction of a tamely ramified absolutely simple adjoint group; 
        \item If $G=\mathrm{Res}_{\widetilde{F}/F}\widetilde{G}$ for a finite extension $\widetilde{F}/F$, then by Proposition \ref{prop:Artin-conductor-Res} the difference $\mathrm{Art}(G_s^\circ)-\mathrm{Art}(G)$ does not change if we replace $F$ (resp. $G$) by $\widetilde{F}$ (resp. $\widetilde{G}$). Also it follows easily from definition that the remaining quantities does not change if we make this replacement.
    \end{itemize}
    Thus to show \eqref{eq:artin-conductor-difference} we may assume that $G$ is tamely ramified. Then by Lemma \ref{lem:tame-centralizer} the group $G_s^\circ$ is also tamely ramified. So the Artin conductors $\mathrm{Art}(G_s^\circ)$ and $\mathrm{Art}(G)$ become the difference between the absolute rank and the $F$-rank and the first equality in \eqref{eq:artin-conductor-difference} follows. The second equality follows by definition and therefore \eqref{eq:artin-conductor-difference} is proved.\par 
    Since $H=G_s^\circ/Z(G_s^\circ)$ and $H_{u_0}=(G_s^\circ)_u/Z(G_s^\circ)=G_\ga/Z(G_s^\circ)$, we deduce from \eqref{eq:artin-conductor-difference} that 
    \[\mathrm{Art}(H)-\mathrm{Art}(H_{u_0})=\mathrm{Art}(G_s^\circ)-\mathrm{Art}(G_\ga)=\mathrm{Art}(G)+\mathrm{def}(\kappa_G(s))-\mathrm{Art}(G_\ga).\]
    We have an embedding $\cB(H)\into\cB(G)$ between the (reduced) Bruhat-Tits buildings (see \cite[Proposition 2.33]{Spice08}). Let $\fa_H\in\cB(H)$ be an alcove in the apartment corresponding to $S$ and let $\fa\in\cB(G)$ be the alcove containing $\fa_H$ in the apartment corresponding to $S_G$. Using notations from Definition \ref{def:defect-Kott-invariant}, we let 
    \[w_\ga:=w_\fa^G(\kappa_G(\ga)),\quad w_s:=w_\fa^G(\kappa_G(s)),\quad w_u:=w_\fa^G(\kappa_G(u)).\] 
    Then we have $w_\ga=w_sw_u$ and $w_s$ has order prime to $p$ while $w_u$ has $p$-power order. In particular, both $w_u$ and $w_s$ can be expressed as certain powers of $w_\ga$ and we deduce that
    \[X_*(S_G)_\bbQ^{w_\ga}=(X_*(S_G)_\bbQ^{w_s})^{w_u}=X_*(S)_\bbQ^{w_u}.\]
    This implies that 
    \[\mathrm{def}(\kappa_G(\ga))=\mathrm{def}(\kappa_G(s))+\dim X_*(S)_\bbQ-\dim X_*(S)_\bbQ^{w_u}.\]
    Let $\Tilde{u}\in G_s^\circ(F)$ be an element that stabilizes the alcove $\fa_H$ (and hence also $\fa$) and satisfies $\kappa_{G_s^\circ}(\Tilde{u})=\kappa_{G_s^\circ}(u)$. Then we also have $\kappa_G(\Tilde{u})=\kappa_G(u)$ by the functoriality of the Kottwitz homomorphism. We see that $\dim X_*(S)_\bbQ-\dim X_*(S)_\bbQ^{w_u}$ equals to the codimension of the fixed point loci $\fa_H^{\Tilde{u}}$ in $\fa_H$, which by Lemma \ref{lem:defect-Kott-invariant} equals to $\kappa_H(u_0)$. Combined with the above equations we are done. 
\end{proof}
As an immediate consequence we obtain the following:
\begin{cor}\label{cor:group-reduction}
    Let $(G,G',\bfP,\bfP')$ be as in \S\ref{sec:setup-adjoint-pair}. Let $(H,H')$ be another pair of reductive groups over $F$ with an isomorphism between the adjoint groups $H_{\mathrm{ad}}\cong H'_{\mathrm{ad}}$ and let $\bfP_H\subset H(F)$, $\bfP_H'\subset H'(F)$ be parahoric subgroups whose images in the adjoint group coincide. Suppose there is an isomorphism $G_\mathrm{ad}\cong H_\mathrm{ad}$ such that the image of $\bfP$ and $\bfP_H$ in the adjoint group coincide. Let $\ga\in G(F)^{\mathrm{rs}}$ and $\ga_H\in H(F)^{\mathrm{rs}}$ be bounded regular semisimple elements whose image in $G_{\mathrm{ad}}(F)\cong H_{\mathrm{ad}}(F)$ coincides. Then Theorem \ref{thm:main-group-case} holds for $X_{\bfP',\ga}^G$ if and only if it holds for $X_{\bfP_H',\ga_H}^H$
\end{cor}

Consequently the study of the Witt vector affine Springer fiber $X^{G}_{\bfP',\ga}$ associated to a regular semisimple element $\ga\in G(F)$ can be reduced to the case where $\ga$ is topologically unipotent. We would like to reduce further to the case where $\ga$ is \emph{stronly} topologically unipotent to make connection with the Lie algebra affine Springer fibers.

\subsection{A study of the case of general linear groups}\label{sec:GLn}
Under the assumption that the residue characteristic $p$ is good and the group is absolutely simple and tamely ramified, the only possible case where there exists topologically unipotent but not strongly topologically unipotent elements is when $G$ is split of type $\mathrm{A}_{n-1}$ and $p$ divides $n$. In this section we study this specific case and prove the following result: 
\begin{thm}\label{thm:GLn}
    Suppose $G_{\mathrm{ad}}\cong\mathrm{PGL}_n$ and $p$ divides $n$. Let $\ga\in G(F)$ be a bounded regular semisimple element that is topologically unipotent but not strongly topologically unipotent. Then for any parahoric subgroup $\bfP\subset G(F)$, the Witt vector affine Springer fiber $X_{\bfP,\ga}$ is represented by a finite dimensional perfect scheme, locally perfectly of finite type (see Definition \ref{def:perfect-finite-type}), whose dimension is\footnote{Although $\mathrm{Art}(G)=0$ when $G$ is split, we write the formula in this form in accordance with the general case in Theorem \ref{thm:main-group-case}.} 
    \[\dim X_{\bfP,\ga}=\frac{1}{2}(d_G(\ga)+\mathrm{def}(\kappa_G(\ga))+\mathrm{Art}(G)-\mathrm{Art}_\ga).\]
    Furthermore, if $\bfP=\bfI$ is an Iwahori subgroup, then $Y_\ga=X_{\bfI,\ga}$ is equidimensional. 
\end{thm}
We will finish the proof by the end of this subsection. Most of the following discussions are valid in more general situations and we will impose various assumptions in the statement only when necessary.\par 
By Corollary \ref{cor:group-reduction} we may assume that $G=\mathrm{GL}_n$. Let $Z_G$ be the center of $G$ and let $B$ (resp. $T$) be the standard Borel subgroup (resp. maximal torus) consisting of upper triangular (resp. diagonal) matrices. Let $\bfI\subset G(\cO)$ be the standard Iwahori subgroup consisting of matrices whose reduction mod $\varpi$ is upper triangular. Let $\Pi_n\in\mathrm{GL}_n(F)$ be the matrix whose $(n,1)$ entry is $\varpi$, whose $(i,i+1)$-entry is $1$ for $i=1,\dotsc,n-1$ and all the other
entries are zero. Then the normalizer $\widetilde{\bfI}=N_{G(F)}(\bfI)$ is generated by $\bfI$ and $\Pi_n$. \par
Let $\bfP\subset G(F)=\mathrm{GL}_n(F)$ be a parahoric subgroup. Then $\bfP$ is contained in a hyperspecial parahoric subgroup and after $G(F)$-conjugation we may assume that $\bfI\subset\bfP\subset G(\cO)$. Then the normalizer $\widetilde{\bfP}$ of $\bfP$ in $G(F)$ is generated by $\bfP$ and certain power of $\Pi_n$. For example if $\bfP=G(\cO)$ is hyperspecial, then $\widetilde{\bfP}=Z_G(F)G(\cO)$ is generated by $G(\cO)$ and $(\Pi_n)^n=\varpi\in Z_G(F)$.\par
Let $\ga\in G(F)$ be a \emph{bounded mod center} regular semisimple element and let $a=\mathrm{val}(\det(\ga))\in\bbZ$. Assume moreover that $\ga$ is not contained in any parahoric subgroup (topologically unipotent but not strongly topologically unipotent elements that we are currently interested in will satisfy this assumption). Then after multiplying by a scalar matrix in $Z_G(F)$ we may assume that $0<a<n$. Let $\tau:=(\Pi_n)^a$. Suppose moreover that $\tau\in\widetilde{\bfP}$. Then the Witt vector affine Springer fiber 
\[X_{\bfP,\ga}^G=\{g\in G(F)/\bfP\mid g^{-1}\ga g\in\widetilde{\bfP}\}=\{g\in G(F)/\bfP\mid g^{-1}\ga g\in\bfP\tau\}\]
is nonempty. In the case $\bfP=\bfI$ we get
\[Y_\ga^G=X_{\bfI,\ga}^G=\{g\in G(F)/\bfI\mid g^{-1}\ga g\in\bfI\tau\}.\]
In general the natural map $\mathrm{Fl}_\bfI=LG/\bfI\to\mathrm{Fl}_\bfP=LG/\bfP$ restricts to a morphism
\[\pi^G_{\bfP,\ga}:Y_\ga^G\to X_{\bfP,\ga}^G\]
whose fibers can be described as follows. The image of $\bfI$ under the natural homomorphism $\bfP\to\bfP/\bfP_+\cong\sG_\bfP$ is a Borel subgroup $\sB_\bfP\subset\sG_\bfP$. Since the  automorphism $\mathrm{Ad}(\tau)$ on $\bfP$ preserves the subgroup $\bfI$, after passing to the integral model and then the reductive quotient of the special fiber we get an automorphism on $\sG_\bfP$ preserving the subgroup $\sB_\bfP$. Therefore $\mathrm{Ad}(\tau)$ induces an automorphism of the flag variety $\sG_\bfP/\sB_\bfP$ that we simply denote by $\tau$. For any $g\bfP\in X_{\bfP,\ga}^G$, let $\ga_g:=g^{-1}\ga g\tau^{-1}\in\bfP$ and let $\Bar{\ga}_g\in\sG_\bfP$ be the image of $\ga_g$ under the natural map $\bfP\to\bfP/\bfP_+\cong\sG_\bfP$. Then the fiber $(\pi_{\bfP,\ga}^G)^{-1}(g\bfP)$ is the perfection of the fixed point locus of the composite automorphism $\Bar{\ga}_g\circ\tau$ on the flag variety $\sG_\bfP/\sB_\bfP$, in particular it is nonempty by \cite[Theorem 7.2 on page 49]{St-End}. Therefore $\pi^G_{\bfP,\ga}$ is surjective. 

\subsubsection{The case of elliptic conjugacy classes}
We first consider the case where $a$ is coprime to $n$. Since $\mathrm{val}\det\ga=a$, the $F$-subalgebra $E:=F[\ga]\subset\mathrm{Mat}_n(F)$ generated by $\ga$ is a degree $n$ totally ramified field extension of $F$. On the other hand we recall that $\widetilde{\bfI}$ is the stabilizer of a complete periodic chain of $\cO$-lattices $\{\cL_i;i\in\bbZ\}$ in $F^n$, where 
\begin{itemize}
    \item $\cL_0=\cO^n$ and $\cL_i=\cO^{n-i}\oplus(\varpi\cO)^i$ for $1\le i\le n-1$;
    \item  for all $i\in\bbZ$, $\cL_i\supset\cL_{i+1}$ and $\cL_{i+n}=\varpi\cL_i$.
\end{itemize}
 
Then there is a subset $J\subset\bbZ/n\bbZ$ such that $\widetilde{\bfP}$ is the stabilizer of the partial periodic lattice chain $\{\cL_i;i+n\bbZ\in J\}$. Since $\tau=(\Pi_n)^a\in\widetilde{\bfP}$ and $a$ is coprime to $n$, we must have $J=\bbZ/n\bbZ$ and hence $\bfP=\bfI$ is the standard Iwahori subgroup. Equivalently, one could also deduce this from the description of $\widetilde{\bfP}/Z_G(F)\bfP$ in terms of the local Dynkin diagram, see \cite[\S3.5]{Ti-Corvallis}.\par
Choose integers $j,r$ such that $aj-nr=1$. Let $\ga':=\varpi^{-r}\ga^j$. Then we have $\mathrm{val}(\det(\ga'))=1$ which means that $\ga'$ is a uniformizer in $\cO_E$ and $\cO_E=\cO[\ga']$. Consequently we have $\ga=(\ga')^ah$ where $h\in\cO_E^\times$ can be written as a polynomial in $\ga'$ with $\cO$-coefficients.\par
For any $g\in G(F)$, we claim that $g^{-1}\ga' g\in\widetilde{\bfI}$ if and only if $g^{-1}\ga g\in\widetilde{\bfI}$. First assume that $g^{-1}\ga' g\in\widetilde{\bfI}$. Then we have $g^{-1}hg\in\mathrm{Lie}(\bfI)$ and $\det(g^{-1}hg)\in\cO^\times$ since $h\in\cO_E^\times$. Therefore $g^{-1}hg\in\bfI$ and hence $g^{-1}\ga g=\widetilde{\bfI}$. Conversely if $g^{-1}\ga g\in\widetilde{\bfI}$, we get $g^{-1}\ga'g=\varpi^{-r}(g^{-1}\ga g)^j\in\widetilde{\bfI}$. This shows that
\[Y_\ga^G=Y_{\ga'}^G=\{g\in G(F)/\bfI\mid g^{-1}\ga'g\in\widetilde{\bfI}\}.\]
Since $a$ is coprime to $n$, the images of $\tau=(\Pi_n)^a$ and $\Pi_n$ in $W=S_n$ generate the same subgroup and therefore 
\[\mathrm{def}(\kappa_G(\ga))=\mathrm{def}(\kappa_G(\tau))=\mathrm{def}(\kappa_G(\Pi_n))=\mathrm{def}(\kappa_G(\ga'))=n-1.\]
Now assume that $p|n$ so that $a$ and $j$ are coprime to $p$. Then we have $d_G(\ga)=d_G(\ga^j)=d_G(\ga')$. Also it is clear that $G_\ga=G_{\ga'}=\mathrm{Res}_{E/F}\bbG_m$. Thus we may assume from now on that $a=1$ and $\tau=\Pi_n$.\par 
By assumption $\ga$ is topologically nilpotent as an element in $\fg(F)=\mathrm{Mat}_n(F)$. We consider the Lie algebra affine Springer fiber associated to $\ga$:
\[Y_\ga^{\fg}=\{g\in G(F)/\bfI\mid g^{-1}\ga g\in\mathrm{Lie}(\bfI)\}\]
together with its companion in the affine Grassmannian:
\[X_\ga^{\fg}:=\{g\in G(F)/G(\cO)\mid g^{-1}\ga g\in\fg(\cO)\}.\]
Let $\pi_\ga^\fg:Y_\ga^{\fg}\to X_\ga^{\fg}$ be the natural morphism. Since $\bfI\tau\subset\mathrm{Lie}(\bfI)$ we have a natural embedding $Y_\ga^G\subset Y_\ga^{\fg}$. Recall that the regular open subspace $X_\ga^{\fg,\mathrm{reg}}\subset X_\ga^{\fg}$ consists of points $g G(\cO)\in X_\ga^{\fg}$ such that the reduction mod $\varpi$ of $g^{-1}\ga g$ is regular and $\pi_\ga^\fg$ restricts to an isomorphism above $X_\ga^{\fg,\mathrm{reg}}$.\par 
We claim that $Y_\ga^G=(\pi^\fg_\ga)^{-1}(X_\ga^{\fg,\mathrm{reg}})$. It is clear that the reduction mod $\varpi$ of a matrix in the coset $\bfI\tau$ is regular nilpotent and hence we have the inclusion $Y_\ga^G\subset(\pi^\fg_\ga)^{-1}(X_\ga^{\fg,\mathrm{reg}})$. Conversely for any $g\in(\pi^\fg_\ga)^{-1}(X_\ga^{\fg,\mathrm{reg}})$ we have $g^{-1}\ga g\in\mathrm{Lie}(\bfI)$ and its reduction mod $\varpi$ is a regular nilpotent upper triangular matrix and then we deduce that $g^{-1}\ga g\tau^{-1}\in\bfI$ and hence $g\bfI\in Y_\ga^G$. 
Therefore by Theorem \ref{thm:main-Lie-algebra-case} we get that $Y_\ga^G$ is equidimensional of dimension
\[\dim Y_\ga^G=\dim X_\ga^{\fg,\mathrm{reg}}=\frac{1}{2}(d_\fg(\ga)-\mathrm{Art}_\ga).\]
By assumption $\mathrm{val}(\det\ga)=1$ and we get that 
\[d_\fg(\ga)=d_G(\ga)+n-1=d_G(\ga)+\mathrm{def}(\kappa_G(\ga)).\]
Combining the two equalities above we get
\[\dim Y_\ga^G=\frac{1}{2}(d_G(\ga)+\mathrm{def}(\kappa_G(\ga))-\mathrm{Art}_\ga)\]
and this proves Theorem \ref{thm:GLn} in the case where $\ga$ is elliptic. 

\subsubsection{Reduction to the elliptic case}
Now we treat the general case. Recall that $a=\mathrm{val}(\det\ga)$, $\tau=(\Pi_n)^a$ and $0<a<n$. Let $c:=\mathrm{gcd}(a,n)$, $n':=n/c$ and $a':=a/c$. Then $\tau^{n'}=(\Pi_n)^{an'}=(\Pi_n)^{na'}=\varpi^{a'}$ is central. We have an isomorphism of $F$-algebras $F[\ga]\cong\prod E_i$ where each $E_i$ is a finite totally ramified extension of $F$ of degree $n_i$ such that $\sum n_i=n$. Then after $G(F)$-conjugation we may assume that $\ga$ is a block diagonal matrix in $\prod\mathrm{GL}_{n_i}(F)$ whose blocks are $\ga_i\in\mathrm{GL}_{n_i}(F)$. Let $a_i:=\mathrm{val}(\det\ga_i)$. Then we have $\sum a_i=a$. Since $\ga^{n'}\in\bfI\tau^{n'}=\bfI\varpi^{a'}$, we get that $a_in'=a'n_i$ for all $i$. Then since $(a',n')=1$ we must have $a_i=a'$ and $n_i=n'$ for all $i$. 
Let $M\subset G$ be the standard Levi subgroup isomorphic to $(\mathrm{GL}_{n'})^c$ (which is a product of $c$ copies of $\mathrm{GL}_{n'}$). Let $Q\subset G$ be the standard parabolic subgroup (consisting of block upper triangular matrices) with Levi factor $M$ and let $N\subset Q$ be the unipotent radical of $Q$. Then we see that $\ga$ is an elliptic element in $M(F)$.\par
On the other hand, the standard parahoric subgroup $\bfP\subset G(\cO)$ corresponds to a standard parabolic subgroup $P$ with Levi factor $L$. More precisely, the special fiber of $L$ is the reductive quotient of the special fiber of the parahoric group scheme $\cG_\bfP$ associated to $\bfP$. Let $W=S_n$ be the Weyl group of $G$ and let $W_M\subset W$ (resp. $W_L\subset W$) be the Weyl group of $M$ (resp. $L$). We choose a representative $\dot{w}\in N_G(T)(F)\cap G(\cO)$ for each element $w\in W$. We have disjoint union decompositions
\[G(F)=\bigsqcup_{w\in W_M\backslash W}Q(F)\dot{w}\bfI=\bigsqcup_{w\in W_M\backslash W/W_L}Q(F)\dot{w}\bfP\]
where in each (double) coset of $W$ we have arbitrarily chosen a representative in $W$ and the resulting decomposition is independent of such choices. 
Accordingly there are decompositions of the (partial) affine flag varieties 
\[\mathrm{Fl}=LG/\bfI=\bigsqcup_{w\in W_M\backslash W}\mathrm{Fl}^{Q,w},\quad\mathrm{Fl}_\bfP=LG/\bfP=\bigsqcup_{w\in W_M\backslash W/W_L}\mathrm{Fl}_\bfP^{Q,w}\]
where $\mathrm{Fl}^{Q,w}:=LQ\dot{w}\bfI/\bfI$ and $\mathrm{Fl}_\bfP^{Q,w}=LQ\dot{w}\bfP/\bfP$ are the semi-infinite $LQ$-orbits. These induce decompositions of the Witt vector affine Springer fibers $Y_\ga^G$ and $X_{\bfP,\ga}^G$ into finitely many locally closed subschemes
\[Y_\ga^w:=Y_\ga^G\cap\mathrm{Fl}^{Q,w}=\{g\in Q(F)\dot{w}\bfI/\bfI\mid g^{-1}\ga g\in \bfI\tau\}=\{h\in Q(F)/Q(F)\cap{}^w\bfI\mid h^{-1}\ga h\in\dot{w}\bfI\tau\dot{w}^{-1}\}\]
\[X_{\bfP,\ga}^w:=X_{\bfP,\ga}^G\cap\mathrm{Fl}_\bfP^{Q,w}=\{g\in Q(F)\dot{w}\bfP/\bfP\mid g^{-1}\ga g\in\bfP\tau\}=\{h\in Q(F)/Q(F)\cap{}^w\bfP\mid h^{-1}\ga h\in\dot{w}\bfP\tau\dot{w}^{-1}\}\]
where ${}^w\bfI{}:=\dot{w}\bfI\dot{w}^{-1}$ and ${}^w\bfP{}:=\dot{w}\bfP\dot{w}^{-1}$. The structure of $Y_\ga^w$ and $X_{\bfP,\ga}^w$ can be understood by their closed subschemes
\[Y_{\ga}^{M,w}:=\{m\in M(F)/{}^w\bfI_M\mid m^{-1}\ga m\in\dot{w}\bfI\tau\dot{w}^{-1}\cap M(F)\},\]
\[X_{\bfP,\ga}^{M,w}:=\{m\in M(F)/{}^w\bfP_M\mid m^{-1}\ga m\in\dot{w}\bfP\tau\dot{w}^{-1}\cap M(F)\}\]
where the intersections ${}^w\bfI_M:={}^w\bfI\cap M(F)$, ${}^w\bfP_M:={}^w\bfP\cap M(F)$ are parahoric subgroups of $M(F)$.\par
For each double coset $w\in W_M\backslash W/W_L$, the inverse image $(\pi_{\bfP,\ga}^G)^{-1}(X_{\bfP,\ga}^{w})$ is a union of certain pieces $Y_\ga^{w'}$ for some $w'\in W_M\backslash W$ in the double coset of $w$. Suppose $X_{\bfP,\ga}^{w}$ is nonempty. Since $\pi_{\bfP,\ga}^G$ is surjective, there exists a representative in the double coset of $w$, which we still denote by $w$, such that $Y_\ga^{w}$ is nonempty. Let ${}^w\tau:=\dot{w}\tau\dot{w}^{-1}$ and let $\sigma$ be its image in $W=S_n$. \par
We claim that ${}^w\tau\in M(F)$. The non-emptiness of $Y_\ga^{w}$ implies the non-emptiness of $\dot{w}\bfI\tau\dot{w}^{-1}\cap M(F)$, so there exists an element $x\in{}^w\bfI$ such that $x\cdot{}^w\tau\in M(F)$. Choose an element $t=\mathrm{diag}(t_1,\dotsc,t_n)\in T(F)$ such that $M$ is the centralizer of $t$ in $G$. Then we have $t(x\cdot{}^w\tau)t^{-1}=x\cdot{}^w\tau$ and hence $tx\sigma(t)^{-1}=x$ where $\sigma(t)=\mathrm{diag}(t_{\sigma^{-1}(1)},\dotsc,t_{\sigma^{-1}(n)})$. Since $x\in{}^w\bfI$ so that $\dot{w}^{-1}x\dot{w}\in\bfI$, the diagonal entries of $x$ are nonzero and we deduce that $t_i=t_{\sigma^{-1}(i)}$ for all $i=1,\dotsc,n$. This shows that $\sigma\in W_M$ and hene ${}^w\tau\in M(F)$. Moreover, ${}^w\tau$ lies in the normalizer of ${}^w\bfI_M$ and also ${}^w\bfP_M$ in $M(F)$. Consequently
\[X_{\bfP,\ga}^{M,w}=\{m\in M(F)/{}^w\bfP_M\mid m^{-1}\ga m\in{}^w\bfP_M{}^w\tau\}\]
is the Witt vector affine Springer fiber for $M$ associated to $\ga$ and the parahoric subgroup ${}^w\bfP_M$. Similarly we have
\[Y_\ga^{M,w}=\{m\in M(F)/{}^w\bfI_M\mid m^{-1}\ga m\in {}^w\bfI_M{}^w\tau\}.\]
There is a natural morphism $X_{\bfP,\ga}^{w}\to X_{\bfP,\ga}^{M,w}$ whose fiber over a point in $X_{\bfP,\ga}^{M,w}$ represented by $m\in M(F)$ is
\begin{equation}
    \begin{split}
        Z_m&:=\{u\in N(F)/{}^w\bfP\cap N(F)\mid u^{-1}\ga_m u\in\dot{w}\bfP\tau\dot{w}^{-1}\}\\
        &=\{u\in N(F)/{}^w\bfP\cap N(F)\mid u^{-1}\ga_m u\ga_m^{-1}\in{}^w\bfP\cap N(F)\}
    \end{split}
\end{equation}
where $\ga_m:=m^{-1}\ga m\in{}^w\bfI_M{}^w\tau$. Then we have $\mathrm{ad}(\ga_m)(N(F)\cap{}^w\bfI)\subset N(F)\cap{}^w\bfI$ and thus we can apply Proposition \ref{prop:Levi-reduction} to conclude that $Z_m$ is a perfect scheme perfectly of finite type of dimension
\[r_N(\ga):=\mathrm{val}(\det(1-\mathrm{ad}(\ga)\mid\mathrm{Lie}N(F))).\]
Therefore we get the formula
\[\dim X_{\bfP,\ga}^{w}=X_{\bfP,\ga}^{M,w}+r_N(\ga).\]
When $\ga$ is topologically unipotent we check that 
\[d_G(\ga)=d_M(\ga)+2r_N(\ga).\] 
Since $\ga$ is elliptic in $M(F)$, by the case already proved we get that
\[\dim X_{\bfP,\ga}^{M,w}=\frac{1}{2}(\mathrm{def}(\kappa_M({}^w\tau))+d_M(\ga)-\mathrm{Art}_\ga)=\frac{1}{2}(\mathrm{def}(\kappa_G(\tau)+d_M(\ga)-\mathrm{Art}_\ga).\]
Combining the three identities above we see that whenever $X_{\bfP,\ga}^{w}$ is nonempty, its dimension does not depend on $w$ and equals to $\frac{1}{2}(\mathrm{def}(\kappa_G(\tau)+d_G(\ga)-\mathrm{Art}_\ga)$. Hence this is also the dimension of $X_{\bfP,\ga}$ and if $\bfP=\bfI$ we also see that $Y_\ga$ is equidimensional since each piece $Y_{\ga}^{M,w}$ is equidimensional by the elliptic case we already proved. This finishes the proof of Theorem \ref{thm:GLn} (keeping in mind that $\mathrm{Art}(G)=0$ for the split group $G$). 

\subsection{Finishing the proof of main theorem in the group case}\label{sec:finish-proof-group}
Now we can finish the proof of Theorem \ref{thm:main-group-case}. By Proposition \ref{prop:HC-descent} we may assume that $G$ is simply connected and $\ga\in G(F)$ topologically unipotent regular semisimple. By Lemma \ref{lem:affine-flag-decompose} we may assume that $G$ is absolutely simple and tamely ramified. Recall that $p$ is a good prime for $G$. If $\ga$ is not strongly topologically unipotent, then $G_\mathrm{ad}=PGL_n$ and $p$ divides $n$. In this case the result follows from Theorem \ref{thm:GLn}. Now assume that $\ga\in G(F)$ is \emph{strongly} topologically unipotent. By Proposition \ref{prop:pairing} and Corollary \ref{cor:group-reduction} we may assume that there exists a quasi-logarithm map $\Phi:G\to\fg$ that is bijective on topologically nilpotent elements (see Definition \ref{def:qlog-top-nilp}). Then we have $X_{\bfP,\ga}^G=X_{\bfP,\Phi(\ga)}^\fg$ by Lemma \ref{lem:compare-group-Lie}.\par 
On the other hand, the numerical invariants of $\ga$ and $\Phi(\ga)$ occurring in the dimension formulas coincide. Indeed, by Lemma \ref{lem:compare-group-Lie}, there is a canonical isomorphism between the centralizers $G_\ga\cong G_{\Phi(\ga)}$ and hence $\mathrm{Art}_\ga=\mathrm{Art}_{\Phi(\ga)}$. By Corollary \ref{cor:disc-group-Lie-alg} we have $d_G(\ga)=d_\fg(\Phi(\ga))$. Finally since $\ga$ is strongly topologically unipotent we have $\mathrm{def}(\kappa_G(\ga))=0$. Therefore Theorem \ref{thm:main-group-case} follows from Theorem \ref{thm:main-Lie-algebra-case} and Proposition \ref{prop:equi-dim-group}.

\end{document}